\documentclass[11pt]{amsart}
\usepackage{amsfonts,latexsym,amsthm,amssymb,graphicx}
\usepackage[all]{xy}
\usepackage[usenames]{color} 
\usepackage{amscd,epsfig}
\usepackage{hyperref}
\usepackage{pdfsync}

\DeclareFontFamily{OT1}{rsfs}{}
\DeclareFontShape{OT1}{rsfs}{n}{it}{<-> rsfs10}{}
\DeclareMathAlphabet{\mathscr}{OT1}{rsfs}{n}{it}

\CompileMatrices

\newtheorem{theorem}{Theorem}[section]
\newtheorem{lemma}[theorem]{Lemma}
\newtheorem{corol}[theorem]{Corollary}
\newtheorem{prop}[theorem]{Proposition}

\newtheorem{conj}{Conjecture}

{\theoremstyle{definition} \newtheorem{defin}[theorem]{Definition}}
{\theoremstyle{remark} \newtheorem{remark}[theorem]{Remark}
\newtheorem{example}[theorem]{Example}}

\numberwithin{equation}{section}

\newcommand{\bbA}{{\mathbb A}}
\newcommand{\bbC}{{\mathbb C}}
\newcommand{\bbF}{{\mathbb F}}
\newcommand{\bbH}{{\mathbb H}}
\newcommand{\bbN}{{\mathbb N}}
\newcommand{\bbP}{{\mathbb P}}
\newcommand{\bbQ}{{\mathbb Q}}
\newcommand{\bbR}{{\mathbb R}}
\newcommand{\bbZ}{{\mathbb Z}}

\newcommand{\calA}{{\mathcal A}}
\newcommand{\calC}{{\mathcal C}}
\newcommand{\calD}{{\mathcal D}}
\newcommand{\calF}{{\mathcal F}}
\newcommand{\calG}{{\mathcal G}}
\newcommand{\calI}{{\mathcal I}}
\newcommand{\calL}{{\mathcal L}}
\newcommand{\calO}{{\mathcal O}}
\newcommand{\calR}{{\mathcal R}}

\newcommand{\fC}{{\mathfrak C}}
\newcommand{\fT}{{\mathfrak R}}

\newcommand{\ovU}{\overline{U}}
\newcommand{\ovf}{\overline{f}}

\newcommand{\csm}{{c_{\mathrm{SM}}}}
\newcommand{\one}{1\hskip-3.5pt1}
\newcommand{\qede}{\hfill$\lrcorner$}

\newcommand{\Kt}{\mathrm{K}} 

\newcommand{\ChevG}{\mathscr{G}} 
\newcommand{\ChevB}{\mathscr{B}} 
\newcommand{\ChevT}{\mathscr{T}} 
\newcommand{\ChevN}{\mathscr{N}} %
\newcommand{\LangG}{G} 
\newcommand{\LangB}{B} 
\newcommand{\LangT}{T} 

\newcommand{\Groth}{K_0} 
\newcommand{\opT}{\fT} 

\newcommand{\id}{\mathrm{id}}
\newcommand{\codim}{\mathrm{codim}}
\newcommand{\rk}{\mathrm{rk}}
\newcommand{\Supp}{\mathrm{Supp}}
\newcommand{\var}{\mathrm{var}}
\newcommand{\sm}{\mathrm{sm}}
\newcommand{\pt}{\mathrm{pt}}
\newcommand{\pr}{\mathrm{pr}}
\newcommand{\loc}{\mathrm{loc}}
\newcommand{\mot}{\mathrm{mot}}
\newcommand{\val}{\mathrm{val}}

\DeclareMathOperator{\Ind}{Ind}
\DeclareMathOperator{\Frac}{Frac}
\DeclareMathOperator{\stab}{stab}
\DeclareMathOperator{\Pic}{Pic}
\DeclareMathOperator{\Lie}{Lie}
\DeclareMathOperator{\Attr}{Attr}
\DeclareMathOperator{\Hom}{Hom}
\DeclareMathOperator{\MC}{MC}
\DeclareMathOperator{\SMC}{SMC}
\DeclareMathOperator{\gr}{gr}
\DeclareMathOperator{\ch}{ch}
\DeclareMathOperator{\vol}{vol}
\DeclareMathOperator{\opp}{opp}
\DeclareMathOperator{\Fl}{Fl}
\DeclareMathOperator{\SL}{SL}
\DeclareMathOperator{\Rad}{Rad}

\begin{document}
\title[Motivic Chern classes]{Motivic Chern classes of Schubert cells,
Hecke algebras, and applications to Casselman's problem}

\date{\today}
\author{Paolo Aluffi}
\address{
Mathematics Department, 
Florida State University,
Tallahassee FL 32306
}
\email{aluffi@math.fsu.edu}

\author{Leonardo C.~Mihalcea}
\address{
Department of Mathematics, 
Virginia Tech University, 
Blacksburg, VA 24061
}
\email{lmihalce@vt.edu}

\author{J\"org Sch\"urmann}
\address{Mathematisches Institut, Universit\"at M\"unster, Germany}
\email{jschuerm@uni-muenster.de}

\author{Changjian Su}
\address{Department of Mathematics, University of Toronto, 
Toronto, ON, M5S, 2E4, Canada}
\email{changjiansu@gmail.com}
\subjclass[2010]{Primary 14C17, 20C08, 14M15; Secondary 17B10, 14N15, 33D80}
\keywords{motivic Chern class; Schubert cells; stable envelopes; Hecke
algebra; Casselman's problem; principal series representation}
\thanks{P.~Aluffi was supported in part by NSA Award H98230-16-1-0016
and a Simons Collaboration Grant; L.~C.~Mihalcea was supported in part
by NSA Young Investigator Award H98320-16-1-0013 and a Simons
Collaboration Grant; J.~Sch\"{u}rmann was funded by the Deutsche
Forschungsgemeinschaft (DFG, German Research Foundation) under
Germany's Excellence Strategy -- EXC 2044-390685587, Mathematics
M\"{u}nster: Dynamics -- Geometry -- Structure}

\begin{abstract} 
Motivic Chern classes are elements in the $\Kt$-theory of an algebraic
variety~$X$, depending on an extra parameter $y$. They are determined
by functoriality and a normalization property for smooth $X$. In this
paper we calculate the motivic Chern classes of Schubert cells in the
(equivariant) $\Kt$-theory of flag manifolds $G/B$. We show that the
motivic class of a Schubert cell is determined recursively by the
Demazure-Lusztig operators in the Hecke algebra of the Weyl group of
$G$, starting from the class of a point. The resulting classes are
conjectured to satisfy a positivity property. We use the recursions to
give a new proof that they are equivalent to certain $\Kt$-theoretic
stable envelopes recently defined by Okounkov and collaborators, thus
recovering results of Feh{\'e}r, Rim{\'a}nyi and Weber. The Hecke
algebra action on the $\Kt$-theory of the Langlands dual flag manifold
matches the Hecke action on the Iwahori invariants of the principal
series representation associated to an unramified character for a
group over a nonarchimedean local field. This gives a correspondence
identifying the
duals of the motivic Chern classes to the
standard basis in the Iwahori invariants, and the fixed point basis to
Casselman's basis. We apply this correspondence to prove two
conjectures of Bump, Nakasuji and Naruse concerning factorizations and
holomorphy properties of the coefficients in the transition matrix
between the standard and the Casselman's basis.
\bigskip

\noindent{\sc R\'esum\'e.}
Les {\em classes de Chern motiviques\/} sont des \'el\'ements de la
$\Kt$-th\'eorie d'une vari\'et\'e {alg\'ebri\-que}~$X$, qui d\'ependent
d'un param\`etre suppl\'ementaire $y$. Elles sont d\'etermin\'ees par la
fonctorialit\'e et une propri\'et\'e de normalisation pour $X$ lisse.
Dans cet article, nous calculons les classes de Chern motiviques des
cellules de Schubert dans la $\Kt$-th\'eorie (\'equivariante) des
vari\'et\'es de drapeaux $G/B$. Nous montrons que la classe motivique
d'une cellule de Schubert est d\'et\'ermin\'ee r\'ecursivement gr\^ace
aux op\'erateurs de Demazure-Lusztig de l'alg\`ebre de
Hecke du groupe de Weyl de $G$, \`a partir de la classe d'un point. 
Nous conjecturons que les classes obtenues satisfont une propri\'et\'e
de positivit\'e.  Nous utilisons nos r\'ecurrences pour obtenir une
nouvelle preuve du fait que les classes sont \'equivalentes \`a
certaines enveloppes stables d\'efinies r\'ecemment en $\Kt$-th\'eorie
par Okounkov et ses collaborateurs, retrouvant ainsi un r\'esultat de 
Feh\'er, Rim\'anyi, et Weber. L'action de l'alg\`ebre de Hecke sur la
$\Kt$-th\'eorie de la vari\'et\'e de drapeaux du dual de Langlands
co\"incide avec l'action de Hecke sur les invariants d'Iwahori de la
repr\'esentation par s\'erie principale associ\'ee \`a un caract\`ere non
ramifi\'e pour un groupe sur un corps local non archim\'edien.  Cela
induit une correspondance identifiant les duaux des
classes de Chern motiviques \`a la base standard du module des
invariants d'Iwahori, et la base des points fixes \`a la base de
Casselman. Nous appliquons ce r\'esultat pour d\'emontrer deux
conjectures dues \`a Bump, Nakasuji et Naruse concernant les factorisations
et les propri\'et\'es d'holomorphie des coefficients de la matrice de
transition entre la base standard et la base de Casselman.
\end{abstract}

\maketitle

\section{Introduction} 
Let $X$ be a complex algebraic variety, and let $\Groth(\var/X)$ be the
(relative) Grothendieck group of varieties over $X$. It consists of
classes of morphisms $[f: Z \to X]$ modulo the scissors relations;
cf.~\cite{looijenga:motivic,bittner:universal} and \S\ref{s:MC}
below.~Brasselet, Sch{\"u}rmann and
Yokura~\cite{brasselet.schurmann.yokura:hirzebruch} defined the {\em
  motivic Chern transformation} $\MC_y: \Groth(\var/X) \to K(X)[y]$ with
values in the $\Kt$-theory group of coherent sheaves in $X$ to which
one adjoins a formal variable $y$. The transformation $\MC_y$ is a
group homomorphism, it is functorial with respect to proper
push-forwards, and if $X$ is smooth, it satisfies the normalization
condition
\[ 
\MC_y[\id_X: X \to X] = \sum [\wedge^j T^*(X)] y^j \/. 
\] 
Here $[\wedge^j T^*(X)]$ is the $\Kt$-theory class of the bundle of
degree $j$ differential forms on $X$. If $Z \subseteq X$ is a
constructible subset, we denote by $\MC_y(Z):= \MC_y[Z \hookrightarrow
X] \in K(X)[y]$ the motivic Chern class of $Z$. Because $\MC_y$ is a
group homomorphism, it follows that if $X = \bigsqcup Z_i$ is a
disjoint union of constructible subsets, then $\MC_y(X) = \sum
\MC_y(Z_i)$.  As explained
in~\cite{brasselet.schurmann.yokura:hirzebruch}, the motivic Chern
class $\MC_y(Z)$ is related by a Hirzebruch-Riemann-Roch type statement
to the Chern-Schwartz-MacPherson (CSM) class $\csm(Z)$ in the homology 
of~$X$. We recall that the existence and functoriality properties of this 
CSM class were conjectured by Deligne and Grothendieck and proved by 
Robert MacPherson~(\cite{macpherson:chern}).
Earlier, Marie-H\'el\`ene Schwartz had independently established a theory
of Chern classes for singular varieties, using obstruction theory
(\cite{schwartz:1, schwartz:2}). Jean-Paul Brasselet and Schwartz proved
that the two classes coincide via the Alexander isomorphism (\cite{BS81}).
Both the
motivic and the CSM classes give a functorial way to attach
$\Kt$-theory, respectively (co)homology classes, to constructible
subsets, and both satisfy the usual motivic relations. There is also an
equivariant version of the motivic Chern class transformation, which
uses equivariant varieties and morphisms, and has values in the
appropriate equivariant $\Kt$-theory group. Its definition was given
in~\cite{feher2018motivic}, following closely the approach
of~\cite{brasselet.schurmann.yokura:hirzebruch}.  

We take this opportunity to provide further details on the
construction and properties of equivariant motivic Chern classes, such
as functoriality and a Verdier-Riemann-Roch formula; see \S\ref{s:MC}
below. However, the main goals of this paper are to build the
computational foundations for the study of the (equivariant) motivic
Chern classes of Schubert cells in the generalized flag manifolds, and
to relate this to the representation theory of $p$-adic groups. Our
main application consists of formulas for the transition coefficients
between the standard and the Casselman bases of the module of Iwahori
invariants of the principal series representation, in terms of
localizations of motivic Chern classes.

Let $G$ be a complex, semisimple, linear algebraic group, and $B$ a
Borel subgroup.  By functoriality, the (equivariant) motivic Chern
classes of Schubert cells in $G/B$ determine those in any flag
manifold $G/P$, where $P$ is a parabolic subgroup. Based on previously
discovered features of the CSM classes of Schubert
cells~\cite{aluffi.mihalcea:eqcsm,rimanyi.varchenko:csm,AMSS:shadows},
it was expected that the motivic classes would be closely related to
objects which appear in geometric representation theory. We prove in
this paper that the motivic Chern classes of Schubert cells are
recursively determined by the Demazure-Lusztig operators which appear
in early works of Lusztig on Hecke algebras~\cite{lusztig:eqK}.
Further, the motivic classes of Schubert cells are equivalent (in a
precise sense) to the $\Kt$-theoretic stable envelopes defined by
Okounkov and collaborators in~\cite{okounkov2015lectures,AO:elliptic,
okounkov2016quantum}.~This equivalence was proved recently by
Feh{\'e}r, Rim{\'a}nyi and Weber~\cite{feher2018motivic}; our
approach, based on comparing the Demazure-Lusztig recursions to the
recursions for the stable envelopes found by Su, Zhao and Zhong
in~\cite{su2017k} gives another proof of this result. Via this
equivalence, the motivic Chern classes can be considered as natural
analogues of the Schubert classes in the $\Kt$-theory of the cotangent
bundle of $G/B$.
 
As in the authors' previous work on CSM classes~\cite{AMSS:shadows}, 
the connections to Hecke algebras and $\Kt$-theoretic stable envelopes
yield remarkable identities among (duals of)
motivic Chern classes. We use these identities to prove two
conjectures of Bump, Nakasuji and
Naruse~\cite{BN11,bump2017casselman,nakasuji2015yang} about the
coefficients in the transition matrix between the Casselman's basis
and the standard basis in the Iwahori-invariant space of the principal
series representation for an unramified character for a group over a
non archimedean local field.

We present next a more extensive description of our results.

\subsection{Statement of results} 
Let $G$ be a complex, semisimple, linear algebraic group, and fix $B,
B^-$ a pair of opposite Borel subgroups of $G$. Denote by $T:= B \cap
B^-$ the maximal torus, by $W:= N_G(T)/T$ the Weyl group, and by $X:=
G/B$ the (generalized) flag variety. For each Weyl group element $w
\in W$ consider the Schubert cell $X(w)^\circ:= B w B/B$, a subvariety
of (complex) dimension $\ell(w)$. The opposite Schubert cell
$Y(w)^\circ:= B^- w B/B$ has complex codimension $\ell(w)$. The
closures $X(w)$ and $Y(w)$ of these cells are the Schubert varieties.
Let $\calO_w$, respectively $\calO^w$ be the $\Kt$-theoretic Schubert
classes associated to the structure sheaves of $X(w)$, respectively
$Y(w)$. The equivariant $\Kt$-theory ring of $X$, denoted by $K_T(X)$,
is an algebra over $K_T(\pt)=R(T)$---the representation ring of
$T$---and it has an $R(T)$-basis given by the Schubert classes
$\calO_w$ (or $\calO^w$), where $w$ varies in the Weyl group $W$.

If $E$ is an equivariant vector bundle over $X$, we denote by $[E]$
its class in $K_T(X)$, and by $\lambda_y(E)$ the class
\[ 
\lambda_y(E) = \sum [\wedge^i E] y^i \in K_T(X) [y] \/. 
\] 
For a $T$-stable subvariety $\Omega \subseteq X$ recall the notation
\[ 
\MC_y(\Omega):= \MC_y[ \Omega \hookrightarrow X] \in K_T(X)[y] \/. 
\] 
Our first main result is a recursive formula to calculate
$\MC_y(X(w)^\circ)$, the (equivariant) motivic Chern class of the
Schubert cell. For each simple positive root $\alpha_i$, consider the
Demazure operator $\partial_i: K_T(X) \to K_T(X)$
(\cite{demazure:desingularisations},
\cite{kostant.kumar:KT});
this is a
$K_T(\pt)$-linear endomorphism.  Extend $\partial_i$ linearly with
respect to $y$, and define the Demazure-Lusztig (DL) operators
$ \opT_i, \opT^\vee_i: K_T(X)[y] \to K_T(X)[y]$ by
\[ 
\opT_i: = \lambda_y(\calL_{\alpha_i}) \partial_i 
- \id; \quad \opT^\vee_i: = \partial_i 
\lambda_y(\calL_{\alpha_i}) - \id \/,
\] 
where $\calL_{\alpha_i} = G \times^B \bbC_{\alpha_i}$ is the 
equivariant line bundle whose fiber over $1.B$ has weight 
$\alpha_i$.  The operator
$\opT^\vee_i$ appeared classically in Lusztig's study of Hecke
algebras~\cite{lusztig:eqK}, and $\opT_i$ appeared recently in
related
works~\cite{lee.lenart.liu:whittaker,brubaker.bump.licata}.~The two
operators are adjoint to each other via the $\Kt$-theoretic
intersection pairing; see \S\ref{ss:DL} below. Our first main
result is the following (cf.~Theorem~\ref{thm:DLrecursion}).

\begin{theorem}\label{thm:introT1} 
Let $w \in W$ and let $s_i$ be a simple reflection such that
$\ell(ws_i) > \ell(w)$. Then
\[ 
\MC_y(X(ws_i)^\circ) = \opT_i( \MC_y(X(w)^\circ)) \/. 
\] 
\end{theorem} 

Using the (equivariant) $\Kt$-theoretic Chevalley
formula~\cite{fulton.lascoux,pittie.ram,lenart.postnikov:affine} to
multiply by classes of line bundles, the DL operators give a recursive
formula to calculate the motivic Chern classes, starting from the
class of a point.

Theorem~\ref{thm:introT1} generalizes the analogous result
from~\cite{aluffi.mihalcea:eqcsm} where it was proved that the CSM
classes of Schubert cells are recursively determined by operators in
the degenerate Hecke algebra.  The proof of Theorem~\ref{thm:introT1}
relies on the calculations of push-forwards of classes from the
Bott-Samelson desingularizations of Schubert varieties.

To illustrate the result, we list below the non-equivariant motivic
Chern classes of Schubert cells in $\Fl(3):= \SL_3(\bbC)/B$,
the manifold parametrizing flags in $\bbC^3$.~In this case, the Weyl
group is the symmetric group $S_3$, generated by simple reflections
$s_1$ and $s_2$, and $w_0 = s_1 s_2 s_1$ is the longest element.
{\small
\[ 
\begin{split} 
\MC_y(X(\id)) &= \calO_{\id}; \\ 
\MC_y(X(s_1)^\circ) &= (1+y) \calO_{s_1} - (1+2y) \calO_{\id}\/;\\ 
\MC_y(X(s_2)^\circ) &= (1+y) \calO_{s_2} - (1+2y) \calO_{\id} \/;\\ 
\MC_y(X(s_1 s_2)^\circ) &= (1+y)^2 \calO_{s_1 s_2} - (1+y) (1+2y)
\calO_{s_1} - (1+y)(1+3y) \calO_{s_2} + (5y^2+ 5y+1) \calO_{\id} \/;\\ 
\MC_y(X(s_2 s_1)^\circ) &= (1+y)^2 \calO_{s_2 s_1} - (1+y) (1+2y)
\calO_{s_2} - (1+y)(1+3y) \calO_{s_1} + (5y^2 + 5y+1) \calO_{\id} \/;\\ 
\MC_y(X(w_0)^\circ) &= (1+y)^3 \calO_{w_0} - (1+y)^2 (1+2y)
(\calO_{s_1s_2} + \calO_{s_2 s_1}) + \\ 
& (1+y)(5 y^2 + 4 y +1)(\calO_{s_1 } + \calO_{s_2}) 
- (8y^3+11y^2+ 5y+1) \calO_{\id}\/. 
\end{split}\]}One observes in these examples, and one can also prove it in general,
that the specialization $y \mapsto 0$ in $\MC_y(X(w)^\circ)$ yields the
(push-forward to $G/B$ of the)
ideal sheaf of the boundary of the Schubert variety~$X(w)$. The
Schubert class $\calO_w$ is obtained if one takes $y \mapsto 0$ in a
recursion given by a renormalization of the inverse of the dual operator
$\opT^\vee_i$; see Example~\ref{ex:motcells}. In fact, Theorem~\ref{thm:introT1} generalizes the
well-known fact from Schubert Calculus that the Schubert classes
$\calO_w$ are obtained recursively by the Demazure operators
$\partial_i$. These and other combinatorial properties of the motivic
Chern classes will be studied in a continuation to this paper.

A remarkable feature in the examples listed above is a positivity
property. Based on substantial computer evidence we make the following
positivity conjecture:

\begin{conj}[Positivity Conjecture]\label{conj:csmpos} 
Consider the Schubert expansion 
\[  
\MC_y(X(w)^\circ) =  \sum c(w;u) \calO_u \in K_T(X)[y] \/. 
\]
Then the coefficients $c(w;u) \in K_T(\pt)[y]$ satisfy
$(-1)^{\ell(w)-\ell(u)}c(w,u)\in \bbZ_{\geq 0}[y][e^{-\alpha_1}, \ldots
  , e^{-\alpha_r}]$, where $\alpha_i$ are the positive simple
roots. In particular, in the non-equivariant case,
\[ 
(-1)^{\ell(w)-\ell(u)}c(w;u) \in \bbZ_{\geq 0}[y]  \/.
\]
\end{conj}

In type A, a similar positivity property was also conjectured
in~\cite[\S6]{FRW:char}, along with a log concavity property. In the
non-equivariant case, 
it is conjectured in~\cite{FRW:char} 
that the polynomials
\[ 
(-1)^{\ell(w) - \ell(u)} \frac{ c(w;u)}{(1+y)^{\ell(u)}} 
\] 
are log-concave.~In cohomology, Aluffi and Mihalcea conjectured that CSM
classes of Schubert cells are 
positive~\cite{aluffi.mihalcea:csm,aluffi.mihalcea:eqcsm}.~For 
Grassmannians, this was established
by J. Huh \cite{huh:positivity}; a few special cases were settled
earlier in~\cite{aluffi.mihalcea:csm, mihalcea:binomial,jones:csm,
  stryker:thesis}. In full $G/P$ generality, and in the
non-equivariant case, the conjecture was recently proved
in~\cite{AMSS:shadows}, using the theory of characteristic cycles in
the cotangent bundle of $G/B$. There is also a stronger version of
this conjecture, which claims that in addition $c(w;u)\neq 0$ whenever
$u \le w$. Huh's proof shows this, and also establishes (implicitly)
the equivariant version for Grassmannians. The statement of the
conjecture is reminiscent of the positivity in (equivariant)
$\Kt$-theory proved by Buch~\cite{buch:Kpos}, Brion~\cite{brion:Kpos},
and Anderson, Griffeth and
Miller~\cite{anderson.griffeth.miller:positivity}.

Let $\MC_y^\vee(Y(v)^\circ)$ be the classes obtained by applying the
(inverse) dual operators $(\opT^\vee_i)^{-1}$ to $\calO^{w_0}$
(instead of $\opT_i$ to $\calO_{\id}$); see
Definition~\ref{def:firstdual}. We prove (Theorem~\ref{thm:dualbasis})
that for every $u, v\in W$,
\begin{equation}\label{E:duality} 
\langle \MC_y(X(u)^\circ), \MC_y^\vee(Y(v)^\circ)\rangle 
= \delta_{u,v}(-y)^{\ell(u)-\dim G/B}\prod_{\alpha>0}
(1+ye^{-\alpha}) \/,
\end{equation} 
where $\langle \cdot , \cdot \rangle$ is the $\Kt$-theoretic
intersection pairing. This orthogonality property, which we call
`Hecke duality', mirrors the similar orthogonality of
Chern-Schwartz-MacPherson (CSM) classes proved by the authors
in~\cite{AMSS:shadows}.

\subsection{Applications to Casselman's problem} 
The main application in this paper is to use the Hecke algebra action
on motivic Chern classes of Schubert cells to prove two conjectures of
Bump, Nakasuji and
Naruse~\cite{BN11,bump2017casselman,nakasuji2015yang} about properties
of certain coefficients of the transition matrix between two natural
bases of the Iwahori invariant part of the principal series
representation. We briefly recall below the relevant history,
definitions and the results; the details and complete proofs are given
in \S\ref{sec:padic} below.

Let $\tau$ be an unramified character for a split reductive Chevalley
group $\ChevG(F)$ over a nonarchimedean local field $F$ with finite
residue field $\bbF_{q'}$.  The principal series representation is the
induced representation $I(\tau):=\Ind_{\ChevB(F)}^{\ChevG(F)}
(\tau)$.~We consider its submodule $I(\tau)^I$ of Iwahori invariants
of $I(\tau)$; this is a Hecke module, with an additive basis indexed
by the Weyl group $W$. There are two important bases: the {\em
  standard basis,\/} given by the characteristic
functions~$\varphi_w$, and the {\em Casselman basis\/} $\{ f_w \}$,
defined using certain intertwiners; see~\eqref{E:deffw}. {\em
  Casselman's problem\/}~\cite{casselman1980unramified} is to find the
transition matrix between the two bases:
\begin{equation}\label{E:auvexp}  
\varphi_u = \sum a_{u,w} f_w \/. 
\end{equation}
As observed by Bump and Nakasuji~\cite{BN11}, it is better to consider
the basis $\psi_u := \sum_{w \ge u} \varphi_w$ and the expansion
\[ 
\psi_u = \sum \tilde{m}_{u,w} f_w \/. 
\] 
By M{\"o}bius inversion, the problems of finding either of the
transition matrices are equivalent. Recent solutions to the
Casselman's problem were obtained by Naruse and
Nakasuji~\cite{nakasuji2015yang}, using the Yang-Baxter basis in the
Hecke algebra introduced by Lascoux, Leclerc and
Thibon~\cite{LLT:flagYB}, and by Su, Zhao and Zhong~\cite{su2017k},
by means of the theory of stable envelopes developed
in~\cite{okounkov2015lectures, okounkov2016quantum,AO:elliptic}. The
$\Kt$-theoretic stable envelopes are certain classes in the
equivariant $\Kt$-theory of the cotangent bundle $T^*(G/B)$, indexed
by the Weyl group elements; see \S\ref{s:MCstab} below. Su, Zhao and
Zhong proved that the Hecke algebra action on the basis of stable
envelopes coincides with the Hecke algebra action on the standard
basis $\varphi_w$. Under their correspondence, the Hecke action on the
Casselman's basis fits with the Hecke action on the fixed point basis
in equivariant $\Kt$-theory.

Feh{\'e}r, Rim{\'a}nyi and Weber~\cite{feher2018motivic,FRW:char} 
observed that motivic Chern classes and
$\Kt$-theoretic stable envelopes are closely related; see also
\S\ref{s:MCstab} below. Therefore it is not a surprise that one can
recover the Hecke correspondence from~\cite{su2017k} using motivic
Chern classes. The advantage of this point of view is that motivic
Chern classes satisfy strong functoriality properties, and this will
allow us to obtain additional properties of the coefficients
$\tilde{m}_{u,v}$.

Let $\LangG$ be the Langlands dual of $\ChevG$.
It turns out that the Hecke module $I(\tau)^I$ is more naturally
related to the equivariant $\Kt$-theory $K_{\LangT}(\LangG/\LangB)$
for the Langlands dual flag manifold.
Let $\iota_w$ be the fixed point basis in $K_{\LangT}(\LangG/\LangB)$,
and let $b_w$ be the multiple of $\iota_w$ determined by the
localization 
\[
b_w|_w =\MC_{y}^\vee(Y(w)^\circ)|_w
\]
at the fixed point $e_w$. The formula for $\MC_{y}^\vee(Y(w)^\circ)|_w
\in K_{\LangT}(\pt)[y^{-1}]$ is explicit; see
Proposition~\ref{prop:chardualmot}. From this formula it follows that
the elements $b_w$ are in the localized ring
\[
K_{\LangT}(\LangG/\LangB)_{\loc}[y^{-1}]:= 
K_{\LangT}(\LangG/\LangB)[y^{-1}] \otimes_{K_{\LangT}(\pt)} 
\Frac(K_{\LangT}(\pt))\:,
\] 
where $\Frac$ denotes the fraction field.  We show 
(Theorem~\ref{thm:padic})
that for a
sufficiently general~$\tau$ there is an isomorphism of Hecke modules
$\Psi: K_{\LangT}(\LangG/\LangB)_{\loc}[y,y^{-1}]
\otimes_{K_{\LangT}(\pt)} \bbC_\tau \to I(\tau)^I$ such that
\[ 
\Psi (\MC_y^\vee(Y(w)^\circ) \otimes 1) =  \varphi_w; \quad 
\Psi(b_w \otimes 1)=  f_w; \quad \Psi(y) = -q' \/, 
\]
with $q'=|\bbF_{q'}|$ the number of elements in the finite residue
field $\bbF_{q'}$, and $\bbC_\tau$ the one-dimensional
$K_{\LangT}(\pt)$-module obtained by evaluation at $\tau$. Using this
result, we prove that $\tilde{m}_{u,w} = m_{u,w}(\tau)$, where
$m_{u,w}$ are the coefficients in the expansion
\[ 
\MC_y^\vee(Y(u)) := \sum_{w\geq u} \MC_y^\vee(Y(w)^\circ) 
= \sum m_{u,w} b_w \quad \in  
K_{\LangT}(\LangG/\LangB)_{\loc}[y^{-1}] \/.
\] 
Implicit in this is that the coefficients $m_{u,w}$ may be regarded as
complex valued functions defined on a certain Zariski open subset of
the dual torus $\LangT$.

The Hecke isomorphism $\Psi$ provides a `dictionary', translating all
statements about $\tilde{m}_{u,w}$ into statements about $m_{u,w}$,
which have geometric meaning. The key result for the
representation theoretic applications is that the coefficients
$m_{u,w}$ are given by localization
(cf.~Proposition~\ref{prop:newformula} below):

\begin{theorem}\label{thm:newformula} 
(a) For every $w\geq u\in W$, the coefficient  $m_{u,w}$ equals 
\[ 
m_{u,w} = \Bigl(\frac{\MC_{y}(Y(u))|_w}{\MC_{y}(Y(w)^\circ)|_w} 
\Bigr)^\vee \/, 
\] 
where $\vee$ is the operator mapping $e^\lambda \mapsto e^{-\lambda}$
for $e^\lambda \in K_T(\pt)$ and $y \mapsto y^{-1}$.

(b) Assume that $Y(u)$ is smooth at the fixed point $e_w$ and denote
by $(N_{Y(w)} Y(u))_w$ the normal space at $e_w$ in $Y(u)$, regarded
as a trivial (but not equivariantly trivial) vector bundle. Then
\[
m_{u,w} = \frac{ \lambda_{y^{-1}} ((N_{Y(w)} Y(u))_w)}
{\lambda_{-1} ({(N_{Y(w)} Y(u))_w})} \/. 
\]
\end{theorem}

In particular, the entries 
$m_{1,w}$
are obtained from the motivic
Chern class of the full flag variety $\MC_y(Y(\id))=\MC_y(G/B) =
\lambda_y(T^*(G/B))$, and one recovers the (geometric version of the)
classical Gindikin-Karpelevich formula, proved by
Langlands~\cite{langlands}:
\[ 
m_{1,w} = \prod \frac{1 +y^{-1} e^{\alpha}}{1 - e^{\alpha}} \/, 
\] 
where the product of over positive roots $\alpha$ such that
$w^{-1}(\alpha) < 0$. Let
\[ 
S(u,w):=\{\alpha \in R^{+}\,|\,u\leq s_\alpha w < w\} \/. 
\]
Our main application is the following factorization formula for
$m_{u,v}$, see Theorem~\ref{thm:geomrefinedconj}
below:

\begin{theorem}[Geometric Bump-Nakasuji-Naruse Conjecture] 
\label{thm:introgeomrefinedconj} 
For every $u\leq w\in W$,
\[
m_{u,w}=\prod_{\alpha\in S(u,w)}\frac{1+y^{-1}e^\alpha}{1-e^\alpha},
\]
if and only if the Schubert variety $Y(u)$ is smooth at the torus
fixed point $e_{w}$.
\end{theorem} 

This is the geometric analogue of a conjecture of Bump and
Nakasuji~\cite{BN11,bump2017casselman} for simply laced types,
generalized to all types by Naruse~\cite{naruse2014schubert}, and
further analyzed by Nakasuji and Naruse~\cite{nakasuji2015yang}.
While this paper was in preparation, Naruse informed us that he also
obtained an (unpublished) proof of the implication assuming
factorization. Both Naruse's and our proofs are based on Kumar's
cohomological criterion for smoothness of Schubert
varieties~\cite{kumar1996nil}; Naruse used Hecke algebra calculations;
ours relies on properties of motivic Chern classes. The original
conjecture of Bump and Nakasuji from~\cite{BN11} was stated in terms
of conditions under which certain Kazhdan-Lusztig polynomials $P_{w_0
  w^{-1}, w_0 u^{-1}}$ {equal} $1$; we explain the equivalence to the
statement above (in simply laced types) and discuss further this
conjecture in sections~\S\ref{sec:geomBNNconj} and
\S\ref{sec:BNNconj}.

A second conjecture refers to a holomorphy property. In relation to
Kazhdan-Lusztig theory, Bump and Nakasuji~\cite{bump2017casselman}
defined the coefficients
\[ 
\tilde{r}_{u,w} := \sum_{u \le x \le w} (-1)^{\ell(x) 
- \ell(u)} \overline{\tilde{m}}_{x,w} \/, 
\] 
where the bar operator replaces $q'$ by $q'^{-1}$. Using
M{\"o}bius inversion it follows that the coefficients $a_{u,w}$ 
from~\eqref{E:auvexp} satisfy $\bar{a}_{u,w}=
\tilde{r}_{u,w}$. Geometrically, 
these correspond to the coefficients
$r_{u,w}$ obtained from the expansion
\[ 
\MC_y^\vee(Y(u)^\circ) = \sum \overline{r}_{u,w} b_w \quad 
\in  K_T(G/B)_{\loc}[y^{-1}] \/, 
\] 
where $\bar{f}(y) := f(y^{-1})$. We prove the following result 
(cf.~Theorem~\ref{thm:hol}), which answers affirmatively
Conjecture~$1$ from~\cite{bump2017casselman}.
\begin{theorem}\label{thm:divintro} 
Let $u \le w$ be two Weyl group elements. Then the functions
\[ 
\prod_{\alpha \in S(u,w)}(1-e^{\alpha})r_{u,w} \quad , 
\prod_{\alpha \in S(u,w)}(1-e^{\alpha})m_{u,w} 
\]
are holomorphic on the torus $T$.
\end{theorem}

Both Theorems~\ref{thm:introgeomrefinedconj} and~\ref{thm:divintro}
are consequences of Theorem~\ref{thm:newformula}. The proof of the
latter requires a second orthogonality property between motivic Chern
classes and their duals, proven by means of the connection with the
the theory of $\Kt$-theoretic
stable envelopes. From this orthogonality we deduce the following key
formula, proved in Theorem~\ref{thm:dual1}:
\[ 
\MC_y^\vee(Y(u)^\circ) = \prod_{\alpha>0}(1+ye^{-\alpha}) 
\frac{\calD(\MC_y(Y(u)^\circ))}{\lambda_y(T^*(G/B))}  
\] 
as elements in the appropriate localized $\Kt$-theory ring,
where $\calD[E]= (-1)^{\dim X} [E^\vee] \otimes [\wedge^{\dim X}
  T^*(X)]$ is the (equivariant) Grothendieck-Serre duality
operator, with $X=G/B$.\begin{footnote}{The class
    $\frac{\calD(\MC_y(\Omega))}{\lambda_y(T^*(G/B))}$ may be regarded
    as the motivic analogue of the Segre-MacPherson class
    $\frac{c_*(\one_\Omega)}{c(T(G/B))}$,
cf.~Definition~\ref{def:dualmotvariety}.}\end{footnote}
The proof requires a precise relationship between the motivic Chern classes and
stable envelopes. If $\iota: X \to T^*_X$ is the zero section, then
our statement is that (roughly)
\begin{equation}\label{E:MCvsstab} 
\calD(\iota^*(\stab_+(w))) = N(q) \MC_{-q^{-1}}(X(w)^\circ) \/, 
\end{equation} 
where $\stab_+(w)$ is a stable envelope, $N(q)$ is a normalization
parameter, and $q=-y^{-1}$ is determined from the dilation action of
$\bbC^*$ on the fibers of the cotangent bundle. The precise statement is
given in Theorem~\ref{thm:poshriek}.

Formula \eqref{E:MCvsstab} is part of a more general paradigm,
stemming from the classical works of Sabbah~\cite{sabbah:quelques} and
Ginzburg~\cite{ginzburg:characteristic}, relating intersection theory
on the cotangent bundle to that of characteristic classes of singular
varieties. For instance, the (co)homological analogues of the motivic
Chern classes of Schubert cells---the CSM classes---are equivalent to
Maulik and Okounkov's {\em cohomological stable
  envelopes}~\cite{maulik.okounkov:quantum}. This statement, observed
by Rim{\'a}nyi and Varchenko~\cite{rimanyi.varchenko:csm}, and by the
authors in~\cite{AMSS:shadows}, is a consequence of the fact that both
the stable envelopes and the CSM classes are determined by certain
interpolation conditions obtained from equivariant localization;
cf.~Weber's article~\cite{Weber}. The relation to stable envelopes was
recently extended to $\Kt$-theory by Feh{\'e}r, Rim{\'a}nyi and
Weber~\cite{feher2018motivic} (see also~\cite{FRW:char}). They showed
that the motivic Chern classes of the Schubert cells satisfy the same
localization conditions as the $\Kt$-theoretic stable envelopes
appearing in papers by Okounkov and
Smirnov~\cite{okounkov2015lectures, okounkov2016quantum} for a
particular choice of parameters. (The result
from~\cite{feher2018motivic} is more general, involving the motivic
Chern classes for orbits in a space with finitely many orbits under a
group action.) We reprove this result by comparing the
Demazure-Lusztig type recursions for motivic Chern classes to the
recursions for the stable envelopes found by Su, Zhao and Zhong
in~\cite{su2017k}. We also discuss the relation between the motivic
Chern class and various choices of parameters for the $\Kt$-theoretic
stable envelopes, which might be of independent interest;
see~\S\ref{s:MCstab} below.\smallskip

{\em Acknowledgments.} 
Part of this work was performed at MSRI during
the special semester in `Enumerative geometry beyond numbers'; LM and
CS would like to thank the organizers of the workshops of the special
MSRI semester for providing support to participate to this program,
and to MSRI for providing a stimulating environment.~LM would like to
thank Anders Buch, Daniel Orr, Richard Rim{\'a}nyi, Mark Shimozono and
Andrzej Weber for stimulating discussions.~CS thanks IHES for
providing excellent research environment and Joel Kamnitzer, Shrawan
Kumar, Andrei Okounkov and Eric Vasserot for useful
discussions. 
Special thanks are due to Hiroshi Naruse for informing us
of the refined Bump-Nakasuji-Naruse conjecture from
Theorem~\ref{thm:refinedconj}. This happened while CS and LM were
attending the `International Festival in Schubert Calculus',
Guangzhou, China, November 2017; we wish to thank the organizers of
this conference for providing support and an environment conducive to
research. PA also thanks the University of Toronto for the hospitality.
 Finally, all authors wish to thank the anonymous referees, for the careful 
reading and for many suggestions which improved the exposition and helped streamline 
some of the arguments.

\section{Schubert varieties and their Bott-Samelson resolutions}\label{s:BS}
In this section we recall the basic definitions and facts about the
Bott-Samelson resolution of Schubert varieties. These will be used in
the next section to perform the calculation of the motivic Chern class
of a Schubert cell. Our main references
are~\cite{aluffi.mihalcea:eqcsm} and~\cite{brion.kumar:frobenius}.

Let $G$ be a complex semisimple linear algebraic group. Fix a Borel 
subgroup $B$
and let $T := B \cap B^-$ be the maximal torus, where $B^-$ denotes
the opposite Borel subgroup. Let $W:=N_G(T)/T$ be the Weyl group, and
$\ell:W \to \bbN$ the associated length function. Denote by
$w_0$ the longest element in $W$; then $B^- = w_0 B w_0$. Let also
$\Delta := \{ \alpha_1, \ldots , \alpha_r \} \subseteq R^+$ denote the
set of simple roots included in the set of positive roots for
$(G,B)$. The simple reflection for the root $\alpha_i \in \Delta$ is
denoted by $s_i$, and $P_i$ denotes the corresponding {\em minimal\/}
parabolic subgroup.

Let $X:=G/B$ be the flag variety. It has stratifications by Schubert
cells $X(w)^\circ:= BwB/B$ and by opposite Schubert cells
$Y(w)^\circ:=B^- w B/B$.  The closures $X(w):= \overline{X(w)^\circ}$
and $Y(w):=\overline{Y(w)^\circ}$ are the Schubert varieties.~Note
that $Y(w)=w_0 X(w_0w)$.~With these definitions, $\dim_{\bbC} X(w) =
\codim_{\bbC} Y(w) = \ell(w)$. The Weyl group $W$ admits a partial
ordering, called the Bruhat ordering, defined by $u \le v$ if and only
if $X(u) \subseteq X(v)$.

We recall next the definition of the Bott-Samelson resolution of a
Schubert variety, following~\cite[\S2.3]{aluffi.mihalcea:eqcsm}
and~\cite[\S2.2]{brion.kumar:frobenius}.  Fix $w \in W$ and a 
decomposition of $w$, i.e., a sequence $(i_1, \ldots , i_k)$ such that
$w=s_{i_1} \cdot \ldots \cdot s_{i_k}$.
If $\ell(w) = k$, then this decomposition is called {\em reduced.\/}
This data determines a tower $Z$ of $\bbP^1$-bundles and a
birational map $\theta: Z \to X(w)$ as follows.

If the word is empty, then define $Z:= \pt=B/B\hookrightarrow G/B$. 
In general assume we have
constructed $Z':= Z_{i_1, \dots, i_{k-1}}$ and the map $\theta': Z'
\to X(w')\to G/B$, for $w'= s_{i_1} \cdots s_{i_{k-1}}$.

Define $Z=Z_{i_1,\dots, i_k}$ so that the diagram
\[
\xymatrix@R=10pt@C=10pt{
Z \ar[rr]^-\theta \ar[dd]_-\pi & & G/B \ar[dd]^-{p_{i_k}} \\
& \square
\\
Z' \ar[rr]^-{p_{i_k}\circ\, \theta'} & & G/P_{i_k}
}
\]
is a fiber square; the morphism $p_{i_k}$ is the natural projection.
In fact, $p_{i_k}: G/B \to G/P_{i_k}$ is the projectivization of a 
homogeneous rank-$2$ vector bundle, hence so is $\pi: Z\to Z'$.
From this construction it follows that $Z$ is a smooth
projective variety of dimension~$k$.

The Bott-Samelson variety $Z$ is equipped with a simple normal
crossing (SNC) divisor~$\partial Z $, constructed inductively as follows.  If
$Z=\pt$, then $\partial Z=\emptyset$. In general,
the map $\pi$ admits a section $\sigma$, defined as
the unique map $Z'\to Z$ making the following diagram commute:
\[
\xymatrix@C=30pt{
Z' \ar@{.>}[dr]|-{\exists!\sigma} \ar@/^1pc/[drr]^{\theta'} \ar@/_1pc/[ddr]_{\id_{Z'}} \\
& Z \ar[r]^-\theta \ar[d]_-\pi & G/B \ar[d]^-{p_{i_k}} \\
& Z' \ar[r]^-{p_{i_k}\circ\, \theta'} & G/P_{i_k}
}
\]
In particular, $\theta'=\theta\circ \sigma$.
We let $D_k:=\sigma(Z')$, and then
the SNC divisor on $Z$ is defined by
\[ 
\partial Z = \pi^{-1}(\partial Z') \cup D_k 
\] 
where $\partial Z'$ is the inductively constructed SNC divisor on $Z'$.
The following result is well known, see e.g.,
\cite[\S2.2]{brion.kumar:frobenius}. 

\begin{prop}\label{prop:resolution} 
If $s_{i_1} \ldots s_{i_k}$ is a {\em reduced\/} word for $w$, then
the image of the composition $\theta= \pr_1 \circ\, \theta_1:
Z_{i_1,\dots,i_k} \to G/B$ is the Schubert variety $X(w)$. Moreover,
$\theta^{-1}(X(w) \smallsetminus X(w)^\circ) = \partial Z_{i_1,\dots,i_k}$
and the restriction map
\[ 
\theta: Z_{i_1, \dots,i_k} \smallsetminus \partial Z_{i_1,\dots,i_k}
\to X(w)^\circ 
\] 
is an isomorphism. 
\end{prop}
The proposition implies that the Bott-Samelson variety
$Z_{i_1,\dots,i_k}$ is a log-resolution of the Schubert variety
$X(w)$.

\section{Equivariant $\Kt$-theory of flag manifolds and Demazure-Lusztig operators} 
In this section we recall the definition and basic properties of
equivariant $\Kt$-theory of flag manifolds, and of certain
Demazure-Lusztig operators acting on equivariant $\Kt$-theory. This
setup is well-known from the theory of Hecke algebras, see
e.g.,~\cite{lusztig:eqK} and~\cite{ginzburg:methods}.

\subsection{Equivariant $\Kt$-theory} 
Let $X$ be a smooth, quasi-projective algebraic variety endowed with a
$T$-action. The (algebraic) equivariant $\Kt$-theory ring $K_T(X)$ is
the ring generated by symbols $[E]$, where $E \to X$ is an equivariant
vector bundle, modulo the relations $[E]=[E_1]+[E_2]$ for all short
exact sequences $0 \to E_1 \to E \to E_2 \to 0$ of equivariant vector
bundles. The ring addition is given by direct sums, and multiplication
by tensor products. Since $X$ is smooth, every (equivariant) coherent
sheaf has a finite resolution by (equivariant) vector bundles
\cite[Proposition~5.1.28]{chriss2009representation}, and $K_T(X)$
coincides with the Grothendieck group of (equivariant) coherent
sheaves on $X$. The ring $K_T(X)$ is an algebra over the Laurent
polynomial ring $K_T(\pt) =\bbZ[e^{\pm t_1}, \ldots , e^{\pm t_r}]$ where
$e^{t_i}$ are characters corresponding to a basis of the Lie algebra
of $T$; alternatively $K_T(\pt)$ may be viewed as the representation
ring $R(T)$ of $T$.

In our situation $X=G/B$ and $T$ acts on $X$ by left
multiplication. Since $X$ is smooth, the ring $K_T(X)$ coincides with
the Grothendieck group of $T$-linearized coherent sheaves on~$X$.
There is a pairing, called the $\Kt$-theoretic intersection pairing,
\[ 
\langle - , - \rangle :K_T(X) \otimes
K_T(X) \to K_T(\pt) = R(T)
\] 
defined on classes $[E]$, $[F]$ of vector bundles by
\[ 
\langle [E], [F] \rangle := \int_X E \otimes F = \chi(X; E \otimes F)
\/.
\] 
Here $\chi(X; -)$ is the (equivariant) Euler characteristic, i.e.,
the virtual character
\[
\chi(X;-) = \sum (-1)^i \: \ch_T (H^i(X; -)) \/. 
\]

Let $\calO_w:=[\calO_{X(w)}]$ be the Grothendieck class determined by
the structure sheaf of $X(w)$ (a coherent sheaf), and similarly
$\calO^w:= [\calO_{Y(w)}]$. The equivariant $\Kt$-theory ring has
$K_T(\pt)$-bases $\{\calO_w \}_{w \in W}$ and $\{ \calO^w \}_{w \in
  W}$. Let $\partial X(w) :=X(w) \smallsetminus X(w)^\circ$ be the boundary
of the Schubert variety $X(w)$, and similarly $\partial Y(w)$ the
boundary of $Y(w)$.~It is known that the dual bases of $\{ \calO_w \}
$ and $\{\calO^w \}$ are given by the classes of the ideal sheaves
$\calI^w:= [\calO_{Y(w)}(-\partial Y(w))]$, respectively $ \calI_w:=
[\calO_{X(w)}(-\partial X(w))]$, i.e.,
\[ 
\langle \calO_u , \calI^v \rangle =
\langle \calO^u, \calI_v \rangle = \delta_{u,v} \/. 
\] 
See e.g.,~\cite[Proposition~3.4.1]{brion:flagv} for the
non-equivariant case; the same proof works equivariantly. See
also~\cite[\S2]{graham.kumar:positivity}.  In fact,
\begin{equation}\label{equ:idealstru}
 \calO_w= \sum_{v \le w} \calI_v 
\quad\textrm{ and }  \quad \calI_w
 = \sum_{v \le w} (-1)^{\ell(w) - \ell(v)} \calO_v
\end{equation}
(\cite[Proposition~4.3.2]{brion:flagv}).
We will also need that
\begin{equation}\label{eq:Opair}
\langle \calO_u, \calO^v \rangle =
\begin{cases}
0 & \text{if } u< v \\
1 & \text{if } u\ge v 
\end{cases} \/;
\end{equation}
this is proved in e.g.,~\cite[Theorem~4.2.1]{brion:flagv}.

\subsection{Demazure-Lusztig (DL) operators}\label{ss:DL} 
Fix a simple root $\alpha_i \in \Delta$ and the corresponding minimal
parabolic subgroup $P_i \subseteq G$. 
Consider the diagram
\[
\xymatrix@C=50pt{
G/B \times_{G/P_{i}} G/B
\ar[r]^-{\pr_1}\ar[d]^{\pr_2} & G/B \ar[d]^{p_{i}} \\ 
G/B \ar[r]^{p_{i}} & G/P_{i}
} 
\]
The Demazure operator $\partial_i: K_T(X) \to K_T(X)$
\cite{demazure:desingularisations} is defined by
\[
\partial_i:= (p_i)^* (p_i)_* = (\pr_1)_* \pr_2^*\:.
\]
Since $G/B\to G/P_i$ is a projective bundle, $(p_i)_*(p_i)^*$ is the identity,
and it follows that $\partial_i^2 = \partial_i$.
The operator $\partial_i$ satisfies (e.g., from \cite[Lemma 4.12]{kostant.kumar:KT})
\begin{equation}\label{equ:BGGonstru}
 \partial_i(\calO_w) = 
\begin{cases} 
\calO_{ws_i} & \textrm{ if } ws_i>w \/; \\ 
\calO_w & \textrm{ otherwise } \/. 
\end{cases} 
\end{equation}
One can verify that for $v\in W$, represented by a reduced word $s_{i_1}\cdots
s_{i_k}$, the operator
\[
\partial_v:=\partial_{i_1}\circ\cdots\circ \partial_{i_k}
\]
is independent of the chosen reduced word;~cf.~\cite[\S3]
{kostant.kumar:KT}.  With this definition we have that if
$\ell(uv^{-1})=\ell(u)+\ell(v^{-1})$, then
$\partial_v(\calO_u)=\calO_{uv^{-1}}$.  Since $p_i$ is $G$-equivariant
and $Y(w)=w_0 X(w_0w)$, it follows easily that
$\partial_i(\calO^w)=\calO^{ws_i}$ if $ws_i<w$ and
$\partial_i(\calO^w) = \calO^w$ otherwise.

Fix an indeterminate $y$. The $\lambda_y$-class of a vector bundle $E$
is the class
\[ 
\lambda_y(E):= \sum_k [\wedge^k E] y^k \in K_T(X)[y] \/.
\] 
The $\lambda_y$-class is multiplicative, i.e., if $0 \to E_1 \to E \to
E_2 \to 0$ is a short exact sequence of equivariant vector bundles,
then $\lambda_y(E) = \lambda_y(E_1) \lambda_y(E_2)$ as elements in
$K_T(X)[y]$. We refer to the
books~\cite{fulton.lang:riemann-roch,hirzebruch:topological} for
details in the non-equivariant case. The equivariant case involves no
additional subtleties.

We define next the main operators used in this paper. 
\begin{defin}\label{def:hecke} 
Let $\alpha_i \in \Delta$ be a simple root. Define the operators 
\[
\opT_i: = \lambda_y(T^*_{p_i})\, \partial_i - \id; \quad
\opT^\vee_i: = \partial_i\, \lambda_y(T^*_{p_i}) - \id \/. 
\] 
\end{defin} 
The operators $\opT_i$ and $\opT^\vee_i$ are $K_T(\pt)[y]$-module
endomorphisms of $K_T(X)[y]$. We will occasionally work in
$K_T(X)[y^{\pm 1}]$ and regard these as $K_T(\pt)[y^{\pm 1}]$-module
endomorphisms.

\begin{remark}\label{rmk:convos} 
The operator $\opT^\vee_i$ was defined by Lusztig
\cite[(4.2)]{lusztig:eqK} in relation to affine Hecke algebras and
equivariant $\Kt$-theory of flag varieties. (Lusztig worked in
topological equivariant $\Kt$-theory, but since $X=G/B$ has a
$T$-invariant algebraic cell-decomposition by Schubert cells, the
algebraic and topological equivariant $\Kt$-theories of $X$ are
naturally
isomorphic~\cite{chriss2009representation}[Proposition~5.5.6,
  p.~272].)~As we shall see below, the `dual' operators~$\opT_i$
arises naturally in the study of motivic Chern classes of Schubert
cells. In an algebraic form, the operators $\opT_i$ appeared recently in
\cite{brubaker.bump.licata,lee.lenart.liu:whittaker} and 
 \cite{MSA:whittaker}, in relation to Whittaker functions.
\qede\end{remark}

\begin{lemma}\label{lemma:adjoint} 
The operators $\opT_i$ and $\opT^\vee_i$ are 
adjoint to each other. That is, for every $a, b \in K_T(X)$, 
\[ 
\langle \opT_i(a), b \rangle = \langle a , 
\opT^\vee_i(b) \rangle \/. 
\] 
The same equality holds for $a, b \in K_T(X)[y^{\pm 1}]$, if one
extends the pairing bilinearly in $y$.
\end{lemma}

\begin{proof}  
The identity is self adjoint and $\partial_i$ is also self adjoint.
Indeed, by the projection formula
\[ 
\langle \partial_i (a), b \rangle = \int_{G/B} p_i^*(p_i)_*(a) \cdot b
= \int_{G/P_i} (p_i)_*(a)\cdot (p_i)_*(b) \/,
\] 
and the last expression is symmetric in $a,b$. It remains to show that
coefficient of $y$ in both sides is the same, i.e., $\langle T^*_{p_i}
\partial_i (a), b \rangle = \langle a, \partial_i T^*_{p_i} (b)
\rangle$.  We calculate
\[ 
\begin{split} 
\langle T^*_{p_i} \partial_i (a) , b \rangle &= 
\int_{G/B} T^*_{p_i}  p_i^*((p_i)_* a) \cdot b = 
\int_{G/P_i} (p_i)_*(a) \cdot (p_i)_*( T^*_{p_i} \cdot b)  \\ 
&= \int_{G/B} a \cdot p_i^*(p_i)_*( T^*_{p_i} \cdot b) = 
\langle a, \partial_i(T^*_{p_i} \cdot b) \rangle \/.\qedhere
\end{split}
\] 
\end{proof}

According to Lusztig's result (\cite[Theorem in~\S5]{lusztig:eqK}),
the operators $\opT^\vee_i$ satisfy the braid relations and the
quadratic relations defining the Hecke algebra $H_W(-y)$ of the Weyl
group~$W$ with parameter $-y$. (In the language of this paper, the
variable $q$ from~\cite{lusztig:eqK} satisfies $q=-y$.)  Since the
 $\Kt$-theoretic pairing is non-degenerate,
Lemma~\ref{lemma:adjoint} implies that both sets of operators
$\opT_i$ and $\opT^\vee_i$ satisfy the same identities. We
record this next.

\begin{prop}[Lusztig]\label{prop:hecke-relations} 
The operators $\opT_i$ and $\opT^\vee_i$ satisfy the braid relations
for the Weyl group $W$. For each simple root $\alpha_i \in \Delta$ the
following quadratic relations hold:
\begin{equation}\label{eq:quad}
(\opT_i + \id) (\opT_i + y) = (\opT^\vee_i + \id) 
(\opT^\vee_i + y) = 0 \/. 
\end{equation}
\end{prop}

{}From these relations it follows that the operators $\opT_i$ and
$\opT^\vee_i$ are invertible in~$K_T(X)[y^{\pm 1}]$:
\begin{equation}\label{E:invTi} 
\opT_i^{-1} = - \frac{1}{y} \opT_i - \frac{1+y}{y} \id \quad \/; \quad (\opT^\vee_i)^{-1} 
= -\frac{1}{y} \opT^\vee_i - \frac{1+y}{y} \id \/.
\end{equation}
Since the operators $\opT_i$, $\opT^\vee_i$ satisfy the braid
relations, we may define
\begin{equation}\label{eq:Tv}
\opT_v := \opT_{i_1}\cdots \opT_{i_k}\quad,\quad
\opT^\vee_v := \opT^\vee_{i_1}\cdots\opT^\vee_{i_k}
\end{equation}
for $v\in W$ represented by a reduced word $s_{i_1}\cdots s_{i_k}$.

The cohomological versions of the Demazure-Lusztig operators, which
appear in the study of degenerate Hecke algebras
\cite{ginzburg:methods,aluffi.mihalcea:eqcsm} are self inverse (i.e.,
$(\opT_i^{coh})^2 = \id$) and therefore satisfy the relations of the
group algebra $\bbZ[W]$ of the Weyl group. Obviously, this is not true
in $\Kt$-theory, due to~\eqref{E:invTi}.  However, the multiplication
of Demazure-Lusztig operators behaves rather nicely, as shown in the
following proposition.

\begin{prop}\label{prop:Heckemult} 
Let $u, v \in W$ be two Weyl group elements. Then 
\begin{equation}\label{eq:TuTvm1}
\opT_u \cdot \opT_v^{-1} = c_{uv^{-1}}(y)
\opT_{uv^{-1}} + \sum_{w < uv^{-1}} c_w(y) \opT_w \/, 
\end{equation}
where $c_{uv^{-1}}$ and $c_w(y)$ are rational functions in $y$
determined by $u,v$. Further, if $\ell(uv^{-1}) = \ell(u)
+\ell(v^{-1})$, then $c_{uv^{-1}}(y) = (-y)^{-\ell(v)}$.  The same
statements hold for the multiplication $\opT^\vee_u \cdot
(\opT^\vee_v)^{-1}$ of the dual operators.
\end{prop} 
To prove Proposition~\ref{prop:Heckemult} we need the following lemma.

\begin{lemma}\label{lemma:bruhat} 
Let $u,v \in W$ be two Weyl group elements, and let $s$ be a root
reflection such that $us>u$ and $vs < v$. Then $usv^{-1} < uv^{-1}$. 
\end{lemma}

\begin{proof} 
(Cf.~\cite[\S5.7]{humphreys:reflection}.)
Let $s = s_\alpha$ for some positive root $\alpha$.  By hypothesis
$us_\alpha>u$ and $vs_\alpha<v$, hence $u(\alpha)>0$ and
$v(\alpha)<0$. Since $uv^{-1}(-v(\alpha))<0$, it follows that
$uv^{-1}s_{v(\alpha)}<uv^{-1}$.  On the other hand we have
$s_{v(\alpha)} = v s_\alpha v^{-1}$; therefore
$uv^{-1}s_{v(\alpha)}=us_\alpha v^{-1}$, and we are done.
\end{proof}

\begin{proof}[Proof of Proposition~\ref{prop:Heckemult}] 
We use ascending induction on $\ell(v) \ge 0$. The statement is clear
if $\ell(v) =0$. For $\ell(v) >0 $ write $v = v' s_k$ where $\ell(v')
< \ell(v)$. By definition, $\opT_u \cdot \opT_v^{-1} = \opT_u \cdot
\opT_{k}^{-1} \cdot \opT_{v'}^{-1}$. We have two cases: $u s_k < u$
and $us_k>u$. Consider first the situation $us_k <u$. Then $u$ has a
reduced decomposition ending in $s_k$, i.e., $u = u' s_k$ and
$\ell(u') < \ell(u)$. Then
\[ 
\opT_u \cdot \opT_{k}^{-1} \cdot \opT_{v'}^{-1} =
\opT_{u'} \cdot \opT_k \cdot \opT_{k}^{-1} \cdot
\opT_{v'}^{-1} = \opT_{u'} \cdot \opT_{v'}^{-1}
\/, 
\] 
and since $v' < v$ the result is known by induction. Assume next that
$us_k >u$. Using equation~\eqref{E:invTi} we obtain
\[ 
\opT_u \cdot \opT_v^{-1} = \opT_u \cdot
\opT_k^{-1} \cdot \opT_{v'}^{-1} = - \frac{1}{y}
\opT_u \cdot (\opT_k + y+1) \cdot \opT_{v'}^{-1}
= - \frac{1}{y} \opT_{us_k} \cdot \opT_{v'}^{-1} -
\frac{1+y}{y} \opT_u \cdot \opT_{v'}^{-1} \/. 
\] 
By induction, the leading term of $\opT_{us_k} \cdot \opT_{v'}^{-1}$
is $\opT_{us_k v'^{-1}}{=\opT_{uv^{-1}}}$ and the leading term of $
\opT_u \cdot \opT_{v'}^{-1} $ is $\opT_{u v'^{-1}}$. Observe that
$vs_k = v'<v$.  From Lemma~\ref{lemma:bruhat} we obtain that
$uv'^{-1}=us_kv^{-1}<uv^{-1}$, and this concludes the proof of the
existence of an expression~\eqref{eq:TuTvm1} as stated. If
$\ell(u)+\ell(v^{-1})=\ell(uv^{-1})$, then arguing again by ascending
induction on $\ell(v)$ gives, with notation as above,
\[
\opT_u \cdot \opT_v^{-1} 
= - \frac{1}{y} \opT_{us_k} \cdot \opT_{v'}^{-1} + \cdots
= - \frac 1y \frac 1{(-y)^{\ell(v')}} \opT_{uv^{-1}} + \cdots\,,
\]
proving that $c(uv^{-1})=(-y)^{-\ell(v)}$ as claimed.

The statements for the dual operators are proved in the same way.
\end{proof}

\subsection{Actions on Schubert and fixed point bases} 
We will need several formulas concerning the action of the
Demazure-Lusztig operators on Schubert classes and on the classes
determined by the torus fixed points. For instance, the definition and
a standard localization argument
(cf.~Lemma~\ref{lem:actiononfixedpoint}(a) below) imply that, for
every $w \in W$,
\[
\opT_i( \calO_w) = 
\begin{cases}
(1+y e^{-w\alpha_i}) \calO_{ws_i} + l.o.t. & \textrm{ if } ws_i>w;\\ 
y e^{w\alpha_i} \calO_{w} + l.o.t.& \textrm{ if } ws_i<w\/. 
\end{cases}
\]
where l.o.t.~(lower order terms) stands for a sum of terms $P(y, e^t)
\calO_u$ with $u<ws_i$ on the first branch and $u <w$ on the second
branch. A similar formula holds for the dual operator:
\begin{equation}\label{E:Tiveeact} 
\opT^\vee_i (\calO^w) = (1+ y e^{w_0w(\alpha_i)}) \calO^{ws_i} +
l.o.t.\quad \textrm{ if } ws_i < w \/,
\end{equation} 
where now the l.o.t.~consist of multiples of $\calO^v$ for $v > ws_i$.
Consider next the localized equivariant $\Kt$-theory ring
\[
K_T(G/B) \hookrightarrow\: K_T(G/B)_{\loc} 
:= K_T(G/B) \otimes_{K_T(\pt)} \Frac( K_T(\pt)) \:,
\]
where $\Frac$ denotes the fraction field. The Weyl group elements
$w\in W$ are in bijection with the torus fixed points $e_w:=wB \in
G/B$. Let $\iota_w:=[\calO_{e_w}] \in K_T(G/B)_{\loc}$ be the class of
the structure sheaf of $e_w$. By the localization theorem, the classes
$\iota_w$ form a basis for the localized equivariant $\Kt$-theory
ring; we call this the {\em fixed point basis}. For a weight $\lambda$
consider the $G$-equivariant line bundle $\calL_\lambda := G \times^B
\bbC_\lambda$ with character $\lambda$ in the fiber over $1.B$.\footnote{
Observe that our notation is opposite to the one 
in~\cite[p.~63, formula~(7)]{brion.kumar:frobenius}.}
For example, the relative cotangent bundle $T^*_{p_i}$ for the projection
$p_i$ is isomorphic to $\calL_{\alpha_i}$.  We need the following
lemma.

\begin{lemma}\label{lem:actiononfixedpoint} 
The following formulas hold in $K_T(G/B)_{\loc}[y^{\pm 1}]$:

(a) For every weight $\lambda$, $\calL_\lambda \cdot \iota_w =
e^{w \lambda} \iota_w$;

(b) For every simple root $\alpha_i$, 
\[ 
\partial_i (\iota_w) =\frac{1}{1-e^{w\alpha_i}}\iota_{w}
+\frac{1}{1-e^{-w\alpha_i}}\iota_{ws_{i}} \/; 
\]

(c) The action of the operator $\opT_i$ on the fixed point basis is
given by the following formula
\[
\opT_i(\iota_{w})=-\frac{1+y}{1-e^{-w\alpha_i}}\iota_{w}
+\frac{1+ye^{-w\alpha_i}}{1-e^{-w\alpha_i}}\iota_{ws_{i}}.
\]

(d) The action of the adjoint operator $\opT^\vee_i$ is given by

\[
\opT^\vee_i(\iota_{w})=-\frac{1+y}{1-e^{-w\alpha_i}}\iota_{w}
+\frac{1+ye^{w\alpha_i}}{1-e^{-w\alpha_i}}\iota_{ws_{i}} \/.
\]

(e) The action of the inverse operator $(\opT^\vee_i)^{-1}$is given
by
\[
(\opT^\vee_i)^{-1}(\iota_{w})=- \frac{1+y^{-1}}{1- e^{w\alpha_i}}
\iota_{w}-\frac{y^{-1}+e^{w\alpha_i}}{1-e^{-w\alpha_i}}\iota_{ws_{i}}\/.
\]
\end{lemma}

\begin{proof} 
Part~(a) is a standard localization calculation, based on the fact
that the fiber of~$\mathcal{L}_\lambda$ over the fixed point $e_w$ is
a one-dimensional $T$-module of weight $w(\lambda)$.  For part~(b),
notice that $\partial_i (\iota_w) = [\calO_{p_i^{-1}(e_w)}]$, where
$p_i:G/B \to G/P_i$ is the projection, and (abusing notation) $e_w$
also denotes the corresponding fixed point in $G/P_i$. The fiber
$p_i^{-1}(e_w)\cong \bbP^1$ equals $w.[X(s_i)]$, the $w$-translate of
the Schubert curve.  It contains only two $T$-fixed points:
$e_w$~and~$e_{ws_i}$.  It follows that $[\calO_{p_i^{-1}(e_w)}] = a
\iota_w + b \iota_{ws_i}$ for suitable $a, b$.  By the projection
formula we can regard this as an expansion in the localized
equivariant $\Kt$-theory of the fiber itself. By the localization
theorem (see e.g.,~\cite{nielsen:diag}, and also \cite[\S5.10]{chriss2009representation}), 
\[ 
1 = a \cdot (\iota_w)|_{e_w} =a (1 - e^{w \alpha_i}) \/; \quad 
1 =b\cdot  (\iota_{ws_i})|_{e_{ws_i}} = b (1 - e^{- w \alpha_i}) \/. 
\]   
Then the statement in (b)
follows.  Part (c) follows from (a) and~(b), applied to
$\opT_i(\iota_{w})=(1+y\calL_{\alpha_i})\partial_i(\iota_w)-\iota_w$.
Formula (d) for the adjoint operator follows similarly.  Finally, part
(e) follows because $(\opT^\vee_i)^{-1}=-y^{-1}(\opT^\vee_i+y+1)$
(cf.~equation~\eqref{E:invTi}).
\end{proof}

We also record the action of several specializations of the
Demazure-Lusztig operators.

\begin{lemma}\label{lemma:yspec} 
(a) The specializations at $y=0$ satisfy
\[ 
(\opT_i)_{y=0} = (\opT^\vee_i)_{y=0} = \partial_i-\id \/.
\] 
Further, the following holds for every $w \in W$:
\[ 
(\partial_i - \id) (\calI_w )= 
\begin{cases} 
\calI_{ws_i} & \textrm{ if } ws_i > w \\ 
- \calI_{w} & \textrm{ if } ws_i < w  \/.
\end{cases} 
\]
(b) Let $w \in W$. Then the specializations at $y=-1$ satisfy
\[  
(\opT_i)_{y=-1} (\iota_w) = \iota_{ws_i} ; \quad 
(\opT^\vee_i)_{y=-1} (\iota_w)= \frac{1- e^{w \alpha_i}}
{1- e^{-w\alpha_i}} \iota_{ws_i} = -e^{w \alpha_i} \iota_{ws_i} \/. 
\]
In other words, this specialization is compatible with the right Weyl
group multiplication.
\end{lemma}

\begin{proof} 
Part~(a) is an easy exercise using
the definition of $\opT_i$ and identities~\eqref{equ:idealstru} 
and~\eqref{equ:BGGonstru}.
Part~(b) is immediate 
from Lemma~\ref{lem:actiononfixedpoint}. 
\end{proof}

\section{Motivic Chern classes}\label{s:MC} We recall the definition
of the motivic Chern classes,
following~\cite{brasselet.schurmann.yokura:hirzebruch}.  For now let
$X$ be a quasi-projective, complex algebraic variety, with an action
of {the complex torus} $T$. First we recall the definition of the
(relative) motivic Grothendieck group $\Groth^T(\var/X)$ of varieties
over $X$, mostly following Looijenga's notes~\cite{looijenga:motivic}
{and Bittner's papers~\cite{bittner:universal,bittner:zeta}}. For
simplicity, we only consider the $T$-equivariant quasi-projective
context, in~\cite{bittner:universal}), which is enough for all applications in
this paper.  The group $\Groth^T(\var/X)$ is the free abelian group
generated by symbols $[f: Z \to X]$ {for isomorphism classes of
$T$-equivariant morphisms $f : Z \to X$, where $Z$ is a
quasi-projective $T$-variety,} modulo the usual additivity relations
\[ [f: Z \to X] = [f: U \to X] + [f:Z  \smallsetminus U \to X] \] for $U
\subseteq Z$ an open invariant subvariety.  If $X=\pt$ then
$\Groth^T(\var/\pt)$ is a ring with the product given by the external
product of morphisms, and the groups $\Groth^T(\var/X)$ have a module
structure over $\Groth^T(\var/\pt)$ also given by the external product.
For every equivariant morphism $g:X \to Y$ of quasi-projective
$T$-varieties there is a {functorial} push-forward $g_!: \Groth^T(\var/X)
\to \Groth^T(\var/Y)$ (given by composition) and pull-back
$g^*:\Groth^T(\var/Y) \to \Groth^T(\var/X)$ (given by fiber product).
Finally there are external products $$\times: \Groth^T(\var/X)\times
\Groth^T(\var/X') \to \Groth^T(\var/X\times X'); \quad [f]\times [f']\mapsto
[f\times f'],$$ which are $\Groth^T(\var/\pt)$-bilinear and commute with
push-forward and pull-back.  In particular push-forward $g_!$ and
pull-back $g^*$ are $\Groth^T(\var/\pt)$-linear.

\begin{remark} 
(Cf.~\cite[\S0]{brasselet.schurmann.yokura:hirzebruch}.)
For every variety $X$, similar functors can be defined on the ring of
constructible functions $\calF(X)$, and the Grothendieck ring
$\Groth(\var/X)$ may be regarded as a motivic version of $\calF(X)$. In
fact, there is a map $e: \Groth(\var/X) \to \calF(X)$ sending $[f:Y \to
  X]$ to $f_!(\one_Y)$, where $f_!(\one_Y)$ is defined using compactly
supported Euler characteristic of the fibers. The map $e$ is a group
homomorphism, and if $X=\pt$ then $e$ is a ring homomorphism. The
constructions extend equivariantly 
{to $\Groth^T(\var/X) \to \calF^T(X)$, with}
$\calF^T(X)\subseteq
\calF(X)$ the subgroup of $T$-invariant constructible functions.
\qede\end{remark}
 
\begin{theorem}\label{thm:existence} 
Let $X$ be a quasi-projective, non-singular, complex algebraic variety
with an action of the torus $T$. There exists a unique natural
transformation $\MC_y: \Groth^T(\var/X) \to K_T(X)[y]$ satisfying the
following properties:
\begin{enumerate} 
\item[(1)] It is functorial with respect to $T$-equivariant proper
  morphisms of non-singular, quasi-projective varieties.
\item[(2)] 
It satisfies the normalization condition 
\[ 
\MC_y[\id_X: X \to X] = \lambda_y(T^*_X) = \sum y^i [\wedge^i T^*_X]_T
\in K_T(X)[y]\/.
\]
\end{enumerate}
The transformation $\MC_y$ satisfies the following properties:
\begin{enumerate}
\item[(3)] It commutes with external products:
\[ 
\MC_y[f\times f': Z\times Z'\to X\times X']=\MC_y[f: Z\to X] \boxtimes
\MC_y[f': Z'\to X'] \:.
\]

\item[(4)] For every smooth, $T$-equivariant morphism $\pi: X \to Y$ of
  quasi-projective and non-singular algebraic varieties, and any $[f:
    Z \to Y] \in \Groth^T(\var/Y)$, the following Verdier-Riemann-Roch
  (VRR) formula holds:
\[ 
\lambda_y(T^*_\pi) \cdot \pi^* \MC_y[f:Z \to Y] = \MC_y[\pi^* f:Z
  \times_Y X \to X] \/.
\]
\end{enumerate} 
\end{theorem}

\begin{proof} 
In the non-equivariant case all these statements are proved
in~\cite[Theorem~2.1]{brasselet.schurmann.yokura:hirzebruch}. Feh{\'e}r,
Rim{\'a}nyi and Weber used similar ideas to extend the motivic Chern
classes to the equivariant situation in~\cite{feher2018motivic}. They
proved that there is a well defined class
\[
\MC_y[f:Z \to X] \in K_T(X)[y]
\]
for $Z$ and $X$ smooth, but omitted the details showing functoriality
properties, referring instead back
to~\cite{brasselet.schurmann.yokura:hirzebruch}. For completeness, we
prove this theorem, using the definition (and well-definedness)
from~\cite{feher2018motivic}.  First one gets as
in~\cite{bittner:universal} a tautological isomorphism
\[
\Groth^T(\sm/X)\to \Groth^T(\var/X): [f:Z \to X]\mapsto [f:Z \to X]\:,
\]
where $\Groth^T(\sm/X)$ is the corresponding (relative) motivic
Grothendieck group of {\em smooth\/} quasi-projective $T$-varieties
mapping to $X$, with the corresponding `additivity' relation only
asked for $T$-invariant open subsets $U\subseteq Z$ with $Z\smallsetminus
U$ a smooth closed subvariety of $Z$.  Here the inverse  image of $[f:Z
  \to X] \in \Groth^T(\var/X)$ in $\Groth^T(\sm/X)$ for $Z$ possibly singular
is defined as
\[
[f:Z \to X] \mapsto \sum_i \: [f:Z_i \to X]
\]
where $Z= \bigsqcup Z_i$ is a decomposition into a finite disjoint
union of $T$-invariant (locally closed) smooth subvarieties
$Z_i\subseteq Z$.  Using this presentation, $\MC_y$ is determined by
its value on $\Groth^T(\sm/X)$: i.e., we have $\MC_y[f:Z \to X] := \sum_i
\: \MC_y[f:Z_i \to X]$ with notation as above. If $U$ is a smooth
quasi-projective $T$-variety we consider the definition of $\MC_y[f: U
  \to X]$ from~\cite{feher2018motivic} and show that it satisfies the
corresponding `additivity' relation. Let $f: U \to X$ be an
equivariant morphism with $U$ quasi-projective and non-singular. Then
there exists a non-singular quasi-projective algebraic variety $\ovU$
containing~$U$ such that the following are satisfied:
\begin{itemize} 
\item $\ovU$ admits an action of $T$ and the inclusion $i:U
  \to \ovU$ is an open, $T$-equivariant embedding,

\item the boundary $D:=\ovU \smallsetminus U =
  \bigcup_{i=1,\dots,s} D_i$ is a $T$-invariant simple normal crossing
  divisor with smooth irreducible components $D_i$,

\item there exists a {\em proper\/} $T$-equivariant morphism $\ovf:
  \ovU \to X$ such that $f = \ovf \circ i$
(such a morphism $\ovf$ is then also projective, by 
e.g.,~\cite[Lemma 28.41.13]{StacksProj}).
\end{itemize}
This follows as in~\cite[p.~544]{weber:HBB} from the
existence of a $T$-equivariant projective completion~\cite{sumihiro:equivariant},
\cite[Theorem~5.1.25]{chriss2009representation} and  from 
equivariant resolution of singularities 
\cite{bierstone:resolution}.
 Then `additivity and normalization' forces
\begin{equation}\label{def:MC}
\MC_y[f: U \to X]:=\sum_{I\subseteq \{1,\dots,s\}}\:
(-1)^{|I|}f_{I*}\lambda_y(T^*D_I)\:,
\end{equation}
with $D_I:=\cap_{i\in I} D_i$ (and $D_{\emptyset}:=\ovU$) and $f_I:=
\ovf|_{D_I}$. By~\cite[\S2.4]{feher2018motivic} the right-hand side is
independent of all choices. In particular $\MC_y$ satisfies the
normalization property from part~(2) of the statement.  We show next
that the transformation $ \MC_y$ satisfies the corresponding
`additivity' property. Let $Y\subseteq U$ be a closed $T$-invariant
subvariety.  By induction on the number of connected components of
$Y$, we can assume $Y$ is connected. Then one can find as
in~\cite{bittner:universal} a partial compactification of $U$ and $f$
as before in such a way that the closure $\overline{Y}$ of $Y$ in
$\ovU$ is smooth and $\overline{Y}$ has normal crossing with $D$. Then
$ \overline{Y}$ and $ \ovf|_{\overline{Y}}$ is such a partial
compactification of $Y$ and $f|_Y$ with the corresponding simple normal
crossing divisor $\overline{Y} \cap D$. Let us now first assume that
$Y$ is hypersurface in $U$, so that $\overline{Y} \cup D =
\overline{Y} \cup \bigcup_{i=1,\dots,s} D_i$ is a simple normal
crossing divisor in $\ovU$.  Then one can use $\ovU$ and $\ovf$ as a
partial compactification of $f: U\smallsetminus Y\to X$, with the
corresponding simple normal crossing divisor $\overline{Y} \cup D$.
 Using in this context the definition given in~\eqref{def:MC}, one gets
 precisely the sought-for `additivity' property in the case of
a hypersurface $Y$ in~$U$.
 In general, let $\widetilde{U}$ be the
blow-up of $\ovU$ along $\overline{Y}$, with exceptional divisor $E$
and $\widetilde{D}_i$ the strict transform of $D_i$ ($i=1,\dots,s$),
as well as $\tilde{f}: \widetilde{U}\to X$ the induced proper
morphism. Then $\widetilde{D}_I$ is the blow-up of $D_I$ along
$\overline{Y}_I:=\overline{Y}\cap D_I$ with exceptional divisor
$E_I:=E \cap \widetilde{D}_I$ for $I\subseteq \{1,\dots,s\}$. As
observed in~\cite[Corollary~0.1 and
  p.~8]{brasselet.schurmann.yokura:hirzebruch} and~\cite[\S2.4]
{feher2018motivic}, the key equality needed is the following `blow-up
relation',
\[
\MC_y[\tilde{f}: \widetilde{D}_I \to X] -\MC_y[\tilde{f}: E_I \to X] =
\MC_y[\ovf: D_I \to X] -\MC_y[\ovf: \overline{Y}_I \to X] \/;
\]
the heart of its proof relies on vanishing of certain sheaf cohomology
groups proved in~\cite[Proposition~3.3]{GNA}. 
Using the blow-up
relations and the additivity from the hypersurface case for the partial
compactification $\widetilde{X}$ of $U\smallsetminus Y$, with corresponding
simple normal crossing divisor $E \cup \widetilde{D}$,  proves the `additivity' property in general. This construction
shows that the transformation $\MC_y$ is determined by its image on
classes $[f:Z \to X]$ where $Z$ is a {\em non-singular}, irreducible,
quasi-projective algebraic variety and $f$ is a $T$-equivariant {\em
  proper} morphism.

To prove part~(1) of the statement, i.e., functoriality with respect
to $T$-equivariant proper morphisms, observe that if $g: X\to X'$ is a
proper equivariant morphism of smooth quasi-projective $T$-varieties,
then $\ovU$ and $\overline{g\circ f}:=g\circ \ovf$ is a partial
compactification of $g\circ f: U\to X'$, with $(g\circ f)_I=g\circ
f_I$ such that
\[
\MC_y[g\circ f: U \to X'] =\sum_{I\subseteq \{1,\dots,s\}}\:
(-1)^{|I|}g_*f_{I*}\lambda_y(T^*D_I)=g_*\left( \MC_y [f: U \to
  X]\right) \:.
\]
With the construction of $\MC_y$ given above, the proofs of parts (3)
and (4) of the statement follow as in the non-equivariant case
of~\cite[Theorem~2.1]{brasselet.schurmann.yokura:hirzebruch}, by
making all $\Kt$-theory classes and morphisms equivariant.
\end{proof}

If one forgets the $T$-action, then the equivariant motivic Chern
class recovers the non-equivariant motivic Chern class from
\cite{brasselet.schurmann.yokura:hirzebruch}, either by its
construction, or by the properties (1)-(2) from
Theorem~\ref{thm:existence} and the corresponding results
from~\cite{brasselet.schurmann.yokura:hirzebruch}. Further,
Theorem~\ref{thm:existence} and its proof work more generally for a
possibly singular, quasi-projective $T$-equivariant base variety $X$,
provided one works with the Grothendieck group $K_0^T(X)$ of
$T$-equivariant coherent $\calO_X$-modules; then one obtains $\MC_y:
\Groth^T(\var/X) \to K_0^T(X)[y]$.

Since most of the time the variety $X$ will be understood from the
context, for $Z \subseteq X$ a (not necessarily closed) subvariety
we use the notation
\[ 
\MC_y(Z) := \MC_y[Z \hookrightarrow X] \/. 
\] 
By functoriality, if $Z \subseteq X$ is a {smooth} closed subvariety, then
$\MC_y[i:Z \hookrightarrow X]= i_*(\lambda_y(T^*Z) )$
as elements in $K_T(X)$. We will often suppress the push-forward
notation.

{\begin{remark} There are some differences between the definition 
of the relative equivariant Grothendieck group of varieties in 
\cite{looijenga:motivic,bittner:universal,bittner:zeta}, and hypotheses 
used therein, and those used in this paper. For instance, 
\cite{bittner:universal,bittner:zeta} use finite groups $G$ with a
`good' action; we use a torus $T$ in the complex quasi-projective
context, but can work similarly with a complex linear algebraic group
$G$. {Bittner} also divides by an additional `projective bundle
relation', stating that for a $G$-equivariant projective bundle
$\bbP(V)\to Z$ over a relative $G$-variety $Z\to X$:
\[
[\bbP(V) \to Z \to X] = [\bbP^{\rk(V) -1} \times Z\to Z \to X] \:,
\]
where on the right-hand side $G$ only acts on $Z$ and $X$. This is 
not needed in this paper; we will show in future work
that the motivic Chern class also 
factorizes over this additional relation. Despite these differences, the 
proof of Theorem~\ref{thm:existence} applies to all these contexts, 
following the ideas from {\it loc.~cit.} and \cite{{feher2018motivic}}. 
At the heart of the arguments is the fact that $\Groth^G(\var/X) \simeq 
\Groth^G(\sm/X)$, together with results on equivariant completion,  
equivariant resolution of singularities and an equivariant weak 
factorization theorem 
\cite{sumihiro:equivariant,bierstone:resolution,abramovich:weak,Bergh:eq-weak} 
as used in~\cite{{feher2018motivic}}. If one also divides by the 
'projective bundle relation', then one can also define a {\em motivic 
duality\/} involution on the Grothendieck group (localized at the class 
of the affine line), which commutes under the motivic Chern class
transformation with the Grothendieck-Serre duality involution. For a 
discussion of this involution see~\cite[\S5C]{schurmann2009characteristic} and~\cite[p.~240]{DiMu}; 
also cf.~\eqref{MC-duality} below.
\qede\end{remark}}

The following general lemma is useful.
\begin{lemma}\label{lemma:int} 
Let $X_1, X_2 \subseteq X$ be three $T$-equivariant varieties. 
The following equalities hold in~$K_T(X)[y]$.

(a) The inclusion exclusion formula: 
\[ 
\MC_y(X_1 \cup X_2) = \MC_y(X_1) +
\MC_y(X_2) - \MC_y(X_1 \cap X_2) \/.
\]

(b) If $X_1, X_2, X$ are smooth, and $X_1,X_2$ intersect transversally
(so that $X_1\cap X_2$ is also smooth), then
\begin{equation}\label{eq:motint}
\MC_y (X_1 \cap X_2) = \frac{ \MC_y(X_1) \MC_y (X_2)}
{\MC_y (X)}\/.
\end{equation}

(c) More generally, \eqref{eq:motint} holds if $X_1$, $X_2$ are unions of
smooth hypersurfaces such that $X_1\cup X_2$ is a divisor with simple
normal crossings.
\end{lemma}

\begin{proof} 
The statement in (a) is immediate from the additivity property in the
Grothendieck group. Part (b) follows from standard exact sequences,
using that the normal bundle $N_{X_1 \cap X_2} X$ is the restriction
to $X_1 \cap X_2$ of $N_{X_1} X \oplus N_{X_2} X$. Finally, part (c)
follows from repeated application of (a) and (b), using
inclusion-exclusion and induction on $\dim X$
and on the number of components of $X_1 \cup X_2$.
\end{proof}

\begin{remark}
In part (a), the scheme structure on the union $X_1\cup X_2$ is
irrelevant; in fact, $\MC_y[Z\hookrightarrow X]=
\MC_y[Z_{red}\hookrightarrow X]$ since both classes equal
$\MC_y(X)-\MC_y(X\smallsetminus Z)$.   
In parts (b) and (c) we work in the ring of formal series in $y$, which 
allows us to invert the class $\MC_y(X) = 1 + \sum_{k>0} y^k [\wedge^k T^*X]$;
the right-hand side must actually land in $K_T(X)[y]$, since it equals the left-hand side.
When $X=G/B$, the inverse of $\MC_y(X)$ can also be calculated from
Remark~\ref{rmk:lambdayprod}.  
\qede\end{remark}

\begin{remark}\label{rmk:SMC}  An alternative formulation of 
Lemma~\ref{lemma:int}(b) may be given in terms of motivic Segre classes 
$\SMC_y(X_i) := \frac{\MC_y(X_i)}{\MC_y(X)}$: 
\[ 
\SMC_y(X_1 \cap X_2) = \SMC_y(X_1) 
\cdot \SMC_y(X_2) \/. 
\]
A statement generalizing this formula for $X_i$ possibly singular and under 
a Whitney transversality assumption will be proved in upcoming work.
In the (co)homological case, i.e.~after replacing the motivic Chern
classes by the CSM classes, the analogue of this statement
was proved in~\cite{schurmann:transversality}.
\qede\end{remark}

\section{Motivic Chern classes of Schubert cells via Demazure-Lusztig operators}\label{ss:MCDL} 
In this section we calculate the motivic Chern classes of
Schubert cells in $X=G/B$, using Demazure-Lusztig operators.

We use the definitions and notation from~\S\ref{s:BS}. Fix a
word $(i_1, \ldots , i_k)$ and let $Z:=Z_{i_1, \ldots , i_k}$ and
$Z':=Z_{i_1, \ldots , i_{k-1}}$ be the corresponding Bott-Samelson
varieties. Recall that we have determined a section $\sigma:Z' \to Z$ 
of the projection $\pi:Z \to Z'$, and let $D:={D_k=} \sigma(Z')$. 
The `boundary' $\partial Z := \pi^{-1}(\partial Z') \cup D$ is a
simple normal crossings divisor.
We have the diagram
\[
\xymatrix@R=12pt@C=12pt{
Z \ar[rr]^-\theta \ar[dd]^\pi   & & G/B \ar[dd]^{p_{i_k}} \\
& \square
\\ 
Z' \ar@/^1pc/[uu]^\sigma \ar[rr]^-{p_{i_k}\circ\, \theta'} & & G/P_{i_k}
}
\]
If $w=s_{i_1} \cdot \ldots \cdot s_{i_k}$ is a reduced decomposition,
then $Z$ is a resolution of the Schubert variety~$X(w)$
(Proposition~\ref{prop:resolution}). But note that the construction of
the Bott-Samelson variety for $(i_1, \ldots , i_k)$ can be carried out
even if the word is non-reduced.  Our main theorem is:

\begin{theorem}\label{thm:DLrecursion} 
Let $(i_1, \ldots , i_k)$ be a (possibly non-reduced) word. Then 
\[ 
\theta_* \MC_y[Z\smallsetminus \partial Z\hookrightarrow Z] 
= \opT_{i_k} \theta'_* \MC_y[Z' \smallsetminus \partial Z'\hookrightarrow Z'] \/,
\] 
as elements in $K_T(G/B)[y]$. In particular, if $w \in W$ and $s_i$ is
a simple reflection such that $ws_i > w$, then
\[ 
\MC_y[X(ws_i)^\circ \hookrightarrow G/B]= \opT_i \MC_y[X(w)^\circ
  \hookrightarrow G/B] \/.
\]
\end{theorem}

\begin{proof} 
The second claim follows from the first. Indeed, take any reduced word
$(i_1,\ldots, i_{k-1})$ for $w\in W$, so that $(i_1,\ldots,
i_{k-1},i)$ is a reduced word for $ws_i$.  The restrictions $\theta:Z
 \smallsetminus \partial Z \to X(ws_i)^\circ$ and $\theta':Z' \smallsetminus
\partial Z' \to X(w)^\circ$ are (equivariant) isomorphisms and, by
functoriality, $\theta_*\MC_y[Z \smallsetminus \partial Z\hookrightarrow Z] =
\MC_y[X(ws_i)^\circ \hookrightarrow G/B]$ and $\theta'_*\MC_y[Z'
\smallsetminus \partial Z'\hookrightarrow Z'] 
= \MC_y[X(w)^\circ \hookrightarrow G/B]$.

We now prove the first assertion\footnote{We are especially grateful to one referee for
suggesting a simplification of our original argument.}. 
By the inductive construction of $\partial Z$,
\[
\MC_y[Z\smallsetminus \partial Z\hookrightarrow Z]
=\MC_y[\pi^{-1}(Z'\smallsetminus \partial Z')\hookrightarrow Z]
-\sigma_* \MC_y[Z'\smallsetminus \partial Z'\hookrightarrow Z']\/.
\]
By the VRR formula in Theorem~\ref{thm:existence}~(4),
\[
\MC_y[\pi^{-1}(Z'\smallsetminus \partial Z')\hookrightarrow Z]
=\lambda_y(T^*_\pi)\, \pi^* \MC_y[Z'\smallsetminus \partial Z'\hookrightarrow Z']\/.
\]
Since $T^*_\pi=\theta^* T^*_{p_{i_k}}$, using the projection formula and the 
base change formula $\theta_* \pi^*=(p_{i_k})^*(p_{i_k}\circ \theta')_*$ gives
\begin{align*}
\theta_* \MC_y[\pi^{-1}(Z'\smallsetminus \partial Z')\hookrightarrow Z]
&=\lambda_y(T^*_{p_{i_k}}) (p_{i_k})^*(p_{i_k})_* \theta'_* 
\MC_y[Z'\smallsetminus \partial Z'\hookrightarrow Z'] \\
&=(\lambda_y(T^*_{p_{i_k}})\, \partial_{i_k}) \theta'_* 
\MC_y[Z'\smallsetminus \partial Z'\hookrightarrow Z'] 
\end{align*}
As $\theta\circ \sigma = \theta'$, we obtain
\begin{align*}
\theta_* \MC_y[Z\smallsetminus \partial Z\hookrightarrow Z]
&=(\lambda_y(T^*_{p_{i_k}}) \partial_{i_k} - \id) \theta'_* 
\MC_y[Z'\smallsetminus \partial Z'\hookrightarrow Z']  \\
&=\opT_{i_k} \theta'_* 
\MC_y[Z'\smallsetminus \partial Z'\hookrightarrow Z']
\end{align*}
as stated.
\end{proof}

We record the following corollary.

\begin{corol}\label{cor:lesswsi} 
Let $w \in W$ and let $s_i$ be a simple reflection. Then 
\[ 
\opT_i (\MC_y  (X(w)^\circ)) =
\begin{cases} 
\MC_y( X(ws_i)^\circ) & \textrm{ if } ws_i>w ;\\ 
-(y+1) \MC_y(X(w)^\circ ) - y \MC_y (X(ws_i )^\circ ) & 
\textrm{ if } ws_i < w \/. 
\end{cases}
\]
\end{corol} 

\begin{proof} 
The identity from the $ws_i > w$ branch was proved in
Theorem~\ref{thm:DLrecursion}. Assume that $ws_i <w$. Then the same
result shows that $\MC_y(X(w)^\circ) = \opT_i (\MC_y(X(ws_i)^\circ))$,
thus
\[  
\opT_i (\MC_y (X(w)^\circ)) = \opT_i^2 (\MC_y (X(ws_i)^\circ)) \/. 
\]
By the quadratic relations from
Proposition~\ref{prop:hecke-relations}, $\opT_i^2 = - (y+1)\opT_i -
y \cdot \id$. Now we apply the right-hand side to $\MC_y
(X(ws_i)^\circ)$, using again Theorem~\ref{thm:DLrecursion} and that
$(ws_i) s_i > ws_i$.
\end{proof}

\begin{remark}
In particular, setting $y=-1$:
\[ 
\opT_i|_{y=-1} (\MC_{-1}  (X(w)^\circ)) =\MC_{-1}  (X(ws_i)^\circ)
\]
regardless of whether $ws_i$ precedes or follows $w$ in the Bruhat order. 
Combined with Lemma~\ref{lemma:yspec}, this implies that 
\[ 
\MC_{-1}(X(w)^\circ) = \iota_w \/, 
\] 
the class of the fixed point $e_w$.
The corresponding statement holds for the CSM class in (co)homology,
cf.~\cite[Proposition~6.5(d)]{aluffi.mihalcea:eqcsm}.
\qede\end{remark}

\begin{remark}
Recall that if $w=s_{i_1}\cdots s_{i_k}$ is a reduced decomposition, the 
operator $\opT_w := \opT_{i_1}\cdots \opT_{i_k}$ is well-defined
(cf.~\eqref{eq:Tv}). With this notation,
\[
\MC_y (X(w)^\circ)=\opT_{w^{-1}}(\calO_{\id})
\]
as a consequence of Theorem~\ref{thm:DLrecursion}.
\qede\end{remark}

\begin{example}\label{ex:P1} 
The equivariant motivic Chern classes for $\mathbb{P}^1$ are: 
\[ 
\MC_y(X(\id)) = \calO_{\id}; \quad \MC_y(X(s)^\circ) = (1+ e^{-\alpha_1} y)
\calO_{\mathbb{P}^1} - (1 + (1+ e^{-\alpha_1})y) \calO_{\id} \/. 
\]
To recover the non-equivariant classes from the equivariant ones one 
substitutes $e^\lambda \mapsto 1$ for each weight $\lambda$. 
For instance, the non-equivariant motivic Chern class of 
$X(s)^\circ \subset \mathbb{P}^1$ is 
\[ 
\MC_y(X(s)^\circ) = (1+ y)\calO_{\mathbb{P}^1} - (1 + 2y) \calO_{\id} \/.
\]
(Note that this recovers the examples of the classes of $X(s_1)^\circ$ 
and $X(s_2)^\circ$ given in the introduction.)
\qede\end{example}

\begin{example}\label{ex:FL3} The equivariant motivic Chern classes for larger 
flag manifolds are much more complicated. For instance, the equivariant motivic 
Chern class of the Schubert cell $X(s_1 s_2)^\circ \subset \Fl(3)$ is 
\[ 
\begin{split} 
\MC_y(X(s_1s_2)^\circ) = & (1+ e^{-\alpha_1}y)(1+ e^{-(\alpha_1 
+ \alpha_2)}y) \calO_{s_1 s_2} - \\ & (1+ e^{-\alpha_1}y)(1+(1+ e^{-(\alpha_1
+ \alpha_2)})y) \calO_{s_1}  - \\ & (1 + (1 + e^{-\alpha_1})(1+e^{-\alpha_2})y 
+ e^{-\alpha_2}(1+ e^{-\alpha_1}+ e^{-2\alpha_1})y^2)\calO_{s_2} + \\ 
& (1 + (2+ e^{-\alpha_1} + e^{-\alpha_2} + e^{-(\alpha_1+\alpha_2)})y) \calO_{\id}
+ \\ & (1+ e^{-\alpha_1} + e^{-\alpha_2} + e^{-(\alpha_1+\alpha_2)} + 
e^{-(2 \alpha_1 + \alpha_2)})y^2\calO_{\id} \/. 
\end{split} 
\]
\qede\end{example}

\subsection{Motivic Chern classes in $G/P$} 
Let $P \supset B$ be a parabolic subgroup containing $B$ and let $W^P
\subseteq W$ be the subset of minimal length representatives for
$W/W_P$, the quotient of $W$ by the subgroup $W_P$ generated by the
reflections in $P$. For $w W_P \in W/W_P$, $\ell(wW_P)$ denotes the
length of the (unique) representative of $wW_P$ in $W^P$. The Schubert
cells in $G/P$ are $X(wW_P)^\circ:= BwP/P \subseteq G/P$; then
$X(wW_P)^\circ \simeq \bbA^{\ell(wW_P)}$. The natural projection
$\pi: G/B \to G/P$ sends $X(w)^\circ$ to $X(wW_P)^\circ$ and it is an
isomorphism if $w \in W^P$. From this and the functoriality of motivic
Chern classes it follows that $\forall w\in W^P$,
\[ 
\pi_* \MC_y[X(w)^\circ \hookrightarrow G/B] = \MC_y[X(wW_P)^\circ
 \hookrightarrow G/P] \/.
\] 

\begin{remark} 
In fact, one can prove more: from~\cite[\S2]{BCMP:qkfin} one obtains
that the restriction $\pi|_{X(w)^\circ}: X(w)^\circ \to X(wW_P)^\circ$
is an equivariantly trivial fibration with fiber a Schubert cell of
dimension $\ell(w) - \ell(wW_P)$ in $\pi^{-1}(e_{wW_P}) \simeq P/B$,
regarded as a homogeneous space for the Levi subgroup of $P$. It is
not difficult to show that this implies that for all $w \in W$,
\[ 
\pi_* \MC_y[X(w)^\circ \hookrightarrow G/B] = (-y)^{\ell(w) - \ell(w
  W_P)} \MC_y[X(wW_P)^\circ \hookrightarrow G/P] \/.
\] 
Details of the proof and applications to point counting in
characteristic $p$ will be included in a continuation to this paper.
\qede\end{remark}

\section{The Hecke duality for motivic Chern classes} 
It was proved in~\cite[\S5]{AMSS:shadows} that the Poincar{\'e} duals
of the CSM classes of Schubert cells are given by the operators which
are adjoint to the Hecke-type operators which determine the CSM
classes. The same phenomenon holds in the context of this paper, with
the same idea of proof. However, in this context the DL operators
satisfy the quadratic relations~\eqref{eq:quad}, while in the
cohomological case studied in~\cite{AMSS:shadows} the corresponding
operators are self-inverse.  This leads to somewhat more involved
calculations for motivic Chern classes. In analogy with the dual CSM
class from \cite[Definition 5.3]{AMSS:shadows} we make the following
definition.

\begin{defin}\label{def:firstdual} 
Let $w \in W$. The {\em dual motivic Chern class} is defined by 
\[
\MC_y^\vee(Y(w)^\circ):=(\opT^\vee_{w_0w})^{-1}(\MC_y(Y(w_0)))
=(\opT^\vee_{w_0w})^{-1}(\calO^{w_0})\in K_T(G/B)[y,y^{-1}] \/.
\]
\end{defin}
The name of this class is explained by the following theorem, which is
the $\Kt$-theoretic analogue of~\cite[Theorem 5.7]{AMSS:shadows}.

\begin{theorem}\label{thm:dualbasis}
For every $u, v\in W$, 
\[ 
\langle \MC_y(X(u)^\circ), \MC_y^\vee(Y(v)^\circ)\rangle =
\delta_{u,v}(-y)^{\ell(u)-\dim G/B}\prod_{\alpha>0}(1+ye^{-\alpha})
\/.
\]
\end{theorem}

\begin{remark}  Another interpretation of the dual class, in terms of 
Serre duality, is given in Theorem \ref{thm:dual1} below.
In fact, we could have alternatively {\em defined} the dual class by this theorem, then 
proved that it is given by the operator from Definition~\ref{def:firstdual}. See also 
\cite[Theorem~4.2]{MSA:whittaker} for 
more about the relation between Serre duality and the Hecke involution on 
Demazure-Lusztig operators.

Also note that, geometrically,
the quantity $\prod_{\alpha>0}(1+ye^{-\alpha})$ equals
$\lambda_y(T^*_{w_0}(G/B))$, i.e., it is the $\lambda_y$ class of the
fiber of the cotangent bundle at $w_0$.  
\qede\end{remark}

\begin{proof}[Proof of Theorem~\ref{thm:dualbasis}] 
Using the definition of both flavors of motivic classes, and the fact
that $\opT_i$ and $\opT^\vee_i$ are adjoint to each other, we obtain
that
\[  
\langle \MC_y(X(u)^\circ), \MC_y^\vee(Y(v)^\circ)\rangle = \langle
\opT_{u^{-1}} (\calO_{\id}), (\opT^\vee_{w_0v})^{-1}(\calO^{w_0})
\rangle = \langle \calO_{\id}, \opT^\vee_u \cdot
(\opT^\vee_{w_0v})^{-1}(\calO^{w_0}) \rangle \/.
\] 
By Proposition~\ref{prop:Heckemult}, 
\[
\opT^\vee_u \cdot (\opT^\vee_{w_0v})^{-1} = c_{uv^{-1} w_0}(y)
\opT^\vee_{uv^{-1} w_0} + \sum_{w < uv^{-1} w_0} c_{w}(y)
\opT^\vee_w \/.
\] 
Since $\opT^\vee_w (\calO^{w_0})$ is a combination of Schubert 
classes $\calO^{w'}$ such that $w\leq w'$ and $\langle\calO_{\id},
\calO^w\rangle=\delta_{w, \id}$, $\langle \calO_{\id}, \opT^\vee_u
\cdot (\opT^\vee_{w_0v})^{-1}(\calO^{w_0}) \rangle$ is $0$ unless
$uv^{-1} w_0 = w_0$, i.e., $u=v$. In this case, by~\eqref{E:Tiveeact}, 
the coefficient of $\calO^{\id}$ in
$\opT^\vee_{w_0}(\calO^{w_0})$ is 
$\prod_{\alpha >0} (1+ y e^{-\alpha})$. 
By Proposition~\ref{prop:Heckemult}, the coefficient of
$\opT^\vee_{w_0}$ in $\opT^\vee_u \cdot (\opT^\vee_{w_0u})^{-1}$ is
$(-y)^{-\ell(w_0u)}$. Therefore,
\begin{align*}
\langle \MC_y(X(u)^\circ), \MC_y^\vee(Y(u)^\circ)\rangle&=\langle
\calO_{\id}, \opT^\vee_u \cdot (\opT^\vee_{w_0u})^{-1}(\calO^{w_0})
\rangle\\ 
&=\langle \calO_{\id},(-y)^{-\ell(w_0u)}T^\vee_{w_0}(\calO^{w_0})
\rangle\\ 
&=\langle \calO_{\id},(-y)^{-\ell(w_0u)}\prod_{\alpha >0} (1+ y
e^{-\alpha})\calO^{\id} \rangle\\ 
&=(-y)^{\ell(u)-\dim G/B}\prod_{\alpha >0} (1+ y e^{-\alpha})\/,
\end{align*} 
concluding the proof.
\end{proof}

\begin{remark} 
It is natural to consider the normalized class 
\begin{equation}\label{eq:normMC}
\widetilde{\MC}_y(Y(w)^\circ):= (-y)^{\dim G/B -
  \ell(w)}\MC_y^\vee(Y(w)^\circ) \/.
\end{equation} 
The classes $\widetilde{\MC}_y(Y(w)^\circ)$ are given by the normalized
operator $\calL_i := \opT^\vee_i + (1+y) \id$; cf.~equation
\eqref{E:invTi}. The coefficients in the Schubert expansion of this class
are polynomial in $y$.  
\qede\end{remark}

\begin{example}\label{ex:motcells} 
The motivic Chern class for Schubert cells in $\Fl(3)$ were listed in
the introduction.  The normalized dual motivic classes
$\widetilde{\MC}_y(Y(w)^\circ)$ for the Schubert cells in $\Fl(3)$,
computed using~\eqref{eq:normMC} and Definition~\ref{def:firstdual}, are:
{\small
\[ 
\begin{split} 
\widetilde{\MC}_y(Y(w_0)) &= \calO^{w_0}; \\ 
\widetilde{\MC}_y(Y(s_1 s_2)^\circ) &= (1+y) \calO^{s_1s_2} + y
\calO^{w_0}\/; \\ 
\widetilde{\MC}_y(Y(s_2 s_1)^\circ) &= (1+y) \calO^{s_2 s_1} + y
\calO^{w_0} \/; \\ 
\widetilde{\MC}_y(Y(s_1)^\circ) &= (1+y)^2 \calO^{s_1} + y(1+y)
\calO^{s_1 s_2} + 2y (1+y)\calO^{s_2 s_1} + y^2 \calO^{w_0} \/;\\ 
\widetilde{\MC}_y(Y(s_2)^\circ) &= (1+y)^2 \calO^{s_2} + 2y(1+y)
\calO^{s_1 s_2} + y (1+y)\calO^{s_2 s_1} + y^2 \calO^{w_0} \/;\\ 
\widetilde{\MC}_y(Y(\id)^\circ) &= (1+y)^3 \calO^{\id} + y (1+y)^2
(\calO^{s_1} + \calO^{s_2}) +2 y^2 (1+y)( \calO^{s_1 s_2} 
+ \calO^{s_2 s_1}) + y^3 \calO^{w_0}
\/.
\end{split} 
\]}
An algebra verification, using the fact that $\langle \calO_u,
\calO^v \rangle = 1$ if $u \ge v$ and $\langle \calO_u, \calO^v
\rangle = 0$ otherwise (cf.~\eqref{eq:Opair}), shows that
\[ 
\langle \MC_y(X(u)^\circ), \widetilde{\MC}_y(Y(v)^\circ) \rangle =
(1+y)^{\dim \Fl(3)} \delta_{u,v} \/,
\] 
as prescribed by Theorem~\ref{thm:dualbasis}. 
(Here we are setting
the equivariant variables $e^{\alpha}$ to $1$.)
 At this time we note
that the analogue of the positivity Conjecture~\ref{conj:csmpos} is
false for the dual classes. For instance the coefficient of
$\calO^{s_3 s_1 s_2}$ in the expansion of
$\widetilde{\MC}_y(Y(\id)^\circ) \in K(\Fl(4))$ equals $y^2
(4y-1)(1+y)^3$.  
\qede\end{example}

In the next result we determine the action of the operators
$\opT^\vee_i$ on the dual motivic classes.

\begin{prop}\label{prop:Tidualaction} 
Let $w \in W$ be a Weyl group element and $s_i$ a simple
reflection. Then the following equalities hold:
\begin{equation}\tag{a}
\opT^\vee_i (\MC_y^\vee(Y(w)^\circ)) = 
\begin{cases} 
\MC_y^\vee(Y(ws_i)^\circ) & \textrm{ if } ws_i > w \\ 
-(y+1) \MC_y^\vee(Y(w)^\circ) - y \MC_y^\vee(Y(ws_i)^\circ) 
& \textrm{ if } ws_i<w  
\end{cases}
\end{equation}
\begin{equation}
\tag{b}
(\opT^\vee_i)^{-1} (\MC_y^\vee(Y(ws_i)^\circ)) = 
\begin{cases} 
\MC_y^\vee(Y(w)^\circ) & \textrm{ if } ws_i > w \\ 
-\frac{1}{y} \MC_y^\vee(Y(w)^\circ) - \frac{y+1}{y} 
\MC_y^\vee(Y(ws_i)^\circ) & \textrm{ if } ws_i<w 
\end{cases} 
\end{equation}
\end{prop}

\begin{proof} 
To prove part (a), consider first the case when $ws_i> w$. Then
$w_0 ws_i < w_0 w$, thus $\opT^\vee_{w_0 w} = \opT^\vee_{w_0w s_i}
\opT^\vee_i$. By Definition~\ref{def:firstdual},
\[ 
\opT^\vee_i (\MC_y^\vee(Y(w)^\circ)) =
\opT^\vee_i(\opT^\vee_{w_0w})^{-1}(\calO^{w_0}) =
(\opT^\vee_{w_0ws_i})^{-1}(\calO^{w_0}) = \MC_y^\vee(Y(ws_i)^\circ)\/.
\] 
The situation when $ws_i <w$ is treated as in the proof of
Corollary~\ref{cor:lesswsi}, using that $\opT^\vee_i$~satisfies the
quadratic relations from Proposition~\ref{prop:hecke-relations}. Part
(b) follows from (a) by applying $(\opT^\vee_i)^{-1}$ to both sides.
\end{proof}

\section{Three recursions for localizations of motivic Chern classes}
In this section, we use the Demazure Lusztig operators $\opT_i$ to
obtain recursive relations for the ordinary and dual motivic Chern
classes of Schubert cells. These recursions will be used to compare
the motivic Chern classes both with stable envelopes and with
Casselman's basis. We also record a divisibility property for
localizations of motivic classes, to be used later in the proof of
Theorem~\ref{thm:divintro}.

\subsection{Recursions} 
Consider the localized equivariant $\Kt$-theory ring defined by 
\[
K_T(G/B) \hookrightarrow \: K_T(G/B)_{\loc} := K_T(G/B)
\otimes_{K_T(\pt)} \Frac(K_T(\pt))\:.
\]
The Lefschetz fixed point formula in equivariant $\Kt$-theory
(see e.g.,~\cite[\S5.10]{chriss2009representation}) gives the expansion of
the motivic Chern classes in terms of the fixed point classes
$\iota_w$, for every $w \in W$:
\begin{align}\label{equ:locofmotivic}
\MC_y(X(w)^\circ)&=\sum_{u\leq w}
\MC_y(X(w)^\circ)|_u\frac{\iota_{u}}{\lambda_{-1}(T_u^*(G/B))}
\nonumber\\
&=\sum_{u\leq
  w}\MC_y(X(w)^\circ)|_u\frac{\iota_{u}}{\prod_{\alpha>0}
(1-e^{u\alpha})}\quad \in K_T(G/B)_{\loc}[y] \/.
\end{align}

The following three propositions give recursions formulas for various
flavors of motivic Chern classes. These will be used later to make the
connection with the Hecke algebra action on the principal series
representation. The similarity of the recursions can be explained by
the fact that they are related either by an automorphism of $G/B$ or
by the involution exchanging the Demazure-Lusztig operators.

\begin{prop}\label{prop:charmot}
The localizations $\MC_y(X(w)^\circ)|_u$ are uniquely determined by the
following conditions:
\begin{enumerate}
\item[(a)] 
$\MC_y(X(w)^\circ)|_u=0$, unless $u\leq w$.
\item[(b)]
If $u=w$:
\[
\MC_y(X(w)^\circ)|_w=\prod_{\alpha>0,w\alpha<0}(1+ye^{w\alpha})
\prod_{\alpha>0,w\alpha>0}(1-e^{w\alpha}).
\]
\item[(c)] If $ws_{i}>w$, then
\[
\MC_y(X(ws_{i})^\circ)|_u=-\frac{1+y}{1-e^{-u\alpha_i}}
\MC_y(X(w)^\circ)|_u+\frac{1+ye^{u\alpha_i}}{1-e^{-u\alpha_i}}
\MC_y(X(w)^\circ)|_{us_{i}}.
\]
\end{enumerate}
\end{prop}

\begin{proof}
Part (a) follows because the motivic class is supported on the
Schubert variety $X(w)$. To prove part (b), observe that
$\MC_y(X(w)^\circ)|_w=\MC_y(X(w))|_w$, by additivity and because
$\MC_y(X(v)^\circ)|_w = 0$ for $v < w$ by part (a). Then
\[
\MC_y(X(w)^\circ)|_w=\MC_y(X(w))|_w =
\lambda_y(T^*_wX(w))\lambda_{-1}(N_w^\vee),
\] 
where $T^*_wX(w)$ and $N_w^\vee$ are the fibers at the fixed point $e_w$
of the dual of the cotangent, respectively the conormal bundle for
$X(w)$. (A more general result is proved in Theorem~\ref{thm:MCsmooth}
below.) Part~(c) follows by applying the operator $\opT_i$ to
Equation~\eqref{equ:locofmotivic} and taking the coefficients of
$\iota_{u}$; this requires the action of $\opT_i$ on the fixed point
basis described in Lemma~\ref{lem:actiononfixedpoint}. Finally, the
uniqueness follows by induction on the length of $w$.
\end{proof}

For later use, we also record the similar result for the motivic
Chern class of the opposite Schubert cells.
\begin{prop}\label{prop:charmotoppo}
The localizations $\MC_y(Y(w)^\circ)|_u$ are uniquely determined by the
following conditions:
\begin{enumerate}
\item[(a)]
$\MC_y(Y(w)^\circ)|_u=0$, unless $u\geq w$.
\item[(b)]
If $u=w$:
\[
\MC_y(Y(w)^\circ)|_w=\prod_{\alpha>0,w(\alpha)
  >0}(1+ye^{w\alpha})\prod_{\alpha>0,w(\alpha) < 0}(1-e^{w\alpha}).
\]
\item[(c)]
If $ws_{i}>w$, then
\[
\MC_y(Y(w)^\circ)|_u=-\frac{1+y}{1-e^{-u\alpha_i}}
\MC_y(Y(ws_{i})^\circ)|_u+\frac{1+ye^{u\alpha_i}}
{1-e^{-u\alpha_i}}\MC_y(Y(ws_{i})^\circ)|_{us_{i}}.
\]
\end{enumerate}
\end{prop}

\begin{proof} 
The left multiplication by $w_0$ induces an automorphism of $G/B$
sending $X(w)$ to $Y(w_0 w)$. This is not $T$-equivariant, but it is
equivariant with respect to the map $T \to T$ defined by $t \mapsto w_0 t
w_0$. This induces an automorphism of $K_T(G/B)$ and its localized
version, twisting the coefficients by $w_0$. Then the proposition
follows from Proposition~\ref{prop:charmot} above, by applying $w_0$.
\end{proof}
Similar formulas hold for the dual classes
  $\MC_y^\vee(Y(w)^\circ)$:
\begin{prop}\label{prop:chardualmot}
The localizations $\MC_y^\vee(Y(w)^\circ)|_u$ are uniquely determined
by the following conditions:
\begin{enumerate}
\item[(a)]
$\MC_y^\vee(Y(w)^\circ)|_u=0$, unless $u\geq w$.
\item[(b)]
If $u=w$: 
\[
\MC_y^\vee(Y(w)^\circ)|_w=(-1)^{\dim G/B-\ell(w)}
\prod_{\alpha>0,w\alpha>0}(y^{-1}+e^{-w\alpha})
\prod_{\alpha>0,w\alpha<0}(1-e^{w\alpha}) \/.
\]
\item[(c)]
If $ws_{i}>w$, then
\[
\MC^\vee_y(Y(w)^\circ)|_u=\frac{1+y^{-1}}{e^{u\alpha_i}-1}
\MC^\vee_y(Y(ws_{i})^\circ)|_u+\frac{y^{-1}+
e^{-u\alpha_i}}{e^{-u\alpha_i}-1}
\MC^\vee_y(Y(ws_{\alpha_i})^\circ)|_{us_{\alpha_i}}.
\]
\end{enumerate}
\end{prop}

\begin{proof} 
These formulae are regarded in $K_T(\pt)[y^{-1}] \hookrightarrow \:
\Frac(K_T(\pt))[y^{-1}]$.  The uniqueness follows directly from
induction. So we only need to show that $\MC_y^\vee(Y(w)^\circ)|_u$
satisfies these properties. The support property follows because
$(\opT^\vee_i)^{-1}$ sends a Schubert class $\calO^u$ into classes
supported on $Y(u) \cup Y(u s_i)$; then one applies
Proposition~\ref{prop:Tidualaction}. To calculate the localization at
$w$, we use the duality from Theorem~\ref{thm:dualbasis} and 
the Lefschetz fixed point formula to obtain
\[ 
\begin{split} 
(-y)^{\ell(w) - \dim G/B} \prod_{\alpha > 0} (1+ ye^{-\alpha}) & =
  \langle \MC_y(X(w)^\circ), \MC_y^\vee(Y(w)^\circ) \rangle \\ 
& = \sum_{u \in W} \frac{(\MC_y(X(w)^\circ) \cdot
    \MC_y^\vee(Y(w)^\circ))|_u}{\prod_{\alpha>0} (1- e^{u(\alpha)})}
  \cdot \int_{G/B} \iota_u ~\/.
\end{split} 
\] 
The only non-zero contribution is for $u=w$, and the integral equals
$1$, thus
\[ 
\MC_y^\vee(Y(w)^\circ)|_w = \frac{(-y)^{\ell(w) - \dim G/B}
  \prod_{\alpha > 0} 
(1+ye^{-\alpha})(1-e^{w(\alpha)})}{\MC_y(X(w)^\circ)|_w} \/.
\] 
Part~(b) follows from this and the localization from
Proposition~\ref{prop:charmot}. Part~(c) follows as in
Proposition~\ref{prop:charmot}, using now
Proposition~\ref{prop:Tidualaction} and part~(e) of
Lemma~\ref{lem:actiononfixedpoint}.
\end{proof}

\subsection{A divisibility property for localization coefficients} 
We record the following property which will be used in our
applications to $p$-adic groups.

\begin{theorem}\label{thm:mcdiv}
For every $w\leq u\in W$, the polynomial $\MC_y(Y(w)^\circ)|_u \in
K_T(\pt)[y]$ is divisible~by
\[ 
\prod_{\alpha>0, u\alpha>0}(1+ye^{u\alpha})\prod_{\alpha>0,w\nleq
  us_\alpha<u}(1-e^{u\alpha}) \/.
\]
\end{theorem}

\begin{proof}
By a general property of motivic classes proved
in~\cite[Theorem~5.3(ii)]{feher2018motivic}, the localization
coefficient $\MC_{y}(Y(w)^\circ)|_u$ is divisible by
$\lambda_y(T_u^*Y(u)^\circ)=\prod_{\alpha>0,u\alpha>0}
(1+ye^{u\alpha})$. As $\alpha$ varies in the set of positive roots,
the factors $1-e^{u\alpha}$ and $1+ye^{u \alpha}$ are relative prime
to each other. Then it remains to show that for every $\alpha>0$ such
that $w\nleq us_\alpha<u$, the localization coefficient
$\MC_y(Y(w)^\circ)|_u$ is divisible by $1-e^{u\alpha}$. Let $C \simeq
\bbP^1$ denote the $T$-stable curve connecting the fixed points
$e_u$ and $e_{us_\alpha}$. The $T$-weight of the tangent space $T_{u}
C$ is $-u\alpha$. By the K-theoretic analogue of the GKM conditions,
applied to $\MC_y(Y(w)^\circ)$ (see e.g., 
\cite[Corollary~5.12]{vezzosi-vistoli:higher}),
the difference
\[ 
\MC_y(Y(w)^\circ)|_u - \MC_y(Y(w)^\circ)|_{us_\alpha} 
\]
is divisible by $1-e^{u \alpha}$. Then the claim follows because
$\MC_y(Y(w)^\circ)|_{us_\alpha}
=0$, as the hypothesis $w \nleq us_\alpha$ implies that 
$e_{u s_\alpha} \notin Y(w)$.\begin{footnote}{We thank A. Okounkov for 
comments leading to this proof.}\end{footnote} 
\end{proof}

\section{Motivic Chern classes and $\Kt$-theoretic stable envelopes}\label{s:MCstab} 
In this section, we recall some basic properties of the
$\Kt$-theoretic stable basis of $T^*(G/B)$, including a recursive
relation. Our main references are~\cite{su2017k,
  okounkov2015lectures,okounkov2016quantum}. We compare the recursive
relation obtained in~\cite{su2017k} to the one for motivic Chern
classes, and we deduce that the two objects are closely related. This
was also found by F{\'e}her and Rim{\'a}nyi and Weber
in~\cite{feher2018motivic} (see also~\cite{FRW:char}), using
interpolation techniques for motivic Chern classes. In cohomology, the
relation between stable envelopes and Chern-Schwartz-MacPherson
classes was noticed in~\cite{rimanyi.varchenko:csm,AMSS:shadows} and
it was used in~\cite{AMSS:shadows} to obtain a second `stable basis
duality'. In theorem~\ref{thm:dual1} we generalize this duality to
$\Kt$-theory.

\subsection{$\Kt$-theoretic stable envelopes} 
The cotangent bundle of $G/B$ is the homogeneous bundle $T^*(G/B) := G
\times^B T^*_{1.B}(G/B)$, given by equivalence classes
\[ 
\{ [g,v]: (g,v) \in G \times T^*_{1.B}(G/B) \text{ and } (gb, v)
\sim (g, b.v), \forall g \in G, b \in B \} \/;
\] 
here $T^*_{1.B}(G/B)$ is the cotangent space at the identity with its
natural $B$-module structure. As before, let $T$ be the maximal torus
in $B$, and consider the $\bbC^*$-module $\bbC$ with character~$q^{1/2}$. 
We let $\bbC^*$ act trivially on $G/B$ and we consider the $T
\times \bbC^*$ action on the cotangent bundle defined by
$(t,z).[g,v]=[tg,z^{-2}v]$. In other words, $T$ acts via its natural
left action; $\bbC^*$ acts such that the cotangent fibers get a weight
$q^{-1}$ and $K_{T \times \bbC^*}(\pt) = K_T(\pt)[q^{\pm 1/2}]$.

The stable basis is a certain basis for the localized equivariant
$\Kt$-theory 
\[
K_{T\times\bbC^*}(T^*(G/B))_{\loc}
:=K_{T\times\bbC^*}(T^*(G/B))\otimes_{K_{T\times\bbC^*}(\pt)}
\Frac(K_{T\times\bbC^*}(\pt))\/,
\] 
where $\Frac$ means taking the fraction field. The basis elements are
called the {\em stable envelopes\/} $\{\stab_{\calC, T^{\frac{1}{2}},
  \calL}(w)\,|\,w\in W\}$ and were defined by Maulik and Okounkov in the
cohomological case. We recall their definition in $\Kt$-theory below,
following mainly Okounkov's lectures~\cite{okounkov2015lectures}
and~\cite{su2017k}.

For a fixed Weyl group element, the definition of the stable envelope
$\stab_{\calC, T^{\frac{1}{2}}, \calL}(w)$ depends on three
parameters:
\begin{itemize}
\item 
a chamber $\calC$ in the Lie algebra of the maximal torus $T$, or
equivalently, a Borel subgroup of $G$.
\item 
a polarization $T^{\frac{1}{2}}\in K_{T\times\bbC^*}(T^*(G/B))$ of the
tangent bundle $T(T^*(G/B))$, i.e., 
a solution of the equation 
\[
T^{\frac{1}{2}}+q^{-1}(T^{\frac{1}{2}})^\vee=T(T^*(G/B)) 
\] 
in the ring $K_{T\times\bbC^*}(T^*(G/B))_{\loc}$.
The only polarizations utilized in this paper 
are $T(G/B)$ and $T^*(G/B)$. For every polarization
$T^{\frac{1}{2}}$, there is an opposite polarization defined as
$T^{\frac{1}{2}}_{\opp}=q^{-1}(T^{\frac{1}{2}})^\vee$.

\item 
A sufficiently general fractional equivariant line bundle on $G/B$,
i.e.~a general element $\calL \in \Pic_T(T^*(G/B))\otimes_\bbZ \bbQ$,
called the {\em slope} of the stable envelope. The dependence on the
slope parameter is locally constant, in the following sense.

The choice of a maximal torus $T \subseteq G$ determines a
decomposition of $(\Lie T)^* \otimes \bbR$ into {\em alcoves\/};
these are the complements of the affine hyperplanes $H_{\alpha^\vee, n}
= \{ \lambda \in (\Lie T)^*\otimes \bbR: \langle \lambda ,
\alpha^\vee \rangle = n \}$ as $\alpha^\vee$ varies in the set of
positive coroots, and $n$ over the integers. The alcove structure is
independent the choice of a chamber (and hence of the Borel subgroup
$B$), and the stable envelopes are constant for fractional multiples
of (pull-backs of) line bundles $\calL_\lambda= G \times^B \bbC_\lambda$
for weights $\lambda$ in a given alcove.
\end{itemize}

The torus fixed point set $(T^*(G/B))^{T}= (G/B)^T$ is in one-to-one
correspondence with the Weyl group $W$. For every $w\in W$, we still use
$e_w$ to denote the corresponding fixed point. For a chosen Weyl chamber
$\fC$ in $\Lie T$, pick any cocharacter $\sigma\in
\fC$. 
The attracting set of $w\in W$,  
also called the 
Bia\l{}ynicki-Birula cell in the literature, is defined as
\[
\Attr_\fC(w)=\left\{x\in T^*(G/B) \mid
\lim\limits_{z\rightarrow 0}\sigma(z)\cdot x=w \right\}.
\] 
It is not difficult to show that $\Attr_\fC(w)$ is the
conormal space over the attracting variety in $G/B$ for $w$; the
latter attracting variety is a Schubert cell in $G/B$.  Define a
partial order on the fixed point set $W$ to be the (transitive 
closure of the) following relation:
\[
e_w\preceq_\fC e_v \text{\quad if \quad}
\overline{\Attr_\fC(v)}\cap e_w\neq \emptyset.
\]
Then the order determined by the positive (resp., negative) chamber is
the same as the Bruhat order (resp., the opposite Bruhat order).

Any chamber $\fC$ determines a decomposition of the tangent
space $N_w:= T_w(T^*(G/B))$ as $N_w=N_{w,+}\oplus N_{w,-}$ into
$T$-weight spaces which are positive and negative with respect to
$\fC$ respectively. For every polarization $T^{\frac{1}{2}}$,
denote 
$N_w\cap T^{1/2}|_w$ by $N_w^{\frac{1}{2}}$. 
Similarly, we have
$N_{w,+}^{\frac{1}{2}}$ and $N_{w,-}^{\frac{1}{2}}$. In particular,
$N_{w,-}=N^{\frac{1}{2}}_{w,-}\oplus
q^{-1}(N_{w,+}^{\frac{1}{2}})^\vee$. Consequently, we have
\[
N_{w,-}- N_w^{\frac{1}{2}}=q^{-1}(N_{w,+}^{\frac{1}{2}})^\vee-
N_{w,+}^{\frac{1}{2}}
\] 
as virtual vector bundles. The determinant bundle of the virtual
bundle $N_{w,-}- N_w^{\frac{1}{2}}$ is a complete square and its
square root will be denoted by $\left(\frac{\det N_{w,-}}{\det
  N_w^{\frac{1}{2}}}\right)^{\frac{1}{2}}$;
cf.~\cite[\S9.1.5]{okounkov2015lectures}. For instance, if we choose
the polarization $T^{1/2} = T(G/B)$, the positive chamber, and $w=\id$
then both $N_{\id}^{\frac{1}{2}}$ and $N_{\id, -}$ have weights $-
\alpha$, where $\alpha$ varies in the set of positive roots; in this
case the virtual bundle $N_{\id,-}- N_{\id}^{\frac{1}{2}}=0$. 

Let $f:=\sum_\mu f_\mu e^\mu\in K_{T\times \bbC^*}
(\pt)$ be a Laurent polynomial, where $e^\mu\in K_T(\pt)$ 
and $f_\mu\in \bbQ[q^{1/2},q^{-1/2}]$. The {\em Newton polytope\/}
of $f$, denoted by $\deg_Tf$, is
\[
\deg_T f=\mbox{Convex hull } (\{\mu| f_\mu\neq 0\})\subseteq
X^*(T)\otimes_\bbZ \bbQ,
\] 
 where $X^*(T)$ denotes the character lattice of $T$. ~The 
following theorem defines the $\Kt$-theoretic stable envelopes.

\begin{theorem}\label{thm:geostable}~\cite{okounkov2015lectures}
For every chamber $\fC$, a sufficiently general $\calL$, and a
polarization~$T^{1/2}$, there exists a unique map of
$K_{T\times\bbC^*}(\pt)$-modules
\[
\stab_{\fC,T^{\frac{1}{2}},\calL}:K_{T\times
  \bbC^*}((T^*(G/B))^T)\rightarrow K_{T\times\bbC^*}(T^*(G/B))
\]
such that for every $w\in W$, the class
$\Gamma:=\stab_{\fC,T^{\frac{1}{2}},\calL}(w)$ satisfies:
\begin{enumerate}
\item (\textit{Support}) 
$\Supp \Gamma\subseteq \cup_{z\preceq_\fC
  w}\overline{\Attr_\fC(z)}$;
\item (\textit{Normalization}) 
$\Gamma|_w=(-1)^{\rk N_{w,+}^{\frac{1}{2}}}\left(\frac{\det
  N_{w,-}}{\det N_w^{\frac{1}{2}}}\right)^{\frac{1}{2}}
  \calO_{\Attr_\fC(w)}|_w$;
\item (\textit{Degree}) For a fixed point $e_u$, identify
$\calL|_u$ with the character of the fiber of 
$\calL$ over~$e_u$. Then for every $v\prec_\fC w$,
\[ 
\deg_T\Gamma|_v\subseteq \deg_T\stab_{\fC,T^{\frac{1}{2}},
\calL}(v)|_v+ \calL|_v-\calL|_w \/. 
\]
\end{enumerate}
\end{theorem}

The difference $\calL|_v-\calL|_w$ in the degree condition implies
that the stable basis does not depend on the choice of the
linearization of $\calL$.

Let $+$ denote the chamber such that all the roots in $B$ are positive
on it, and let $-$ denote the opposite chamber. From now on we fix the
`fundamental slope' given by $\widetilde{\calL}:= \calL_{\rho} \otimes
1/N$, where $\rho$ is the sum of fundamental weights and $N$ is a
large enough positive integer. Recall that $\omega_{G/B} :=\calL_{2
  \rho}$ is the canonical bundle of $G/B$, therefore the slope $
\widetilde{\calL}$ can also be thought as a (fractional version of a)
square root of the canonical line bundle. We will use the following
notation:
\[
\stab_+(w):=\stab_{+, T(G/B), (\widetilde{\calL})^{-1}}(w), 
\text{ and } 
\stab_-(w):=\stab_{-, T^*(G/B), \widetilde{\calL}}(w).
\] 
The positive chamber and negative chamber stable basis are dual bases
in the localized equivariant ring, i.e.,
\begin{equation}\label{equ:dualstab}
\langle\stab_+(w), \stab_-(u)\rangle_{T^*(G/B)}=\delta_{w,u},
\end{equation}
where $\langle\cdot, \cdot\rangle_{T^*(G/B)}$ is the equivariant
$\Kt$-theory pairing on $T^*(G/B)$ defined via localization;
see~\cite[Example~9.1.17]{okounkov2015lectures},~\cite[\S2.2.1,
  Proposition~1]{okounkov2016quantum},
or~\cite[Remark~2.3]{su2017k}. We will study the pairing on $K_{T
  \times \bbC^*}(T^*(G/B))$ in more detail below,
in~\S\ref{ss:stduality}.

\subsection{Automorphisms} 
The stable envelopes for various triples of parameters can be related
to each other by automorphisms of the equivariant $\Kt$-theory ring
$K_{T \times \bbC^*}(T^*(G/B))$. We will use the following types of
automorphisms:

\begin{enumerate} 
\item[a.]  
the automorphism induced by the left Weyl group multiplication. Recall
that this induces an automorphism of $K_{T \times \bbC^*}(T^*(G/B))$
which twists the coefficients in $K_{T \times \bbC^*}(\pt)$ by $w$. In
terms of localization, for every $\calF\in
K_{T\times\bbC^*}(T^*(G/B))$, we have
\begin{equation}\label{equ:weyl}
w(\calF)|_u=w(\calF|_{w^{-1}u}).
\end{equation}

\item[b.] 
The duality automorphism, mapping $[E] \mapsto [E^\vee]$, i.e., the
class of a vector bundle to its dual.  For $[\calF]\in K_T(G/B)$,
$[\calF]^\vee$ denotes the class obtained by taking the alternating
sum of duals in an equivariant resolution of $\calF$ by vector
bundles. This automorphism also acts on $K_{T \times \bbC^*}(\pt)$ by
taking $e^\lambda \mapsto e^{-\lambda}$ and $q^{\frac{1}{2}} \mapsto
q^{-\frac{1}{2}}$.
\item[c.] 
The multiplication by the class of a line bundle.  We can fix an
integral weight $\lambda \in X^*(T)$ and a Borel subgroup $B$, and
consider the equivariant line bundle $\calL_\lambda= G \times^B
\bbC_\lambda$. We will abuse notation and will denote with the same
symbol a line bundle on $G/B$ and on its cotangent bundle.
\item[d.] 
For the ring $K_{T\times
  \bbC^*}(G/B)=K_T(G/B)[q^{\frac{1}{2}},q^{-\frac{1}{2}}]$, a
composition of the previous two automorphisms gives the
(equivariant) Grothendieck-Serre duality. This is an automorphism
$\calD$ of $K_T(G/B)[q^{\frac{1}{2}},q^{-\frac{1}{2}}]$ defined as
follows: for every $[\calF]\in K_T(G/B)$,
\[
\calD[\calF]:=[RHom(\calF,\omega^\bullet_{G/B})]
:= \omega_{G/B}^\bullet \otimes [\calF]^\vee \in K_T(G/B),
\] 
where $\omega^\bullet_{G/B}\simeq \omega_{G/B}[\dim G/B]$ is the
(equivariant) dualizing complex of the flag variety; 
thus, $[\omega^\bullet_{G/B}]=(-1)^{\dim G/B}[\calL_{2 \rho}]$.
Observe that
\[
([\calF]^\vee)^\vee = [\calF]; \quad \calD ([\calF] \otimes
  \omega_{G/B}^\bullet) = [\calF]^\vee \/.
\]
Extend the operation $\calD$ to
$K_T(G/B)[q^{\frac{1}{2}},q^{-\frac{1}{2}}]$ by sending
$q^{\frac{1}{2}} \mapsto q^{-\frac{1}{2}} $.
\end{enumerate}
The following lemma, proved in the appendix, records the effect of
these automorphisms on $\Kt$-theoretic stable envelopes.

\begin{lemma}\label{lemma:autos} 
(a) Let $u,w \in W$. Under the left Weyl group multiplication, 
\[ 
w. \stab_{\fC, T^{1/2}, \calL}(u) = \stab_{w \fC, w T^{1/2},
  w.\calL}(wu) \/. 
\] 
In particular, if both the polarization $T^{1/2}$ and the line bundle
$\calL$ are $G$-equivariant, then
\[ 
w. \stab_{\fC, T^{1/2}, \calL}(u) = \stab_{w \fC, T^{1/2},
  \calL}(wu) \/. 
\]

(b) The duality automorphism acts by sending {$q^{\frac{1}{2}} 
\mapsto q^{-\frac{1}{2}} $} and
\begin{equation}\label{equ:dual}
(\stab_{\fC,T^{\frac{1}{2}}, \calL}(w))^\vee=q^{-\frac{\dim
      G/B}{2}}\stab_{\fC,T^{\frac{1}{2}}_{\opp},
    \calL^{-1}}(w),
\end{equation}
where $T^{\frac{1}{2}}_{\opp}:=q^{-1}(T^{\frac{1}{2}})^\vee$
is the opposite polarization;
see~\cite[Equation~(15)]{okounkov2016quantum}, i.e., this duality
changes the polarization and slope parameters to the opposite ones,
while keeping the chamber parameter invariant.

(c) Let $\calL,
\calL'\in \Pic_T(T^*(G/B))$ be any equivariant line
bundles, let $w \in W$ and $a \in \bbQ$ a rational number. Then
\[ 
\stab_{\fC, T^{1/2},a \calL\otimes \calL'}(w) =  (\calL'|_w)^{-1}
\calL' \otimes \stab_{\fC, T^{1/2},a \calL}(w) \/, 
\] 
as elements in $K_{T\times\bbC^*}(T^*(G/B))_{\loc}$.
\end{lemma}

\subsection{Recursions for stable envelopes} 
Because the $\bbC^*$-fixed locus of the cotangent bundle is the zero
section (i.e.~$G/B$), it follows that the torus fixed point locus
$(T^*(G/B))^{T \times \bbC^*}$ coincides with the fixed locus $(G/B)^T$,
a discrete set indexed by the Weyl group $W$. Therefore the
equivariant $\Kt$-theory classes associated to fixed points form a
basis in the localized ring $K_{T \times \bbC^*}(T^*(G/B))_{\loc}$. In
order to compare motivic Chern classes to stable envelopes, we need
the following result proved in~\cite[Proposition~4.6]{su2017k}.

\begin{prop}\label{prop:charnegstable}
The restriction coefficients $\stab_-(w)|_u$ are uniquely
characterized by
\begin{enumerate}
\item 
$\stab_-(w)|_u=0$, unless $u\geq w$.
\item 
$\stab_-(w)|_w=q^{\frac{\ell(w)}{2}}\prod_{\alpha>0,w\alpha<0}
(1-e^{-w\alpha})\prod_{\alpha>0,w\alpha>0}(1-qe^{-w\alpha})$.
\item 
If $ws_{i}>w$, then
\[
q^{\frac{1}{2}}\stab_-(w)|_{u}=\frac{(1-q)}{1-e^{-u\alpha_i}}
\stab_-(ws_{i})|_u+\frac{1-qe^{-u\alpha_i}}{1-e^{u\alpha_i}}
\stab_-(ws_{i})|_{us_{i}}.
\]
\end{enumerate}
\end{prop}

Applying parts (a) and (b) from Lemma~\ref{lemma:autos}, and from the
definitions of the stable envelopes, we obtain that for every $u \in W$,
\begin{equation}\label{E:stab+-} 
w_0 . (\stab_- (u))^\vee = q^{- \frac{\dim G/B}{2}} \stab_+(w_0 u) \/. 
\end{equation} 

Then we immediately obtain the following analogue of
Proposition~\ref{prop:charnegstable}:

\begin{prop}[\cite{su2017k}]\label{prop:charstab}
The localizations $\stab_+(w)|_u$ are uniquely characterized by the
following properties:
\begin{enumerate}
\item
$\stab_+(w)|_u=0$, unless $u\leq w$.
\item
$\stab_+(w)|_w=q^{\frac{\ell(w)}{2}}
\prod_{\alpha>0,w\alpha<0}(1-q^{-1}e^{w\alpha})
\prod_{\alpha>0,w\alpha>0}(1-e^{w\alpha})$.
\item
If $ws_{i}>w$, then
\[
q^{\frac{1}{2}}\stab_+(ws_{i})|_u=\frac{q-1}{1-e^{u\alpha_i}}
\stab_+(w)|_u-\frac{e^{u\alpha_i}-q}{1-e^{-u\alpha_i}}
\stab_+(w)|_{us_{i}}.
\]
\end{enumerate}
\end{prop}

\subsection{Motivic classes are pull-backs of stable envelopes} 
One of the key formulas in~\cite{AMSS:shadows} shows that the dual CSM
class equals the Segre-Schwartz-MacPherson (SSM) class, up to a
normalization coefficient. The proof of that identity is based on a
transversality argument, which can be expressed either in terms of
(cohomological) stable basis elements or in terms of transversality of
characteristic cycles. The same phenomenon occurs in $\Kt$-theory.
Let $i:G/B\hookrightarrow T^*(G/B)$ be the inclusion of the zero
section into the cotangent bundle. Define
\[
\stab_+'(w):=\calD(i^*\stab_+(w))=(-1)^{\dim
  G/B}(i^*\stab_+(w))^\vee\otimes [\calL_{2\rho}]\in
K_T(G/B)[q^{\frac{1}{2}} ,q^{-\frac{1}{2}} ] \/.
\]
The following result relates motivic Chern classes and stable
envelopes and it is the $\Kt$-theoretic analogue of the cohomological
results from~\cite[Corollary~6.6]{AMSS:shadows}
and~\cite{rimanyi.varchenko:csm}. It is also equivalent to results
from~\cite{feher2018motivic}, where it is shown that motivic Chern
classes satisfy the same localization properties as the stable
envelopes for a certain triple of parameters; cf.~Remark~\ref{rmk:FRW}
below.

\begin{theorem}\label{thm:poshriek} 
For every $w\in W$, we have
\[
q^{-\frac{\ell(w)}{2}}\stab_+'(w)=\MC_{-q^{-1}}(X(w)^\circ)\in
K_T(G/B)[q,q^{-1}].
\]
\end{theorem}

\begin{proof} 
We compare localization properties of the motivic Chern classes with
those for the Grothendieck-Serre dual of $\stab_+(w)$. We have that
\begin{equation}\label{equ:dualloc}
\stab'_+(w)|_u=(-1)^{\dim
  G/B}e^{2u\rho}(\stab_+(w)|_u)|_{e^\lambda\rightarrow e^{-\lambda},
  {q^{\frac{1}{2}} \rightarrow q^{-\frac{1}{2}} } }.
\end{equation} 
Then the corresponding result from Proposition~\ref{prop:charstab} for
$\stab_+'(w)$ is that the localizations $\stab'_+(w)|_u$ are uniquely
characterized by the following properties
\begin{enumerate}
\item
$\stab_+'(w)|_u=0$, unless $u\leq w$.
\item
$\stab_+'(w)|_w=q^{\frac{\ell(w)}{2}}\prod_{\alpha>0,w\alpha<0}
(1-q^{-1}e^{w\alpha})\prod_{\alpha>0,w\alpha>0}(1-e^{w\alpha})$.
\item
If $ws_{\alpha_i}>w$, then
\[
q^{-\frac{1}{2}}\stab_+'(ws_{\alpha_i})|_u=
\frac{q^{-1}-1}{1-e^{-u\alpha_i}}\stab_+'(w)|_u
+\frac{1-q^{-1}e^{u\alpha_i}}{1-e^{-u\alpha_i}}
\stab_+'(w)|_{us_{\alpha_i}}.
\]
\end{enumerate} 
Comparing this with localizations of motivic Chern classes from
Proposition~\ref{prop:charmot} finishes the proof.
\end{proof}

A similar statement relates $\stab_{-}(w)$ to the motivic Chern class
of the opposite Schubert cells. We record the statement next; the
proof is essentially the same, and details are left to the reader.
Define
\[
\stab_-'(w):=q^{-\dim G/B}i^*(\stab_-(w))\otimes 
[\omega^\bullet_{G/B}] \in K_T(G/B)[q^{\frac{1}{2}},
q^{-\frac{1}{2}}]\/.
\]

\begin{theorem}\label{thm:negshriek}
For every $w\in W$,
\[
q^{\frac{\ell(w)}{2}}\stab_-'(w) =\MC_{-q^{-1}}(Y(w)^\circ)\in
K_T(G/B)[q,q^{-1}].
\]
\end{theorem}

\begin{remark}\label{rmk:FRW} 
Using Theorem~\ref{thm:poshriek} and Lemma~\ref{lemma:autos}, one can
show that for every $w\in W$,
\begin{equation}\label{E:FRW} 
q^{\frac{-\ell(w)}{2}}i^*\stab_{+,T(G/B),\widetilde{\calL}}(w)
=\MC_{-q^{-1}}(X(w)^\circ),
\end{equation} 
where $\widetilde{\calL}$ is the fundamental slope. This is consistent
with the choices of parameters for stable envelopes
from~\cite{feher2018motivic}. In fact, a direct check of the
normalization and degree conditions shows that for any slope
$\mathcal{L}$,
\[ 
\mathcal{D} (i^*( \stab_{+,T(G/B),\mathcal{L}})) =
i^*(\stab_{+,T(G/B),(\mathcal{L})^{-1}}) \/.
\]
Again we leave the proof details to the reader.  
\qede\end{remark}

\begin{remark} 
In upcoming work we will prove that
\[ 
(-q)^{- \dim G/B} i^*( {\gr(i_{w!}\bbQ_{Y(w)^\circ}^H)) \otimes
  [\omega^\bullet_{G/B}]} = \MC_{-q^{-1}}(Y(w)^\circ) \/, 
\] 
where $i_w: Y(w)^\circ \to G/B$ is the inclusion, and {$\gr(i_{w!}
  \bbQ_{Y(w)^\circ}^H)$ is the associated graded (or
  $\bbC^*$-equivariant)} sheaf on $T^*(G/B)$ determined by the shifted
mixed Hodge module $\bbQ_{Y(w)^\circ}^H$; see~\cite{tanisaki:hodge}
(with our $q ^{-1}$
  corresponding to the parameter $q$ in Tanisaki's paper). Since
$i^*$ is an isomorphism, we deduce from Theorem~\ref{thm:negshriek}
that:
\begin{equation}\label{E:stabvsMC} 
\stab_{-}(w) = (-1)^{\dim G/B} \gr(i_{w!} \bbQ_{Y(w)^\circ}^H) \/. 
\end{equation} 
Since the cycle associated to the coherent sheaf $\gr((i_w)_!
\bbQ_{Y(w)^\circ})$ is the characteristic cycle of the constructible
function $\one_{Y(w)^\circ}$, this equation can be seen as the
$\Kt$-theoretic generalization of the equivalence between
(cohomological) stable envelopes and characteristic cycles, indicated
by Maulik and Okounkov~\cite[Remark 3.5.3]{maulik.okounkov:quantum};
see also~\cite[Lemma 6.5]{AMSS:shadows} for a proof.
\qede\end{remark}

\subsection{Stable basis duality}\label{ss:stduality} 
As for the CSM classes, there are two sources for Poincar{\'e} type
dualities of the motivic Chern classes. The first is a consequence of
the existence of two adjoint Demazure-Lusztig operators. The second,
which has a geometric origin, uses the duality from
\eqref{equ:dualstab} for the stable envelopes, on the cotangent
bundle. Given that the localization pairing on the cotangent bundle
can also be expressed in terms of a twisted Poincar{\'e} pairing on
the zero section, this leads to some remarkable identities among
motivic Chern classes. Recall from the equation \eqref{equ:dualstab}
that the `opposite' stable envelopes are dual to each other with
respect to the $\Kt$-theoretic pairing on $T^*(G/B)$, defined as
follows: for every $\calF,\calG\in K_{T\times \bbC^*}(T^*(G/B))$,
\[
\langle\calF,\calG\rangle_{T^*(G/B)}:=\sum_{w\in
  W}\frac{[\calF]|_w\cdot[\calG]|_w}{\prod_{\alpha>0}
(1-e^{w\alpha})(1-qe^{-w\alpha})}.
\]
Recall that $i:G/B\hookrightarrow T^*(G/B)$ is the inclusion of the
zero section. By localization, the pairing in $T^*(G/B)$ is related to
the ordinary 
pairing in the equivariant $\Kt$-theory 
of~$G/B$:
\begin{equation}\label{equ:twopairings}
\langle \calF,\calG\rangle_{T^*(G/B)}
=\left\langle i^*\calF,\frac{i^*\calG}{\lambda_{-q}(T(G/B))}\right\rangle.
\end{equation}
Here the ordinary pairing in the equivariant $\Kt$-theory of
$G/B$ is extended (by the equivariant projection formula) bilinearly
over $\Frac( K_{T\times \bbC^*}(\pt))$ to a pairing:
\[
\langle -,- \rangle : K_{T\times \bbC^*}(G/B)_{\loc}\times K_{T\times 
\bbC^*}(G/B)_{\loc}\to \Frac( K_{T\times \bbC^*}(\pt))\:.
\]
Moreover, $\lambda_{-q}(T(G/B))\in K_{T\times \bbC^*}(G/B)
\hookrightarrow K_{T\times \bbC^*}(G/B)_{\loc}$ is invertible in the 
localized ring by the following observation.

\begin{remark}\label{rmk:lambdayprod} 
A $\Kt$-theoretic analogue of~\cite[Lemma 8.1]{AMSS:shadows} gives 
\[
\lambda_{-q}(T^*(G/B))\lambda_{-q}(T(G/B))=
\prod_{\alpha>0}(1-q e^\alpha)(1-q e^{-\alpha} )\/.
\]
As in {\it loc.~cit.,} this follows by localization, because for all $w
\in W$,
\[
\lambda_{-q}(T^*(G/B))|_w \cdot  \lambda_{-q}(T(G/B))|_w=
\prod_{\alpha>0}(1-q e^\alpha)(1-q e^{-\alpha} )
\]
as $w$ permutes the set of roots.
\qede\end{remark}

We need the following lemma.

\begin{lemma}\label{lem:dualpair}
Let $[\calF], [\calG]\in K_{T\times \bbC^*}(G/B)_{\loc}$ such
that
\[
\langle [\calF], [\calG] \rangle=f(e^t,q^{\frac{1}{2}})\in K_{T\times
  \bbC^*}(\pt)_{\loc} \/.
\] 
Then
\[ 
\langle \calD([\calF]), [\calG]^\vee \rangle = \langle [\calF]^\vee,
\calD([\calG]) \rangle = f(e^{-t},q^{-\frac{1}{2}}) = \left( \langle
    [\calF], [\calG] \rangle \right)^\vee \/,
\] 
i.e., all weights are inverted by this operation.
\end{lemma}

\begin{proof}  
By the definition of the equivariant Grothendieck-Serre duality
  operator, it suffices to prove the equality
\[
\langle [\calF]^\vee, \calD([\calG]) \rangle = f(e^{-t},q^{-\frac{1}{2}})\:.
\]
Applying the Lefschetz fixed point formula in equivariant
$\Kt$-theory~\cite[\S5.10]
{chriss2009representation} we obtain
\[
f(e^t,q^{\frac{1}{2}})=
\langle [\calF], [\calG] \rangle=\sum_w\frac{[\calF]|_w\cdot
  [\calG]|_w}{\prod_{\alpha>0}(1-e^{w\alpha})}\/.
\]
Recall that $[\omega_{G/B}^\bullet] = 
(-1)^{\dim G/B} [ \calL_{2\rho}]$. Then
\begin{align*}
\langle \calD([\calF]), [\calG]^\vee \rangle 
&= (-1)^{\dim G/B} \langle [\calF]^\vee, [\calG]^\vee\otimes 
[\calL_{2\rho}]\rangle \\
&= (-1)^{\dim G/B} \sum_w\frac{([\calF]|_w)^\vee([\calG]|_w)^\vee}
{\prod_{\alpha>0}
(1-e^{w\alpha})}e^{2w\rho}\\ 
&=\sum_w\frac{([\calF]|_w)^\vee([\calG]|_w)^\vee}
{\prod_{\alpha>0}(1-e^{-w\alpha})}\\ 
&=f(e^{-t} , q^{-\frac{1}{2}}) \/.
\end{align*}
The second-to-last equality holds because $2 \rho = \sum_{\alpha >0}
\alpha$, thus $e^{2 w (\rho)} = \prod_{\alpha > 0} e^{w(\alpha)}$, and
the last equality follows since the effect of taking $(-)^\vee$ is to
invert the $T$ and $\bbC^*$ weights.
\end{proof}

\begin{theorem}\label{thm:dual1} 
Let $u,w \in W$ and $y = -q^{-1}$. Then the following orthogonality
relation holds:
\[
\left\langle \MC_y(X(w)^\circ), \frac{\calD(\MC_y(Y(u)^\circ))}
{\lambda_y(T^*(G/B))}(-y)^{\dim G/B-\ell(u)}\right\rangle=\delta_{w,u}.
\] 
Equivalently,
\[ 
\MC_y^\vee(Y(u)^\circ) = \prod_{\alpha>0}(1+ye^{-\alpha})
\frac{\calD(\MC_y(Y(u)^\circ))}{\lambda_y(T^*(G/B))} \: 
{\in K_{T\times \bbC^*}(G/B)_{\loc} }\/. 
\]
\end{theorem}

\begin{proof} 
The idea is to use Theorem~\ref{thm:poshriek} to express
$\MC_y(X(u)^\circ)$ in terms of the Grothendieck-Serre dual, then use
Lemma~\ref{lem:dualpair} to relate the pairing in the statement of the
theorem to the pairing between orthogonal stable envelopes. We start
by observing that
\[
\lambda_y(T^*(G/B)) = (\lambda_{y^{-1}} T(G/B))^\vee
\in K_T(G/B)[y,y^{-1}]\:,
\]
 and that by Theorem~\ref{thm:negshriek} 
\begin{align*}
\calD(\MC_y(Y(u)^\circ)) &= \calD( (-y)^{ - \frac{
    \ell(u)}{2}}(-y)^{\dim G/B} \iota^*(\stab_{-}(u)) \otimes
     [\omega_{G/B}^\bullet] ) \\ 
&= (-y)^{\frac{\ell(u)}{2} - \dim G/B} (\iota^*(\stab_{-}(u)))^\vee
     \/.
\end{align*}   
{}From this, the second term of the pairing equals 
\[  
\frac{\calD(\MC_y(Y(u)^\circ))}{\lambda_y(T^*(G/B))} (-y)^{\dim
  G/B-\ell(u)}= \Bigl(\frac{\iota^*(\stab_{-}(u))}{\lambda_{y^{-1}}
  (T(G/B)) } (-y)^{\frac{\ell(u)}{2}} \Bigr)^\vee \/. 
\] 
Then Theorem~\ref{thm:poshriek}, Lemma~\ref{lem:dualpair}, and
orthogonality of stable envelopes (Equation~\eqref{equ:dualstab})
imply that
\[ 
\begin{split} 
&\quad \left\langle \MC_y(X(w)^\circ), \frac{\calD(\MC_y(Y(u)^\circ))}
{\lambda_y(T^*(G/B))}(-y)^{\dim G/B-\ell(u)}\right\rangle \\ 
&= \left\langle \calD ( (-y)^{-\frac{\ell(w)}{2}}\iota^* \stab_+(w)), 
\Bigl(\frac{\iota^*(\stab_{-}(u))}{\lambda_{y^{-1}} (T(G/B)) } 
(-y)^{\frac{\ell(u)}{2}} \Bigr)^\vee \right\rangle \\ 
&=\left\langle (-y)^{\frac{\ell(u)-\ell(w)}{2}}\iota^* \stab_+(w), 
\frac{\iota^* \stab_-(u)}{\lambda_{y^{-1}} (T(G/B)) } 
\right\rangle_{y \mapsto y^{-1}, e^t \mapsto e^{-t}} \\
&=(-y)^{\frac{\ell(w)-\ell(u)}{2}}\langle \stab_+(w), \stab_-(u)
  \rangle_{T^*(G/B)}|_{y \mapsto y^{-1}, e^t \mapsto e^{-t}} \\
&= \delta_{w,u} \/,
\end{split}
\]
where the second equality follows from Lemma~\ref{lem:dualpair}, the
third equality follows from Equation~\eqref{equ:twopairings} and the
last one follows from Equation~\eqref{equ:dualstab}. This proves the
first assertion.  The second assertion follows from the `Hecke
orthogonality' of motivic Chern classes, proved in
Theorem~\ref{thm:dualbasis}.
\end{proof}
 \begin{remark} The proof of the previous theorem depends 
in an essential way on the orthogonality of stable envelopes. This 
dependence can be removed by proving the transversality 
formula mentioned in Remark~\ref{rmk:SMC}. This approach, based on the 
transversality formula from~\cite{schurmann:transversality},
was utilized in \cite{AMSS:shadows} to prove the cohomological case of 
Theorem \ref{thm:dual1}.
\qede\end{remark}

Theorem~\ref{thm:dual1} justifies the definition of a dual motivic
Chern class of a Schubert variety:

\begin{defin}\label{def:dualmotvariety} 
Let $w \in W$. Define the dual motivic Chern class of a dual Schubert
{\em variety\/} by
\[ 
\MC_y^\vee(Y(w)):= \sum_{u \ge w} \MC_y^\vee(Y(u)^\circ) \/.
\]
\end{defin}

By Theorem~\ref{thm:dual1},
\[ 
\MC_y^\vee(Y(w)) = \frac{\prod_{\alpha>0} (1 + y e^{-\alpha})}
{\lambda_y(T^*(G/B))} \calD(\MC_y (Y(w))) \: 
\in K_{T\times \bbC^*}(G/B)_{\loc}\/.
\] 
The class
\[ 
\frac{\calD(\MC_y (Y(w)))}{\lambda_y(T^*(G/B))} \: 
{\in K_{T\times \bbC^*}(G/B)_{\loc} }
\] 
can be thought as a {\em motivic Segre class}, i.e.,~a $\Kt$-theoretic
analogue of the Segre-Schwartz-MacPherson class discussed
in~\cite{aluffi:inclusionI,ohmoto:eqcsm,AMSS:shadows,fejer.rimanyi:CSMdeg,  mihalcea.naruse.su}. More precisely, it will be the motivic Segre
class of the {\em motivic dual}\/ $\calD_{\mot}([Y(w)])$ of the dual
Schubert variety $Y(w)$, for an equivariant motivic duality {$\calD_{\mot}$}
(extending~\cite{bittner:universal}),
with
\begin{equation} \label{MC-duality}
\calD(\MC_y (Y(w)))  = \MC_y(\calD_{\mot}([Y(w)]))
\end{equation}
as an equivariant extension 
of~\cite[{Corollary~5.19}]{schurmann2009characteristic}.

\section{Smoothness of Schubert varieties and localizations of motivic Chern classes}\label{s:smoothness} 
Among the main applications of this paper are properties about the
transition matrix between the standard and the Casselman's basis for
Chevalley groups over nonarchimedean local fields. The matrix
coefficients are rational functions, and of particular interest to us
are certain factorizations and polynomial properties of these
coefficients conjectured by Bump, Nakasuji, and Naruse; see section
\S\ref{sec:padic} below and
\cite{BN11,nakasuji2015yang,bump2017casselman}. We will prove 
in~\S\ref{sec:padic} that the transition matrix from the `Casselman
setting' corresponds to transition matrix between (dual) motivic Chern
classes of Schubert varieties and an appropriate normalization of the
fixed point basis. This motivates the study in this section of the
underlying `geometric' transition matrix between the motivic classes
and fixed point basis. The main result of this section is
Theorem~\ref{thm:geomrefinedconj}
(Theorem~\ref{thm:introgeomrefinedconj} from the introduction); a 
representation-theoretic version of this theorem 
will be proved in Theorem~\ref{thm:refinedconj} below.
\subsection{A smoothness criterion} 
In this section we prove a criterion for the smoothness of Schubert
varieties in terms of the motivic Chern classes. We need the following 
lemma, which is implicit in \cite{feher2018motivic}.
\begin{lemma}\label{lemma:loc} 
(a) Let $i:X \subseteq M $ be a 
closed embedding
of $G$-equivariant,
non-singular, quasi-projective, algebraic varieties, with 
$N^\vee_{X} M$ the conormal bundle of $X$ inside $M$.
Then: 
\[ i^* \MC_y[X \to M] = \lambda_y(T^*_X) \otimes \lambda_{-1}(N^\vee_{X} M) \/. \]
(b) Let $X \subseteq M$ be a closed embedding
of $T$-equivariant,
algebraic varieties, and assume that $M$ is smooth. Let $p \in X$ be a
smooth $T$-fixed point, and $j:V \subseteq M$ any $T$-invariant open set such
that $p \in V$ and $X':=V \cap X$ is smooth (e.g., $V:=M\smallsetminus
  X_{\mathrm{sing}}$). Let $\iota_p: \{ p \} \to M$ be the inclusion. Then: 
\[
\iota_p^* \MC_y[X \to M] = \lambda_y (T^*_p X) \cdot
\lambda_{-1}((N^\vee_X M)_p) \/.
\]
\end{lemma}
\begin{proof} 
Part (a) follows from the functoriality of motivic Chern classes and the self-intersection formula in $\Kt$-theory
\cite[Proposition~5.4.10]{chriss2009representation}: 
\[ 
i^* \MC_y[X \to M] = i^* i_* \MC_y [\id_X] = \MC_y[\id_X] \otimes
\lambda_{-1}(N^\vee_{X} M) = \lambda_y(T^*_X) \otimes
\lambda_{-1}(N^\vee_{X} M) \/.\]
Now let us prove (b).
Let $\iota_p': \{p\} \to V$ denote the embedding. Note that 
\[ 
\iota_p^* \MC_y[X \to M] = (\iota_p')^* j^* \MC_y[X \to M] =
(\iota_p')^* \MC_y ( j^* [X \to M] ) = (\iota_p')^* \MC_y[ X' \to V] \/,
\] 
where the second equality follows from the Verdier-Riemann-Roch
formula from Theorem~\ref{thm:existence}, as $j$ is an open embedding
(thus a smooth morphism), with relative tangent bundle equal to $1$. 
Applying Part (a) to the closed
embedding $X' \hookrightarrow V$,
we obtain:
\[ 
\MC_y[ X' \to V] = \lambda_y(T^*_{X'}) \otimes \lambda_{-1}(N^\vee_{X'}
V) \/.
\] 
The claim follows by pulling back via $(\iota'_p)^*$, using that
$(\iota'_p)^*$ is a ring homomorphism in (equivariant) $\Kt$-theory.
\end{proof}

We also need a variant of Kumar's cohomological criterion for smoothness
of Schubert varieties.

\begin{theorem}[\cite{kumar1996nil}]\label{thm:kumar} 
Let $u,w$ be two Weyl group elements such that $u \le w$. Then the
Schubert variety $Y(u)$ is smooth at $e_w$ if and only if the
localization of the equivariant fundamental class $ [Y(u)]\in
A^T_*(G/B)$ in the equivariant Chow group is given by:
\[ 
[Y(u)|_{w}= \left( \prod_{\beta>0, u\nleq s_\beta w}\beta\right) \/
  \in A_*^T(\pt)=\bbZ[\alpha_i\:|\: i=1,\dots, r]\: .
\]
If $Y(u)$ is smooth at $e_w$, then the torus weights of $T_{w}Y(u)$ are
$\{-w\alpha|\alpha>0,ws_\alpha\geq u\}$.
\end{theorem}
The statement above is an equivalent, but different, 
formulation from that in~\cite{kumar1996nil}. We briefly indicate next the 
steps needed to bring it into the original formulation.
Consider the automorphism of the set $R^+$ of positive roots given by
$\alpha \mapsto -w_0(\alpha)$. One checks that this is actually an
automorphism of the Dynkin diagram. It induces the automorphism $w
\mapsto w_0 w w_0$ of the Weyl group $W$, preserving the length and
the Bruhat order. It also induces an automorphism of $G/B$ sending the
Schubert cell $Y(w)^\circ$ to $Y(w_0 w w_0)^\circ$. In particular,
$Y(u)$ is smooth at $e_w$ if and only if $Y(w_0 u w_0)$ is smooth at
$w_0 w w_0$. Then
\begin{align*}
&\,Y(w_0uw_0) \text{ is smooth at } e_{w_0ww_0} \\
\Longleftrightarrow &X(uw_0) \text{ is smooth at } e_{ww_0} \\
\Longleftrightarrow &[Y(u)]|_{w} =\prod_{\beta>0, u\nleq s_\beta w}\beta.
\end{align*}
Here in the last equivalence, we used the original version of
Kumar's criterion, as stated in
\cite[Corollary~7.2.8]{billey2000singular}. (Notice that the term $d_{w,u}$ in {\it loc.~cit.} is equal to the localization $[Y(u)]|_w$, see \cite[Theorem~7.2.11]{billey2000singular} and \cite{billey:kostant}.)

\begin{remark}\label{rmk:weights} Let $S'(u,w):=\{\alpha > 0: u\leq w s_\alpha< w\}$.~An immediate
consequence of the theorem is that if $Y(u)$ is smooth at $e_w$, then
the weights of the normal space 
$(N_{Y(w)} Y(u))_w
= T_{w} (Y(u))/T_w
(Y(w))$ of $Y(w)$ at $e_w$ in $Y(u)$ are
\begin{equation}\label{E:defS} 
S(u,w):= \{\beta >0: u \leq s_\beta w < w \} =\{-w(\alpha): \alpha 
\in S'(u,w) \} \/.
\end{equation} 
\end{remark}

The main theorem of this section is
the following \begin{footnote}{We are thankful to a referee for suggesting the current formulation.}\end{footnote}:
\begin{theorem}\label{thm:MCsmooth} 
Let $u,w\in W$ such that $u\leq w$. The following are equivalent:\\
(a) The opposite Schubert variety
$Y(u)$ is smooth at the torus fixed point $e_w$.\\
(b) The localization of the motivic Chern class $\MC_y(Y(u))$ at $w$ is given by 
\begin{equation}\label{eq:smcr}
\MC_y(Y(u))|_w=\prod_{\alpha>0,ws_\alpha\geq u}(1+ye^{w\alpha})
\prod_{\alpha>0,u\nleq ws_\alpha}(1-e^{w\alpha}).
\end{equation}\\
(c) The localization of the structure sheaf $ \calO^u\in
K_T(G/B)$ is given by:
\[\calO^u|_{w}= \prod_{\alpha>0,u\nleq ws_\alpha}(1-e^{w\alpha}).\]
(d) The localization of the equivariant fundamental class $ [Y(u)]\in
A^T_*(G/B)$ in the equivariant Chow group is given by:
\[[Y(u)]|_{w}= \prod_{\alpha>0,u\nleq ws_\alpha}(-w\alpha)=
\prod_{\beta>0, u\nleq s_\beta w}\beta.\]
\end{theorem}
\begin{proof}
By Lemma \ref{lemma:loc}(b) and the weight space description from 
Remark~\ref{rmk:weights},
(a) implies~(b). From the normalization property, the specialization
 $\MC_{y=0}(Y)=[\calO_Y]\in
\Kt_T(Y)$ if $Y$ is smooth, and it follows from functoriality that
$\MC_{y=0}(Y)=[\calO_Y]$
if $Y$ has rational singularities. This is the case
for Schubert varieties (\cite[Theorem~2.2.3]{brion:flagv}), so
$\MC_{y=0}(Y(u))=\calO^u$.
Hence, (c) follows from (b) by setting $y=0$.  Consider the `geometric'
equivariant Chern character
\[
\ch_T:K_T(G/B) \to \widehat{A}_*^T(G/B)
\]
to an appropriate completion of the equivariant Chow group; see
\cite{edidin2000}. From the definition of $\ch_T$ and by
\cite[Theorem~18.3]{fulton:IT} it follows that the top {homological}
degree term of the equivariant Chern character $\ch_T(\calO^u)$ is the
equivariant fundamental class $[Y(u)]_T$. Together with the fact that
$\ch_T(e^{\lambda}) = 1+ \lambda + $ higher degree {\em cohomological}
terms, this implies that if (c) holds, then the top degree
term of $\ch_T(\calO^u)$ localizes to
\[
[Y(u)]|_w= \prod_{\alpha>0, u \nleq ws_\alpha} (-w(\alpha))=
\prod_{\beta>0, u\nleq s_\beta w}\beta.
\] 
Here the second equality uses the change of variable $\beta=-w(\alpha)$ and the fact that $w(\alpha)<0$ since $ws_\alpha <w$. Thus, (d) holds.
Finally, (d) implies (a) by Kumar's Theorem \ref{thm:kumar}.
\end{proof}

\subsection{The geometric Bump-Nakasuji-Naruse 
conjecture}\label{sec:geomBNNconj} 

Motivated by the applications to
representation theory from~\S\ref{sec:padic}, we study the
following problem. Define the element $b_w \in
K_{T}(G/B)_{\loc}[y^{-1}]$ by the formula
\begin{equation}\label{E:defbw} 
b_w:=(-1)^{\dim G/B - \ell(w)}\prod_{\alpha>0, w\alpha>0}\frac{y^{-1}
  +e^{-w\alpha}}{1-e^{w\alpha}}\iota_{w} \/.
\end{equation}
Equivalently, $b_w$ is the multiple of the fixed point basis element
$\iota_w$ which satisfies $(b_w)|_w = \MC_y^\vee(Y(w)^\circ)|_w$.
Consider the expansion of the class $\MC_y^\vee(Y(u))$ from Definition
\ref{def:dualmotvariety}:
\begin{equation}\label{E:MCdualexp} 
\MC_y^\vee(Y(u)) = \sum m_{u,w} b_w \quad \in
K_T(G/B)_{\loc}[y^{-1}] \/\:.
\end{equation}
It is easy to see that
$m_{u,w}=0$ unless $u \leq w $. For pairs
$u\leq w\in W$, recall that $S(u,w):=\{\beta \in R^{+}\,|\,u\leq s_\beta w
< w\}$.  The main result of this section is the following geometric
analogue of
a conjecture of Bump, Nakasuji, and Naruse
\cite{BN11,nakasuji2015yang,bump2017casselman}.

\begin{theorem}[Geometric Bump-Nakasuji-Naruse Conjecture]
\label{thm:geomrefinedconj}
For every $u\leq w\in W$, 
\begin{equation}\label{equ:geommwz}
m_{u,w}=\prod_{\alpha\in S(u,w)}\frac{1+y^{-1}e^\alpha}{1-e^\alpha}
\end{equation}
if and only if the Schubert variety $Y(u)$ is smooth at the torus
fixed point $e_{w}$.
\end{theorem} 

The $p$-adic representation theory counterpart of this theorem will be
given in Theorem~\ref{thm:refinedconj} below.
The coefficients $m_{u,w}$ calculate
the transition matrix between the `standard basis' and `Casselman's
basis' for the Iwahori invariants of the principal series
representation. The statement is a generalization of the original
Bump-Nakasuji conjecture, communicated to us by H.~Naruse;
see~\cite{naruse2014schubert}. In fact, Naruse informed us that he
obtained the implication of this theorem which assumes the
factorization. Naruse's proof of this implication, and ours, are both
based on Kumar's cohomological criterion for smoothness
(\cite{kumar1996nil}; see Theorem~\ref{thm:kumar}).  Naruse's proof is
based on Hecke algebra calculations, while ours uses motivic Chern
classes.

After harmonizing conventions between this paper
and~\cite{BN11,bump2017casselman}, and passing to the `geometric'
version, the original conjecture states the following (see
\cite[Conjecture~1.2]{BN11} and \cite[p.~3]{bump2017casselman}):

\begin{corol}\label{cor:BNconj} 
Let $G$ be a complex, simply laced, reductive, linear algebraic group. 
Then the coefficient $m_{u,v}$ satisfies the factorization in
\eqref{equ:geommwz} if and only if 
\[ 
P_{w_0 w^{-1}, w_0 u^{-1}} = 1 \/, 
\]
where $P_{w_0 w^{-1}, w_0 u^{-1}}$ denotes the Kazhdan-Lusztig
polynomial.
\end{corol}

We first prove this statement, assuming
Theorem~\ref{thm:geomrefinedconj}.

\begin{proof} 
Since the group $G$ is simply laced, an unpublished result of D.
Peterson, re-proved by Carrell and Kuttler (see
e.g.,~\cite{carrell.kuttler:smooth} or
\cite[Theorem~6.0.4]{billey2000singular}) shows that the condition
that $Y(u)$ is smooth at $e_w$ is equivalent to $Y(u)$ being
rationally smooth at $e_w$. For arbitrary $G$, rational smoothness is
equivalent to the fact that the Kazhdan-Lusztig polynomial $P_{w_0 u,
  w_0 w}$ equals $1$, by a theorem Kazhdan and Lusztig
\cite[Theorem~A2]{KL:representations}. By Theorem
\ref{thm:geomrefinedconj}, it remains to shows that~$P_{w_0 w^{-1},
  w_0 u^{-1}} = 1$ if and only if $P_{w_0 w, w_0 u} = 1$. In turn,
this is equivalent to
\[ 
P_{u, w} = 1 \Longleftrightarrow P_{w_0 u^{-1}w_0, w_0 w^{-1} w_0} = 1
\/. 
\] 
This is proved in the next lemma below.
\end{proof} 

\begin{lemma}\label{lemma:KLequiv} 
Let $G$ be a complex reductive linear algebraic group of arbitrary
Lie type. Then $P_{u, w} = 1$ if and only if $P_{w_0 u^{-1}w_0, w_0
w^{-1} w_0} = 1$.
\end{lemma}

\begin{proof} 
We use a characterization of the condition that the Kazhdan-Lusztig
polynomials is equal to $1$, proved in various generality by Deodhar,
Carrell and Peterson; see
\cite[Theorem~6.2.10]{billey2000singular}. Let $\calR$ be the set of
(not necessarily simple) reflections in $W$. Then $P_{u,w} = 1$ if and
only if
\[ 
\# \{ r \in \calR: y < ry \le w \} = \ell(w) - \ell(y) , \quad \forall
u \le y \le w \/.
\]
It is well known that taking inverses, and conjugating by $w_0$ are
bijections of $W$ which preserve both the length and the Bruhat order
of elements. Thus $y < w$ if and only if $w_0 y^{-1} w_0 < w_0 w^{-1}
w_0$ and $\ell(w) - \ell(y) = \ell(w_0 w^{-1} w_0) - \ell(w_0 y^{-1}
w_0)$. This finishes the proof.
\end{proof}

We note that in general rational smoothness is different from
smoothness, therefore the statement from Corollary~\ref{cor:BNconj}
does not generalize to non-simply laced case. The statement of
\cite[Conjecture~1.2]{BN11} is slightly different from the final
version stated in Corollary~\ref{cor:BNconj} and in
\cite{bump2017casselman}. The initial statement was analyzed by Lee,
Lenart and Liu in \cite{lee.lenart.liu:whittaker}, and they found that
under certain conditions on the reduced words of $w$ and $z$ the
factorization holds, but in general there are counterexamples. We
refer to \cite{naruse2014schubert,nakasuji2015yang, bump2017casselman}
for work closely related to \cite{BN11}.

We now return to the proof of Theorem~\ref{thm:geomrefinedconj}. The
key part is the following result, which may be of interest in its own
right.

\begin{prop}\label{prop:newformula} 
(a) For every $w\geq u\in W$, the coefficient $m_{u,w}$ equals 
\[ 
m_{u,w} = \left(\frac{\MC_{y}(Y(u))|_w}{\MC_{y}(Y(w)^\circ)|_w}
\right)^\vee \/,
\] 
and it is an element in $K_T(\pt)_{\loc}[y^{- 1}]$.

\noindent (b) Assume that $Y(u)$ is smooth at $e_w$. Then 
\[ 
m_{u,w} = \frac{ \lambda_{y^{-1}} ((N_{Y(w)} Y(u))_w)}
{\lambda_{-1} ({(N_{Y(w)} Y(u))_w})} \/.
\]
In particular, we obtain a geometric analogue of the
Langlands-Gindikin-Karpelevich formula~\cite{langlands}:
\[ 
m_{1,w} = \prod_{\alpha <0, w^{-1} (\alpha)>0}
\frac{1 +y^{-1} e^{\alpha}}{1 - e^{\alpha}} \/. 
\]
\end{prop}

\begin{proof} 
Localizing both sides of~\eqref{E:MCdualexp} at the 
fixed point~$e_w$ gives $\MC_y^\vee(Y(u))|_w = m_{u,w}\cdot
\MC_y^\vee(Y(w)^\circ)|_w$. Therefore,
using Theorem~\ref{thm:dual1} and
Definition~\ref{def:dualmotvariety}, we obtain:
\[ 
m_{u,w} =\frac{\MC_y^\vee(Y(u))|_w}{\MC_y^\vee(Y(w)^\circ)|_w} =
\frac{\calD(\MC_y(Y(u))|_w)}{\calD(\MC_y(Y(w)^\circ)|_w)} =
\left(\frac{\MC_y(Y(u))|_w}{\MC_y(Y(w)^\circ)|_w} \right)^\vee \/.
\] 
This proves the first part of (a). The claim that the element $m_{u,w}$ 
is in $K_T(\pt)_{\loc}[y^{- 1}]$ follows from Theorem \ref{thm:mcdiv}
and Proposition \ref{prop:charmotoppo}(b), which show that 
$\prod_{\alpha > 0, w(\alpha)>0} (1+ ye^{w\alpha})$ divides $\MC_{y}(Y(u))|_w$.

For part (b), we use part (a)
and Theorem~\ref{thm:MCsmooth}
to obtain
\[ 
\begin{split} 
m_{u,w} & = \left(\frac{\MC_y(Y(u))|_w}{\MC_y(Y(w)^\circ)|_w}
\right)^\vee \\ 
& = \left(\frac{\lambda_y(T^*_w Y(u)) \cdot
  \lambda_{-1}((N_{Y(u)}^\vee (G/B))_w)}{\lambda_y(T^*_w Y(w)) \cdot
  \lambda_{-1}((N_{Y(w)}^\vee (G/B))_w)}\right)^\vee \\ 
& = \frac{\lambda_{y^{-1}}(T_w Y(u)) \cdot
  \lambda_{-1}((N_{Y(u)} (G/B))_w)}{\lambda_{y^{-1}}(T_w Y(w)) \cdot
  \lambda_{-1}((N_{Y(w)} (G/B))_w)} \\ 
& = \frac{ \lambda_{y^{-1}} ((N_{Y(w)} Y(u))_w)}{\lambda_{-1}
  ((N_{Y(w)} Y(u))_w)} \/.
\end{split} 
\] 
The last equality follows from the multiplicativity of the $\lambda_y$
class, and the short exact sequences
\[
0 \to T_w Y(w) \to T_w Y(u) \to (N_{Y(w)}Y(u))_w \to 0
\]
and 
\[
0 \to N_{w,Y(u)} (G/B) \to N_{w,Y(w)} (G/B) \to (N_{Y(w)}Y(u))_w 
\to 0 \:.
\]
The case when $u=1$ follows from the description of the weights from
Theorem~\ref{thm:MCsmooth}.
\end{proof} 

\begin{proof}[Proof of Theorem~\ref{thm:geomrefinedconj}] 
If $Y(u)$ is smooth at $e_w$ the claim follows from
Proposition~\ref{prop:newformula}(b), using the description of
appropriate weights from Remark~\ref{rmk:weights}.
Conversely, assume that
\[
m_{u,w}=\prod_{u\leq s_\alpha w<w}\frac{1+y^{-1}e^\alpha}
{1-e^\alpha}\/.
\] 
Part (a) of Proposition~\ref{prop:newformula}, together with the
localization result from Proposition~\ref{prop:charmotoppo} imply that
\[
\MC_{y}(Y(u))|_w=\prod_{\alpha>0,ws_\alpha\geq
  u}(1+ye^{w\alpha})\prod_{\alpha>0, u\nleq ws_\alpha}(1-e^{w\alpha}).
\] 
Therefore, $Y(u)$ is smooth at $e_w$ by
Theorem~\ref{thm:MCsmooth}.
\end{proof}

\section{Motivic Chern classes and the principal series representation}\label{sec:padic} 
The goal of this section is to establish an isomorphism of Hecke
modules between the Iwahori invariants of the unramified principal
series representations of a group over a non archimedean local field
and the (localized) equivariant $\Kt$-theory of the flag variety for
the complex Langlands dual group; see Theorem~\ref{thm:padic}. A
similar relation was established recently in~\cite{su2017k}, using the
equivariant $\Kt$-theory of the cotangent bundle and the stable
basis. The advantage of using motivic Chern classes is that their
functoriality properties will help get additional properties of this
correspondence. For instance, we use functoriality to relate
localization coefficients of the motivic Chern classes to coefficients
in the transition matrix between the standard basis and {\em
  Casselman's basis} (defined below).~This will be applied to solve
some conjectures of Bump, Nakasuji and Naruse about the coefficients
in the transition matrix between the standard and the Casselman's
basis.

\subsection{Iwahori invariants of the principal series representation} 
We recall below the definition and properties of the two bases in the
Iwahori invariants of the principal series representation. The
literature in this subject uses several normalization conventions. We
will be consistent with the conventions used in the paper of
Reeder~\cite{reeder1992certain} and in~\cite{su2017k}, because they
fit with our previous geometric calculations in this paper; these
conventions differ from those in~\cite{BN11,bump2017casselman}
or~\cite{brubaker.bump.licata}, and when necessary we will explain the
differences. Let $\ChevG$ be a split, reductive, Chevalley group
defined over $\bbZ$; see e.g.,~\cite{steinberg:chevalley}.  Let $\ChevB=
\ChevT\ChevN \leq \ChevG$ be a standard Borel subgroup containing a maximal
torus $\ChevT$ and its unipotent radical $\ChevN$. Let $W:=
N_\ChevG(\ChevT)/\ChevT$ be the Weyl group.  We will also consider the
Langlands dual $\LangG$ of $\ChevG$; by definition this will be a {\em
complex} reductive linear algebraic group, of type dual to the Lie
type of $\ChevG$. See e.g.,~\cite[\S2.1]
{borel:automorphic}\begin{footnote}{In the previous section we
used flag varieties associated to complex semisimple Lie groups.
If $G$ is any reductive group with radical $\Rad(G)$, then $G/\Rad(G)$ is semisimple, 
and the flag varieties for $G$ and $G/\Rad(G)$ are the same; 
see e.g.,~\cite[Corollary 8.1.6]{springerlinear}. We will tacitly use these facts in
this section.}\end{footnote}.  
Let $F$ be a non archimedean local field, with ring of integers $\calO$,
uniformizer $\varpi\in \calO$, and residue field $\bbF_{q'}$. Examples
are finite extensions of the field of $p$-adic numbers, or of the
field of Laurent series over $\bbF_p$. Since $\ChevG$ is defined over
$\bbZ$, we may consider $\ChevG(F)$, the group of the $F$-points of
$\ChevG$, with maximal torus $\ChevT(F)$ and Borel subgroup
$\ChevB(F)=\ChevT(F) \ChevN(F)$. Let $I$ be an Iwahori 
subgroup, i.e., the
inverse image of $\ChevB(\bbF_{q'})$ under the evaluation map
$\ChevG(\calO)\rightarrow \ChevG(\bbF_{q'})$.  To simplify formulas,
we let $\alpha, \beta$ denote the coroots of $\ChevG$. Let $R^+$ and
$R^{+\vee}$ denote the positive roots and coroots, respectively.

Let $\bbH=\bbC_c[I\backslash \ChevG(F)/I]$ be the Iwahori Hecke
algebra, consisting of compactly supported functions on $\ChevG(F)$
which are bi-invariant under $I$. As a vector space,
$\bbH=\Theta\otimes_\bbC H_W(q')$, where $\Theta$ is a commutative
subalgebra isomorphic to the coordinate ring $\bbC[\LangT]$ of the
complex dual torus $\LangT=\bbC^*\otimes X^*(\ChevT)$, and where
$H_W(q')$ is the finite Hecke (sub)algebra with parameter $q'$
associated to the (finite) Weyl group $W$. The finite Hecke algebra
$H_W(q')$ is also a subalgebra of $\bbH$, and it is generated by
elements $T_w$ ($w\in W$) such that the following relations hold: $T_u
T_v = T_{uv}$ if $\ell(uv) = \ell(u) + \ell(v)$, and
$(T_{s_i}+1)(T_{s_i} -q') = 0$ for a simple reflection $s_i$ in $W$.
 
For every character $\tau$ of $\ChevT$, and $\alpha$ a coroot define
$e^\alpha$ by $e^{\alpha}(\tau)=\tau(h_{\alpha}(\varpi))$, where
$h_{\alpha}:F^\times\rightarrow \ChevT(F)$ is the one parameter
subgroup. There is a pairing
\[
\langle \cdot , \cdot \rangle: \ChevT(F)/\ChevT({\calO})\times 
\LangT\rightarrow\bbC^*
\] 
given by $\langle a,z\otimes \lambda
\rangle=z^{\val(\lambda(a))}$.  This induces an isomorphism
between $\ChevT(F)/\ChevT({\calO})$ and the group $X^*(\LangT)$ of
rational characters of $\LangT$. It also induces an identification
between $\LangT$ and unramified characters of $\ChevT(F)$, i.e.,
characters which are trivial on $\ChevT(\calO)$.

Following Reeder \cite{reeder1992certain},
from now on we take $\tau$ to be an unramified
character of $\ChevT(F)$ such that $e^\alpha(\tau) \neq 1$ for all
coroots $\alpha$, and for which the stabilizer $W_\tau = 1$. The {\em
  principal series representation\/} is the induced representation
$I(\tau):=\Ind_{\ChevB(F)}^{\ChevG(F)} (\tau)$. As a $\bbC$-vector
space, $I(\tau)$ consists of locally constant functions $f$ on
$\ChevG(F)$ such that $f(bg)=\tau(b)\delta^{\frac{1}{2}}(b)f(g)$ for
every $b\in \ChevB(F)$, where
$\delta(b):=\prod_{\alpha>0}|\alpha^\vee(b)|_F$ is the modulus
function on the Borel subgroup. The Hecke algebra $\bbH$ acts through
convolution from the right on the Iwahori invariant
subspace~$I(\tau)^I$, so that the restriction of this action to
$H_W(q')$ is a regular representation. One can pass back and forth
between left and right $\bbH$-modules by using the standard
anti-involution $\iota$ on $\bbH$ given by $\iota(h)(x) =
h(x^{-1})$. If $T_w$ denote the standard generators of the Hecke
algebra $H_W(q')$, then $\iota(T_w)=T_{w^{-1}}$ and $\iota(q')=q'$,
see~\cite[Section~3.2]{haines2003iwahori}. This is of course
consistent with the left $\bbH$-action on $I(\tau)$ described by
Reeder in~\cite[p.~325]{reeder1992certain}.

We are interested in the Iwahori invariants $I(\tau)^I$ of the
principal series representation for an unramified character. One
reason to study the invariants is that as a $\ChevG(F)$-module, the
principal series representation $I(\tau)$ is generated by $I(\tau)^I$;
cf.~\cite[Proposition~2.7]{casselman1980unramified}. As a vector
space, $\dim_{\bbC} I(\tau)^I = |W|$, the order of the Weyl group
$W$. We will study the transition between two bases of
$I(\tau)^I$. From the decomposition $\ChevG(F) = \bigsqcup_{w \in W}
\ChevB(F) w I$ one obtains the basis of the {\em characteristic
  functions} on the orbits, denoted by $\{\varphi_w\mid w\in
W\}$\footnote{Our $\varphi_w$ is equal to $\phi_{w^{-1}}$
  in~\cite{BN11,bump2017casselman}.}. For $w\in W$, the element
$\varphi_w$ is characterized by the following two
conditions~\cite[p.~319]{reeder1992certain}:
\begin{enumerate}
\item $\varphi_w$ is supported on $\ChevB(F)wI$;
\item $\varphi_w(bwg)=\tau(b)\delta^{\frac{1}{2}}(b)$ for every $b\in
  \ChevB(F)$ and $g\in I$.
\end{enumerate}
The (left) action of $\bbH$ on $I(\tau)^I$, denoted by $\pi$, was
calculated e.g.,~by Casselman
in~\cite[Theorem~3.4]{casselman1980unramified}. With the conventions
from Reeder~\cite[p.~325]{reeder1992certain}, for every simple 
coroot~$\alpha_i$:
\begin{equation}\label{equ:firstaction}
\pi(T_{s_i})(\varphi_w)=
\left\{ 
\begin{array}{cc}
q'\varphi_{ws_i}+(q'-1)\varphi_w & \text{ if } ws_{i}<w; \\
\varphi_{ws_i}& \text{ if } ws_{i}>w.
\end{array}\right.
\end{equation}

The second basis, called {\em Casselman's basis}, and denoted by
$\{f_w\mid w\in W\}$, was defined by
Casselman~\cite[\S3]{casselman1980unramified} by duality using certain
intertwiner operators. We recall the relevant definitions, following
again~\cite{reeder1992certain}. For every character $\tau$ and $x\in
W$, define $x\tau\in X^*(\ChevT)$ by the formula
$x\tau(a):=\tau(x^{-1}ax)$ for every $a\in \ChevT$.  Since $\tau$ is
unramified and it has trivial stabilizer under the Weyl group action,
the space $\Hom_{\ChevG(F)}(I(\tau), I(x^{-1}\tau))$ is known to be
one dimensional, spanned by an operator\footnote{The intertwiner
  $\calA_x$ is related to $M_x$
  from~\cite{casselman1980unramified,BN11} by the formula
  $\calA_x=M_{x^{-1}}$.} $\calA_x=\calA_x^\tau$ defined by
\[
\calA_x(\varphi)(g):=\int_{N_x}\varphi(\dot{x}ng)dn,
\]
where $\dot{x}$ is a representative of $x\in W$, $N_x=N(F)\cap
\dot{x}^{-1}N^{-}(F)\dot{x}$ where $N^{-}$ is the unipotent radical of
the opposite Borel subgroup $\ChevB^-$; the measure on $N_x$ is
normalized by the condition that $\vol(N_x\cap
\ChevG(\calO))=1$~\cite{reeder1992certain}. If $x, y\in W$ satisfy
$\ell(x)+\ell(y)=\ell(xy)$, then
$\calA^{x^{-1}\tau}_y\calA^\tau_x=\calA_{xy}^{\tau}$. Then there exist
unique functions $f_w \in I(\tau)^I$ such that
\begin{equation}
\label{E:deffw}\calA_x^\tau (f_w)(1)=\delta_{x,w}.
\end{equation} 
(Under our conventions $f_w$ equals the element denoted
$f_{w^{-1}}$ in~\cite{BN11}.) For the longest element $w_0$ in the
Weyl group, Casselman showed
in~\cite[Proposition~3.7]{casselman1980unramified} that
\[
\varphi_{w_0}=f_{w_0}.
\] 
Reeder~\cite{reeder1992certain} calculated the action of $\bbH$ on the
functions $f_w$: he showed in~\cite[Lemma~4.1]{reeder1992certain} that
the functions $f_w$ are $\Theta$-eigenvectors, and he calculated
in~\cite[Proposition~4.9]{reeder1992certain} the action of
$H_W(q')$. To describe the latter, let
\begin{equation}\label{equ:calpha}
c_\alpha=\frac{1-q'^{-1}e^\alpha(\tau)}{1-e^\alpha(\tau)}.
\end{equation}

For every simple coroot $\alpha_i$ and $w\in W$, write
\[
J_{i,w}=
\left\{ \begin{array}{cc}
c_{w(\alpha_i)}c_{-w(\alpha_i)}& \text{ if } ws_i>w;\\
1& \text{ if } ws_i<w.
\end{array}\right.  
\]
Then, we have 
\begin{equation}\label{equ:secondaction}
\pi(T_{s_i})(f_w)=q'(1-c_{w(\alpha_i)})f_w+q'J_{i,w}f_{ws_i}.
\end{equation}

\subsection{A conjecture of Bump, Nakasuji, and Naruse}
\label{sec:BNNconj} 
In this section we state a conjecture of Bump, Nakasuji, and Naruse
regarding a factorization of certain coefficients of the transition
matrix between the bases $\{ \varphi_w \}$ and $\{ f_w \}$. Let
\[
\phi_u:=\sum_{u\leq w}\varphi_w \in I(\tau)^I,
\] 
and consider the expansion in terms of the Casselman's basis: 
\[
\phi_u=\sum_{w} \tilde{m}_{u,w}f_w.
\]
Then by the definition of $f_w$,
$\tilde{m}_{u,w}=\calA_w(\phi_u)(1)$. It is also easy to see that
$\tilde{m}_{u,w}=0$ unless $u\leq w$, see~\cite[Theorem~3.5]{BN11}. We
shall see below that $\tilde{m}_{u,w}$ equals the evaluation at~$\tau$
of the coefficient $m_{u,w}$ from \eqref{E:MCdualexp}, defined
for the Langlands dual flag variety. For every $u\leq w\in W$, recall
the definition $S(u,w):=\{\beta \in R^{+, \vee}\,|\,u\leq s_\beta w <
w\}$ (cf.~\eqref{E:defS}).  Recall that $\LangG$ is the complex Langlands 
dual group, with the corresponding Borel subgroup $\LangB$ and the 
maximal torus $\LangT$. The goal is to prove the following statement.

\begin{theorem}[Bump-Nakasuji-Naruse Conjecture]\label{thm:refinedconj}
For every $u\leq w\in W$, 
\begin{equation}\label{equ:mwz}
\tilde{m}_{u,w}=\prod_{\alpha\in
  S(u,w)}\frac{1-q'^{-1}e^\alpha(\tau)}{1-e^\alpha(\tau)},
\end{equation}
if and only if the opposite Schubert variety
$Y(u):=\overline{\LangB u \LangB/\LangB}$ in the (dual, complex) flag
manifold $\LangG/\LangB$ is smooth at the torus fixed point $e_{w}$.
\end{theorem} 

This is the representation-theoretic counterpart of
Theorem~\ref{thm:geomrefinedconj}; its proof will be given in the next
subsection.

We provide further historical context.
Casselman~\cite{casselman1980unramified} asked for an expression of
the basis $f_w$ as a linear combination of the standard basis
$\varphi_w$. Bump and Nakasuji found that the basis $\phi_w$ is better
behaved for this question. Of course the original Casselman's basis
can be obtained from the M{\"o}bius inversion
\[
\varphi_u=\sum_{w\geq u}(-1)^{\ell(u)-\ell(w)}\phi_w.
\] 

The case $u=1$ of the Bump-Nakasuji-Naruse conjecture is well
known. In this case $\phi_1$ is the spherical vector in $I(\tau)$,
i.e., the vector fixed by the maximal compact subgroup $G(\calO)$, and
\begin{equation}\label{E:LGK}
\calA_w(\phi_1)(1)=\tilde{m}_{1,w} = \prod_{\alpha\in
  S(1,w)}\frac{1-q'^{-1}e^\alpha(\tau)}{1-e^\alpha(\tau)}.
\end{equation}
This is the {\em Gindikin--Karpelevich formula}, which in the
non-archimedean setting was actually proved by Langlands
\cite{langlands} after Gindikin and Karpelevich proved a similar
statement for real groups. Casselman obtained another proof using his
basis $f_w$, and this plays a crucial role in his computation of the
Macdonald formula and the spherical Whittaker functions, see
\cite{casselman1980unramified, casselman1980unramifiedII}. See also
\cite{su2017k} for an approach using the stable basis and the
equivariant $\Kt$-theory of the cotangent bundle $T^*(\LangG/\LangB)$.
We will recover~\eqref{E:LGK} below, as a consequence of
Theorem~\ref{thm:padic}.  Other special cases of the conjecture follow
from work of Reeder~\cite{reeder1993p}.

\subsection{Casselman's problem and motivic Chern classes}
In this section, we construct the promised isomorphism between the
$H_W(q')$-module of the Iwahori invariants of the principal series
representation of $\ChevG$ and the equivariant $\Kt$ group of the flag
variety for the dual group $\LangG$, regarded as an $H_W(-y)$-module
via the action of the operators $\opT^\vee_i$.  This construction,
together with the cohomological properties of the motivic Chern
classes from~\S\ref{s:smoothness}, will be used to
prove~Theorem~\ref{thm:refinedconj}.

For now we assume that the unramified character $\tau$ is in the open
set in $\LangT$ such that $1-q' e^{\alpha}(\tau)\neq 0$, for every
(positive or negative) coroot $\alpha$. Regard the representation ring
$K_{\LangT}(\pt)$ as a subring of $\bbC[\LangT]$ and let $\bbC_\tau$
denote the one dimensional $K_{\LangT}(\pt)$-module induced by
evaluation at $\tau$. Recall that the operators $\opT^\vee_i$ from
Definition~\ref{def:hecke} satisfy
\[
(\opT^\vee_i+1)(\opT^\vee_i+y)=0,
\]
and the braid relations (Proposition~\ref{prop:hecke-relations}). 
Hence, they induce an action of the Hecke
algebra $H_W(-y)$ with parameter $-y$ on the $\Kt$-theory ring
$K_{\LangT}(\LangG/\LangB)[y,y^{-1}]$ by sending $T_w$ 
to~$\opT^\vee_w$ (here the parameter $-y$
corresponds to the parameter $q$ in~\cite{lusztig:eqK}). We use the
symbol $\pi$ to denote this action.

As in \eqref{E:defbw}, define the element
$\tilde{b}_w \in K_{\LangT}(\LangG/\LangB)_{\loc}[y^{-1}]$ by the
formula
\[
\tilde{b}_w:=(-1)^{\dim \LangG/\LangB - \ell(w)}
\prod_{\alpha>0, w\alpha>0}\frac{y^{-1} +e^{-w\alpha}}
{1-e^{w\alpha}}\iota_{w} \otimes 1  \/. 
\] 
Equivalently, $\tilde{b}_w = b_w \otimes 1$, with $b_w$ from equation
\eqref{E:defbw}.

We now state the main comparison theorem.\begin{footnote}
{We are thankful to a referee for suggesting the current formulation.}
\end{footnote} 
 \begin{theorem}\label{thm:padic} 
There exists a unique isomorphism of left $H_W(q')$-modules (assuming the identification $q'=-y$)
\[
\Psi:K_{\LangT}(\LangG/\LangB)[y,y^{-1}]
\otimes_{K_{\LangT}(\pt)[y,y^{-1}]}\bbC_\tau\xrightarrow{\sim} 
I(\tau)^I \/, 
\] 
such that\\
(a) $\Psi(\MC_y^\vee(Y(w)^\circ) \otimes 1)= \varphi_w$ and\\
(b) $\Psi (\tilde{b}_w) = f_w$. 
\end{theorem}
\begin{remark} 
There is an analogue of this theorem for the equivariant $\Kt$-theory
of $T^*(G/B)$
proved in~\cite{su2017k}. This is also studied by
Lusztig and Braverman--Kazhdan in~\cite{lusztig1998bases,
braverman1999schwartz} from different points of view.
\qede\end{remark}
\begin{proof} 
The uniqueness is obvious. Next we define the map $\Psi$ 
by property~(a), and prove the remaining claims. The fact that $\Psi$ is a 
map of $H_W(q')$-modules follows from
comparing Proposition~\ref{prop:Tidualaction}(a) to equation
\eqref{equ:firstaction}; these describe the Hecke actions on the basis
of dual motivic Chern classes $\MC_y^\vee(Y(w)^\circ)$ and on the basis
of characteristic functions $\varphi_w$.\\
To prove property (b), we argue by descending induction on
$\ell(w)$. Recall that $f_{w_0} = \varphi_{w_0}$ and $b_{w_0} =
\iota_{w_0} = \MC_y^\vee(Y(w_0))$ (from Definition~\ref{def:firstdual}); 
therefore
$\Psi(\widetilde{b}_{w_0})= f_{w_0}$. Now take any $w<w_0$ and assume
that $\Psi(\widetilde{b}_z) = f_z$ for all $\tilde{b}_z$ with
$\ell(z)>\ell(w)$. Pick a simple root $\alpha_i$ such that
$ws_{i}>w$. Then by induction, $\Psi (\tilde{b}_{ws_i}) = f_{ws_i}$.
Since $\Psi$ is a homomorphism of Hecke modules,
\[
\Psi(\opT^\vee_i (\widetilde{b}_{ws_i})) = 
\pi(T_i)(\Psi(\widetilde{b}_{ws_i})) = 
\pi(T_i)(f_{ws_i}).
\] 
On the one hand, Lemma~\ref{lem:actiononfixedpoint} gives 
\[
\Psi(\opT^\vee_i (\widetilde{b}_{ws_i}))=
\Psi\left(\frac{1+y}{e^{w\alpha_i}(\tau)-1}
\tilde{b}_{ws_i}-y\tilde{b}_w\right)
=\frac{1-q'}{e^{w\alpha_i}(\tau)-1}f_{ws_i}
-y\Psi(\tilde{b}_w).
\] 
Here we have used that $\{ \alpha > 0 : w(\alpha) > 0 \} = \{ \alpha >
0 : ws_i(\alpha) > 0 \} \cup \{ \alpha_i \}$. On the other hand,
equation~\eqref{equ:secondaction} gives
\[
\pi(T_i)(f_{ws_i})=\frac{1-q'}{e^{w\alpha_i}(\tau)-1}
f_{ws_i}-yf_w.
\]
Therefore, $\Psi(\widetilde{b}_{w}) = f_{w}$. By induction, this
finishes the proof of property (b).
\end{proof}

\begin{corol}\label{cor:m}
The coefficients $\tilde{m}_{u,w}$ are represented by the
meromorphic functions $m_{u,w}$ on $\LangT$, defined in equation
\eqref{E:MCdualexp}, for the Langlands dual complex flag variety
$\LangG/\LangB$. More precisely, let $\tau \in \LangT$ be any regular
unramified character, i.e., with trivial stabilizer $W_\tau$.
Then
\[ 
\tilde{m}_{u,w} = m_{u,w}(\tau) \/. 
\]
\end{corol}
\begin{proof}
Observe that Theorem~\ref{thm:padic}, together with the
definitions of $\tilde{m}_{u,w}$ and $m_{u,w}$ imply the equality
$\tilde{m}_{u,w} = m_{u,w}(\tau)$ for all regular unramified
characters $\tau$ satisfying $1-q' e^{\alpha}(\tau)\neq 0$, for every
coroot $\alpha$. However, it is known that the intertwiners $\calA_x$
depend holomorphically on regular characters $\tau \in \LangT$; see
e.g.,~\cite[\S3]{casselman1980unramified} 
or~\cite[\S6.4]{casselman:intro}. Then one can extend meromorphically 
the equality $\tilde{m}_{u,w} = m_{u,w}(\tau)$ to any regular unramified
character $\tau$.
\end{proof}

Combining Corollary~\ref{cor:m} with Proposition
\ref{prop:newformula} above gives a formula for $\tilde{m}_{u,w}$ in
terms of localizations of motivic Chern classes, and in particular it
recovers the Langlands-Gindikin-Karpelevich formula from
\eqref{E:LGK}. 
Also, Theorem~\ref{thm:refinedconj} follows now from 
Theorem~\ref{thm:padic}(c) and the main theorem 
from~\S\ref{s:smoothness}.
\begin{proof}[Proof of Theorem~\ref{thm:refinedconj}] 
This follows from Theorem~\ref{thm:geomrefinedconj} above together
with the equality $\tilde{m}_{u,w} = m_{u,w}(\tau)$ for all regular
unramified characters $\tau$.
\end{proof}

\subsection{Analytic properties of transition coefficients} 
In this section we prove a conjecture of Bump and
Nakasuji~\cite[Conjecture~1]{bump2017casselman} about analytic
properties for the transition coefficients $\tilde{m}_{u,w}$ and the
set of coefficients $\tilde{r}_{u,w}$ defined as follows
(cf.~\cite[Theorem~3]{bump2017casselman}).  If $f=f(q')$ is a
function, let $\bar{f}(q'):=f(q'^{-1})$. 
Define
\begin{equation}\label{E:defr} 
\tilde{r}_{u,w} := \sum_{u \le x \le w} (-1)^{\ell(x) - \ell(u)}
\overline{\tilde{m}}_{x,w} \/. 
\end{equation} 
Since we are interested only in analytic properties, by the comparison
Theorem~\ref{thm:padic}(c) we can replace the coefficients
$\tilde{m}_{u,w}$ by the `geometric' ones $m_{u,w}$. We accordingly
let $r_{u,w}:= \sum_{u \le x \le w} (-1)^{\ell(x) - \ell(u)}
\overline{{m}}_{x,w}$ be the corresponding coefficients, where
$\bar{f}(y):=f(y^{-1})$.  Therefore, $\tilde r_{u,w} =
r_{u,w}(\tau)|_{y=-q'}$ for any regular unramified character~$\tau$.
With this notation, we can prove the following statement,
cf.~\cite[Conjecture~1]{bump2017casselman}.

\begin{theorem}\label{thm:hol} 
Let $u \le w$ be two Weyl group elements. Then the functions
\[ 
\prod_{\alpha \in S(u,w)}(1-e^{\alpha})r_{u,w} \quad , 
\prod_{\alpha \in S(u,w)}(1-e^{\alpha})m_{u,w} 
\]
are holomorphic on the dual torus $\LangT$.
\end{theorem}

\begin{proof} 
As observed by Bump and Nakasuji in {\it loc.~cit.}, the conjecture
for $r_{u,w}$ implies the conjecture for $m_{u,w}$. Further, using the
formula for $m_{u,w}$ from Proposition~\ref{prop:newformula}, 
\[ 
\begin{split}  
(\overline{r}_{u,w})^\vee & = \sum_{u \le x \le w} (-1)^{\ell(x) -
    \ell(u)} (m_{x,w})^\vee = \frac{1}{\MC_y(Y(w)^\circ)|_w}
\sum_{u \le x \le w} (-1)^{\ell(x) - \ell(u)} \MC_y (Y(x))|_w \\ 
& = \frac{1}{\MC_y(Y(w)^\circ)|_w} \sum_{u \le x}
  (-1)^{\ell(x) - \ell(u)} \MC_y (Y(x))|_w =
  \frac{\MC_y(Y(u)^\circ)|_w}{\MC_y(Y(w)^\circ)|_w}
  \/; 
\end{split} 
\] 
here the third equality holds because $\MC_y(Y(x))|_w = 0$ for $x \nleq
w$, as $e_w \notin Y(x)$, and the last equality follows by M{\"o}bius
inversion on the Bruhat poset $W$. It follows that
\[ 
\Bigl( \prod_{\alpha \in
  S(u,w)}(1-e^{\alpha})r_{u,w}\Bigr)^\vee = \prod_{\alpha \in
  S(u,w)}(1-e^{-\alpha})
\frac{\MC_{y^{-1}}(Y(u)^\circ)|_w}{\MC_{y^{-1}} (Y(w)^\circ)|_w}
\/, 
\] 
therefore it suffices to show that the right-hand side is holomorphic
in $y$.  Using Proposition~\ref{prop:charmotoppo}, the definition of
$S(u,w)$ \eqref{E:defS}, and the description of $\MC_y(Y(w)^\circ)|_w$
from Theorem~\ref{thm:MCsmooth} (and its proof), we obtain
\[ 
\begin{split} 
\prod_{\alpha\in S(u,w)}(1-e^{-\alpha})
\frac{\MC_{y^{-1}}(Y(u)^\circ)|_w} {\MC_{y^{-1}} (Y(w)^\circ)|_w} &
= \frac{\MC_{y^{-1}}(Y(u)^\circ)|_w} {\lambda_{y^{-1}} (T^*_w
  Y(w))} \times \frac{\prod_{\alpha \in S(u,w)}
  (1-e^{-\alpha})}{\prod_{\alpha > 0, ws_\alpha < w} (1 -
  e^{w\alpha})} \\ 
& = \frac{\MC_{y^{-1}}(Y(u)^\circ)|_w} {\lambda_{y^{-1}} (T^*_w
  Y(w))} \times \frac{\prod_{\alpha > 0, u \leq w s_\alpha <
    w}(1-e^{w\alpha})}{\prod_{\alpha > 0, ws_\alpha < w} (1 -
  e^{w\alpha})} \\ 
& = \frac{\MC_{y^{-1}}(Y(u)^\circ)|_w}{\lambda_{y^{-1}} (T^*_w
  Y(w)) \cdot \prod_{\alpha > 0, u \nleq ws_\alpha < w} (1 -
  e^{w\alpha})} \/. 
\end{split}
\] 
The last expression is a holomorphic function by
Theorem~\ref{thm:mcdiv} and we are done.
\end{proof}

We also record the following result, obtained in the proof of
Theorem~\ref{thm:hol}.
\begin{prop}
The coefficients $r_{u,w}$ are obtained from the expansion: 
\[ 
\MC_{y}^\vee(Y(u)^\circ) = \sum \overline{r}_{u,w} b_w  \/
\] 
or, equivalently 
\[ 
(\overline{r}_{u,w})^\vee =  \frac{\MC_{y}(Y(u)^\circ)|_w}
{\MC_{y}(Y(w)^\circ)|_w} \: 
\in \Frac\left(K_\LangT(\pt)[y^{\pm 1}] \right)  \quad \/. 
\] 
\end{prop}

\begin{remark} 
The coefficients $r_{u,w}$ satisfy other remarkable properties, such
as a certain duality and orthogonality; see~\cite{nakasuji2015yang,
  bump2017casselman}. We will study these properties using motivic
Chern classes in an upcoming note.  
\qede\end{remark}

\section{Appendix: proof of Lemma~\ref{lemma:autos}}
Recall Lemma~\ref{lemma:autos}. 

\begin{lemma}\label{lemma:autosinapp} 
(a) Let $u,w \in W$. Under the left Weyl group multiplication, 
\[ 
w. \stab_{\fC, T^{1/2}, \calL}(u) 
= \stab_{w \fC, w T^{1/2}, w.\calL}(wu) \/. 
\] 
In particular, if both the polarization $T^{1/2}$ and the line bundle
$\calL$ are $G$-equivariant, then
\[ 
w. \stab_{\fC, T^{1/2}, \calL}(u) 
= \stab_{w \fC, T^{1/2}, \calL}(wu) \/. 
\]

(b) The duality automorphism acts by sending $q \mapsto q^{-1}$ and 
\begin{equation}\label{equ:appdual}
(\stab_{\fC,T^{\frac{1}{2}}, \calL}(w))^\vee
=q^{-\frac{\dim G/B}{2}}\stab_{\fC,T^{\frac{1}{2}}_{\opp}, \calL^{-1}}(w),
\end{equation}
where $T^{\frac{1}{2}}_{\opp}:=q^{-1}(T^{\frac{1}{2}})^\vee$
is the opposite polarization; see~\cite[Equation~(15)]
{okounkov2016quantum}, i.e., this duality changes the polarization
and slope parameters to the opposite ones, while keeping the chamber
parameter invariant.

(c) Fix integral weights $\lambda, \mu \in X^*(T)$ and the equivariant
line bundles $\calL_\lambda= G \times^B \bbC_\lambda$ and
$\calL_\mu$. Let $a \in \bbQ$ be a rational number. Then
\[ 
\stab_{\fC, T^{1/2},a \calL_{\lambda} \otimes
  \calL_\mu}(w) = e^{- w(\mu)} \calL_\mu \otimes
\stab_{\fC, T^{1/2},a \calL_{\lambda}}(w) \/, 
\] 
as elements in $K_{T\times\bbC^*}(T^*(G/B))$.
\end{lemma}

\begin{proof} 
Part (a) can be proved directly by checking that the left-hand side 
satisfies the defining properties of the stable basis on the right-hand side.
Part (b) is~\cite[Equation (15)]{okounkov2016quantum} with
$\hbar^{-\frac{\dim X}{2}}$ replaced by $q^{-\frac{\dim G/B}{2}}$ \begin{footnote}
{In~\cite{okounkov2016quantum} the variety $X$ is the symplectic
  resolution, and it corresponds to our $T^*(G/B)$, so
  $\hbar^{-\frac{\dim X}{2}}$ should be $q^{-\dim G/B}$; the missing
  factor of $\frac{1}{2}$ is a typo.}\end{footnote}. Since this equation is not 
  proved in {\it loc.~cit.}, we include a proof in the case when the fixed point set 
  $X^T$ is finite
(e.g., $X=T^*(G/B)$).\footnote{Our $T$ is
  denoted by $A$ in {\it loc.~cit.} and this is the torus
  preserving the symplectic form of $X$.} By the uniqueness of stable
envelopes, we need to show $q^{\frac{\dim X}{4}}(\stab_{\fC,
  T^{\frac{1}{2}},\calL})^\vee$ satisfies the defining properties of
$\stab_{\fC, T_{\opp}^{\frac{1}{2}},\calL^{-1}}$. The support
condition is obvious, and the degree condition follows because
$\deg_T$ remains unchanged after multiplication by powers of $q^{\pm
  \frac{1}{2}}$. We turn to the normalization condition. Denote by $F$
a component of $X^T$ (a point, in our case), and use the notation
$N_+$, $N_-$, $N^{\frac{1}{2}}$, etc., for the appropriate normal
subspaces to $F$, as in the paragraphs preceding
Theorem~\ref{thm:geostable}. Since
$N_--T^{\frac{1}{2}}=q^{-1}(T^{\frac{1}{2}}_{>0})^\vee
-T^{\frac{1}{2}}_{>0}$ (see~\cite[p.~13]{okounkov2016quantum}), the
normalization is (see~\cite[Equation~(10)]{okounkov2016quantum})
\[
\stab_{\fC, T^{\frac{1}{2}},\calL}|_{F}
=(-1)^{\rk T^{\frac{1}{2}}_{>0}}\left(\frac{\det N_-}
{\det T^{\frac{1}{2}}}\right)^{\frac{1}{2}}
\calO_{\Attr}|_{F}=(-1)^{\rk T^{\frac{1}{2}}_{>0}}
q^{-\frac{\rk T^{\frac{1}{2}}_{>0}}{2}}
(\det T^{\frac{1}{2}}_{>0})^\vee\calO_{N_+}|_{F}.
\] 
The last equality holds because the normal bundle of $\Attr$
at $F$ is spanned by the {\em non-attracting} weights at $F$; this is
the same as the normal bundle of $N_+$ inside $N$, therefore
$\calO_{\Attr}|_{F} = \calO_{N_+}|_{F} = \lambda_{-1}^{T
  \times \bbC^*}(N_-^\vee)$.

Let $\{ \gamma_j\}$ be the torus weights of $T^{\frac{1}{2}}_{>0}|_F$
and let $\{\beta_i \}$ be the torus weights of
$T^{\frac{1}{2}}_{<0}|_F$.  Since
$T(X)=T^{\frac{1}{2}}+q^{-1}(T^{\frac{1}{2}})^\vee$, the torus weights
of $N_-|_{F}$ are $\{ \beta_i \}$ and $\{ q^{-1}\gamma_j^{-1} \}$. We
abuse notation and write $\lambda$ for $e^\lambda \in R(T)$. Then,
\begin{align}
\stab_{\fC, T^{\frac{1}{2}},\calL}|_{F}&=(-1)^{\rk
  T^{\frac{1}{2}}_{>0}}q^{-\frac{\rk T^{\frac{1}{2}}_{>0}}{2}}\prod_j
\gamma_j^{-1}\prod_i(1-\beta_i^{-1})\prod_j(1-q\gamma_j)\nonumber\\ 
&=q^{\frac{\rk T^{\frac{1}{2}}_{>0}}{2}}\prod_i(1-\beta_i^{-1})
\prod_j(1-q^{-1}\gamma_j^{-1}).\label{equ:nor}
\end{align}
Since $T^{\frac{1}{2}}_{\opp}=q^{-1}(T^{\frac{1}{2}})^\vee$,
the torus weights of $T^{\frac{1}{2}}_{\opp,>0}|_F$ are 
$\{q^{-1}\beta_i^{-1} \}$ and the torus weights of
$T^{\frac{1}{2}}_{\opp,<0}|_F$ are $\{
q^{-1}\gamma_j^{-1}\}$. A similar calculation shows
\[
\stab_{\fC, T^{\frac{1}{2}}_{\opp},\calL^{-1}}|_{F}
=q^{\frac{\rk T^{\frac{1}{2}}_{<0}}{2}}\prod_j(1-q\gamma_j)
\prod_i(1-\beta_i).
\]
Taking the dual of \eqref{equ:nor}, we get 
\[
\left(\stab_{\fC, T^{\frac{1}{2}},\calL}\right)^\vee|_{F}
=q^{-\frac{\rk T^{\frac{1}{2}}_{>0}}{2}}\prod_i(1-\beta_i)
\prod_j(1-q\gamma_j).
\]
Therefore, 
\[
\left(\stab_{\fC, T^{\frac{1}{2}},\calL}\right)^\vee|_{F}
=q^{-\frac{\rk T^{\frac{1}{2}}}{2}}
\stab_{\fC, T^{\frac{1}{2}}_{\opp},\calL^{-1}}|_{F}
=q^{-\frac{\dim X}{4}}\stab_{\fC, T^{\frac{1}{2}}_{\opp},
\calL^{-1}}|_{F}.
\]
This proves the normalization condition, whence part (b). Part (c)
follows directly from the uniqueness of the stable envelope.
\end{proof}

\bibliographystyle{halpha}
\bibliography{AENSbiblio}

\end{document}